\documentclass{article}

\usepackage{amssymb,amsthm,stmaryrd,amsmath}
\usepackage{graphicx}
\usepackage{pstricks}

\DeclareFontFamily{T1}{pzc}{}
\DeclareFontShape{T1}{pzc}{m}{it}{<-> [1.2] pzcmi8t}{}
\DeclareMathAlphabet{\mathpzc}{T1}{pzc}{m}{it}
\newcommand{\dt}{\partial_t}
\newcommand{\dz}{\partial_z}
\newcommand{\dive}{\mbox{\textnormal{div }}}
\newcommand{\curl}{\mbox{\textnormal{curl }}}
\newcommand{\grad}{\nabla_{X,z}}
\newcommand{\dn}{\partial_n}
\newcommand{\nam}{\nabla^\mu}
\newcommand{\nampm}{\nabla^{\mu^\pm}}
\newcommand{\namm}{\nabla^{\mu^-}}
\newcommand{\namp}{\nabla^{\mu^+}}

\newcommand{\Pp}{{\mathfrak P}}

\newcommand{\eps}{\varepsilon}

\newcommand{\inteps}{{\vert_{z=\eps\zeta}}}
\newcommand{\interft}{{\vert_{\Gamma_t}}}

\newcommand{\interfpm}{{\vert_{\Gamma^\pm}}}
\newcommand{\interff}{{\vert_{z=0}}}
\newcommand{\K}{{\mathpzc K}}
\newcommand{\Sy}{{\mathpzc S}}
\newcommand{\Ka}{{\mathpzc K}_{(\alpha)}[\nabla\zeta]}
\newcommand{\Kae}{{\mathpzc K}_{(\alpha)}[\eps\sqrt{\mu}\nabla\zeta]}
\newcommand{\Kaeb}{{\mathpzc K}_{(\alpha)}}
\newcommand{\G}{{\mathpzc G}[\zeta]}
\newcommand{\Gmu}{{\mathpzc G}_\mu[\eps\zeta]}
\newcommand{\Ga}{{\mathpzc G}_{(\alpha)}}

\newcommand{\Gp}{{\mathpzc G}^+[\zeta]}
\newcommand{\Gpmu}{{\mathpzc G}^+_\mu[\eps\zeta]}
\newcommand{\Gpm}{{\mathpzc G}^\pm[\zeta]}
\newcommand{\Gpmb}{{{\mathpzc G}^\pm}}
\newcommand{\Gpmhpm}{{\mathpzc G}_\mu^\pm[\eps\zeta,\uH^\pm]}
\newcommand{\Gm}{{\mathpzc G}^-[\zeta]}
\newcommand{\Gmb}{{{\mathpzc G}^-}}
\newcommand{\Gb}{{{\mathpzc G}}}
\newcommand{\Gpb}{{{\mathpzc G}^+}}

\newcommand{\E}{{\mathpzc E}[\zeta]}
\newcommand{\Eb}{{\mathpzc E}}
\newcommand{\T}{{\mathpzc T}[U]}
\newcommand{\RT}{{\mathpzc{Ins}}[U]}

\newcommand{\Tb}{{\mathpzc T}}
\newcommand{\Upm}{{\mathpzc V}^\pm_\sslash[\eps\zeta]}

\newcommand{\Um}{{\mathpzc V}^-_\sslash[\eps\zeta]}

\newcommand{\cw}{{\mathpzc w}}
\newcommand{\wpm}{{\mathpzc w}^\pm[\eps\zeta]}
\newcommand{\uw}{\underline{w}}
\newcommand{\uwpm}{\underline{w}^\pm}
\newcommand{\uwp}{\underline{w}^+}
\newcommand{\uwm}{\underline{w}^-}
\newcommand{\uVpm}{\underline{V}^\pm}
\newcommand{\uVp}{\underline{V}^+}

\newcommand{\wpl}{{\mathpzc w}^+[\zeta]}
\newcommand{\wm}{{\mathpzc w}^-[\zeta]}
\newcommand{\urp}{\underline{\rho}^+}
\newcommand{\urm}{\underline{\rho}^-}
\newcommand{\urpm}{\underline{\rho}^\pm}
\newcommand{\psiap}{\psi^+_{(\alpha)}}

\newcommand{\psiapm}{\psi^\pm_{(\alpha)}}
\newcommand{\psia}{\psi_{(\alpha)}}
\newcommand{\psiaa}{\psi_{\av{\check{\alpha}}}}
\newcommand{\zetaaa}{\zeta_{\av{\check{\alpha}}}}
\newcommand{\Uaa}{U_{\av{\check{\alpha}}}}

\newcommand{\zetaa}{\zeta_{(\alpha)}}
\newcommand{\Ua}{U_{(\alpha)}}
\newcommand{\bU}{{\bf U}}
\newcommand{\bh}{{\bf h}}
\newcommand{\enz}{\abs{\zeta}_{<N+1/2>}}

\newcommand{\R}{{\mathbb R}}
\newcommand{\N}{{\mathbb N}}
\newcommand{\uH}{\underline{H}}
\newcommand{\Id}{\mbox{\textnormal{Id }}}
\newcommand{\cS}{{\mathcal S}}

\newcommand{\cE}{{\mathcal E}}
\newcommand{\cF}{{\mathcal F}}
\newcommand{\mfa}{{\mathfrak a}}
\newcommand{\mfh}{{\mathfrak h}}
\newcommand{\mfm}{{\mathfrak m}}
\newcommand{\mfe}{{\mathfrak e}}
\newcommand{\mfk}{{\mathpzc k}}

\newcommand{\bP}{{\bf P}}
\newcommand{\Op}{\mbox{\textnormal{Op}}}
\newcommand{\Bo}{{\mbox{\textnormal{Bo}}}}
\newcommand{\Hm}{\dot{H}_\mu}
\newcommand{\dsp}{\displaystyle}
\newcommand{\abs}[1]{\vert#1\vert}
\newcommand{\Abs}[1]{\Vert#1\Vert}
\newcommand{\babs}[1]{\big\vert#1\big\vert}

\newcommand{\jump}[1]{\llbracket#1\rrbracket}
\newcommand{\av}[1]{\langle#1\rangle}
\newcommand{\pb}[1]{\{#1\}}
\newcommand{\diag}{\mbox{\textnormal{diag}}}

\newtheorem{theorem}{Theorem}[section]
\newtheorem{lemma}[theorem]{Lemma}
\newtheorem{proposition}[theorem]{Proposition}

\newtheorem*{theoremp}{Theorem \ref{theomain}'}
\newtheorem{corollary}[theorem]{Corollary}
\theoremstyle{remark}
\newtheorem{notation}[theorem]{Notation}
\newtheorem{remark}[theorem]{Remark}
\newtheorem{example}[theorem]{Example}
\theoremstyle{definition}
\newtheorem{definition}[theorem]{Definition}

\numberwithin{equation}{section}

\title{A stability criterion for two-fluid interfaces and applications}

\date{}
\author{David Lannes\footnote{DMA, Ecole Normale Sup\'erieure et CNRS UMR 8553, 45 rue d'Ulm, 75005 Paris, France}}
\begin{document}

\maketitle

\begin{abstract}
We derive here a new stability criterion for two-fluid interfaces. This
criterion ensures the existence of 
``stable'' local solutions that do no break down too fast due to Kelvin-Helmholtz
instabilities. It can be seen both as a two-fluid generalization of the Rayleigh-Taylor criterion and as a nonlinear version of the Kelvin stability condition. We show that gravity
can control the inertial effects of the shear up to frequencies that are high enough for the
surface tension to play a relevant role. This explains why surface tension is a necessary condition for well-posedness while the (low frequency) main dynamics of interfacial waves is unaffected
by it. In order to derive a practical version of this criterion, we work with
a nondimensionalized version of the equations and allow for the possibility
of various asymptotic regimes, such as the shallow water limit. This limit
being singular, we have to derive a new symbolic analysis of the Dirichlet-Neumann operator 
that includes an infinitely smoothing ``tail'' accounting for the contribution of the bottom.  We then validate our criterion by comparison with experimental data in two important settings: air-water interfaces and internal waves. The good agreement we observe allows us to discuss the scenario of wave breaking and the behavior of water-brine interfaces. We also 
show how to rigorously justify
two-fluid asymptotic models used  for applications and how to
relate some of their properties to Kelvin-Helmholtz instabilities.
\end{abstract}

\section{Introduction}

\subsection{General setting and overview of the results}

We are interested here in the motion of the interface between two incompressible fluids 
of different densities $\rho^+>\rho^-$, with vorticity concentrated at the interface, and at 
rest at infinity. This problem
is sometimes called the Rayleigh-Taylor problem; the limit cases $\rho^+=\rho^-$ and $\rho^-=0$ are
 also known as the Kelvin-Helmholtz and water waves problems respectively
(see \cite{BardosLannes} for a recent review).

The mathematical analysis of the stability issue of such interfaces has raised a lot of interest. 
The particular case of the water waves problem is certainly the best understood. Following the pioneer works of \cite{Nalimov,Ovsjannikov,Yosihara,Craig,BeyerGunther}, S. Wu established the well posedness of the water waves equation for $1D$ and $2D$ surface waves \cite{WU,WUb} in infinite depth, and without surface tension, provided that
the \emph{Rayleigh-Taylor criterion} is satisfied,
$$
\mbox{(Rayleigh-Taylor)}\qquad -\dz P\,_{\vert_{\mbox{surface}}}>0,
$$
where $z$ is the vertical coordinate and $P$ the pressure (she also established that
this condition is satisfied as long as the wave does not self-intersect). It is instructive to remark that
the linearized version of this criterion (around the rest state) is simply $\rho^+ g>0$, 
where
$g$ is the (vertical) acceleration of gravity -- in other words, water must stands \emph{below} the interface.
Other approaches have also been developed, 
such as \cite{LannesJAMS} (finite depth, Eulerian framework), 
\cite{Lindblad,ShatahZeng,CoutandShkoller} (fluid droplet in zero gravity, Lagrangian framework, irrotationality assumption removed), \cite{CCG1} (weak solutions in the whole space with density discontinuity at the surface), etc. 
More recently, various works gave some additional information on the solution. 
It was shown in \cite{AL,Iguchi} that the existence time is uniform with respect 
to the shallowness 
parameter (among others), allowing for the justification of various models used in coastal 
oceanography. In another direction, Wu \cite{Wu3,Wu4} and Germain-Masmoudi-Shatah \cite{GMS} showed almost global ($2D$) and global ($3D$) existence results under smallness conditions. 

In presence of surface tension, the analysis is technically more difficult, 
but there is no striking difference for the behavior of the solution. Local well-posedness remains
true and the solution of the water waves problem with surface tension converges to the solution of 
the water waves problem without surface tension when the surface tension coefficient goes to 
zero \cite{Yosihara2,AM2,AM3,MingZhang}. It is not known however whether this zero surface tension limit is compatible 
with the shallow water limit; the justification proposed in \cite{SW,Iguchi00,Iguchi0} 
for instance
of some shallow water models 
in presence of surface tension is relevant for large values of the surface tension 
(more precisely, for Bond numbers of order $O(1)$) but provide an existence time extremely short 
when the Bond number is large (small surface tension), which is the case in coastal oceanography 
for instance.

For the Rayleigh-Taylor problem ($\rho^->0$) the picture is completely different. It is known, 
at least for $1D$ interfaces, that, outside the analytic framework of 
\cite{SSBF,SS}, the evolution equations are ill-posed in absence of surface 
tension \cite{Ebin,ITT,KL}
-- see also \cite{LE,Wu2} for the related Kelvin-Helmholtz problem where $\rho^+=\rho^-$ -- even in 
very low regularity ($C^{1+\alpha}$ in \cite{KL,LE} and chord-arc in \cite{Wu2}). 
The reason of this ill-posedness is 
that the nonlinearity creates locally a discontinuity of the tangential velocity at the interface 
that induces Kelvin-Helmholtz instabilities. \\
More recently, it has been proved that taking into 
account the surface tension restores the local well-posedness of the equations \cite{Am,Ammas,ShatahZeng2,CCS1}. However, 
the existence time of the solution provided by these results is very small when 
the surface tension 
is small. The fact that the role of gravity (or gravity itself) is not considered 
in these references suggests that these general results can be improved in the ``stable'' 
configuration where the heavier fluid is placed below the lighter one. In particular, we would like to 
explain two physical phenomena
\begin{enumerate}
\item[$\sharp 1$ ]  Air-water interface. In coastal areas, for instance, 
waves propagate over several wavelength 
(depending on various physical parameters introduced below). This observation is consistent with 
the results obtained in \cite{AL} where the air density is neglected ($\rho^-=0$). If this density, which is small but non zero, is taken into account, the existence results of \cite{Am,Ammas,ShatahZeng2,CCS1} ensure that surface tension can prevent Kelvin-Helmholtz
instabilities, but for times that allow the wave to propagate over a few millimeters only. We therefore want to understand the mechanism
that controls Kelvin-Helmholtz instabilities for time scales 
consistent with the observations.
\item[$\sharp 2$] Internal waves. Waves at the interface of two immiscible fluids of possibly 
very similar densities have been widely investigated in the physics literature. 
Waves at the interface of two layers of water of different densities in the ocean are also 
commonly described with this formalism. 
As for the case of the air-water interface above, the known results only provide an existence time smaller by many orders of magnitude than what is observed.
\end{enumerate}

These two phenomena raise an apparent paradox:
\begin{enumerate}
\item[$\sharp 3$] The asymptotic models describing waves in coastal areas or internal waves, and which provide very accurate results, all neglect surface tension.  
On the other hand, the analysis of the full equations tells us that surface tension is
\emph{crucial} for the existence of a solution. In other words, \emph{surface tension allows interfacial waves to exist but does not play any role in their dynamics}.
\end{enumerate}
Very recently, some progress has been made to understand $\sharp 1$. 
It has been proved in \cite{CCS2,Pusateri1} that, the surface tension coefficient being kept fixed, the solution
of the two-fluid system converges to the solution of the water waves problem (with surface tension) as the density $\rho^-$ goes to zero. In \cite{Pusateri2}, the following important result was proved: the solution of the two-fluid system converges to the solution of the water waves equations (without surface tension) when \emph{both} the density $\rho^-$ and the surface tension go to zero. 
It is
just required that $(\rho^-)^2\leq \sigma^{7/3}$, where $\sigma$ is the
surface tension coefficient. A corollary of the result presented here is a slight improvement of this condition, showing that the result holds under the weaker
condition $(\rho^-)^2=o(\sigma)$ as suggested by the linear analysis.
However, in physical applications, the ratio $(\rho^-)^2/\sigma$ 
cannot be made arbitrarily small, and it would be interesting to have at our disposal
a criterion to check whether a given interfacial wave  will be ``stable'' (in the sense of ``observable on a relevant time scale'') or will almost instantaneously break down due to Kelvin-Helmholtz instabilities. 

The need for such a criterion is also crucial to explain $\sharp 2$. Indeed,
the smallness of the ratio $(\rho^-)^2/\sigma$ 
obviously cannot be invoked to explain the observation
of such phenomena as internal waves since the relative density of the upper fluid  is not small in that case (while $\sigma$ remains very small), and a better understanding of the formation of Kelvin-Helmholtz instabilities
is required.

\medbreak

A hint of what such a criterion should look like can be obtained by considering the related
problem of the \emph{linear} stability of two fluid layers moving with constant horizontal velocities  $c^+\neq c^-$ (with $c^\pm\in\R^d$, $d=1,2$). It is known since Kelvin that small
perturbations of the interface are linearly stable if the following criterion is satisfied
(see for instance \cite{Chandrasekhar})
$$
\mbox{(Kelvin)} \qquad  g({\rho}^+-{\rho}^-)>\frac{1}{4\sigma}\frac{({\rho}^+{\rho}^-)^2}{(\rho^++\rho^-)^2} {\mathfrak c}(0)\vert \llbracket c^\pm\rrbracket \vert^4,
$$
where $\jump{c^\pm}=c^+-c^-$ and ${\mathfrak c}(0)$ is a constant taking into
account the geometry of the problem\footnote{${\mathfrak c}(0)=1$ for the case of two layers of
infinite depth considered in \cite{Chandrasekhar}; see Remark \ref{constexpl} for an explicit
expression in the case of two layers of finite depth}.
The equivalent of this linear criterion in the water waves case is 
the linear Rayleigh-Taylor criterion already mentioned, namely, $\rho^+ g>0$.  In the problem we 
investigate here both fluids are at rest at infinity, but there is is a local discontinuity 
of the horizontal velocity created by the \emph{nonlinear} motion of the interfacial waves.
We can therefore expect that a local version of the Kelvin criterion may be derived
to assess the local linear stability of interface perturbations. A natural question is then: \emph{is there a nonlinear generalization of this expected local  Kelvin criterion?} Or equivalently, \emph{is there a generalization to two-fluids interfaces of the Rayleigh-Taylor criterion?}

The first main result of this paper is to show that such a nonlinear criterion exists.
It can be stated as 
\begin{equation}\label{SCintro}
\jump{-\dz P^\pm\,_{\vert_{z=\zeta}}}> \frac{1}{4}\frac{(\rho^+\rho^-)^2}{\sigma(\rho^++\rho^-)^2}
{\mathfrak c}(\zeta)\,\abs{\omega}_\infty^4,
\end{equation}
where  $\zeta$ is the interface parametrization, $\omega=\jump{V^\pm\,_{\vert_{z=\zeta}}}$ is the jump of the horizontal velocity
at the interface, and ${\mathfrak c}(\zeta)$ is a constant that depends on the geometry
of the problem (two layers of finite depth in this paper) and that can be estimated quite
precisely\footnote{for a flat interface $\zeta=0$, this constant is the same as in the
linear Kelvin criterion}.

\medbreak

The question is now to check whether the criterion (\ref{SCintro}) can explain  (for instance)
the observation
of the phenomena $\sharp 1$ and $\sharp 2$ mentioned above.\\
For $\sharp 1$ (air-water interface), it is likely that the relative density
$\urm=\rho^-/(\rho^++\rho^-)$ is small enough to make the r.h.s. of (\ref{SCintro}) smaller
than the l.h.s. even though $\sigma$ is small. This is at least the case if we consider
the limit $(\urm)^2/\sigma\to 0$ as in \cite{Pusateri2}.\\
For $\sharp 2$ (internal waves), the relative density $\urm$ is close to $1/2$ and 
cannot be used to explain why (\ref{SCintro}) should be satisfied in that case. Therefore,
\emph{the smallness of the  discontinuity on the horizontal velocity is the only reason that could explain $\sharp 2$}.

The main idea behind the stability criterion (\ref{SCintro}) is the following
explanation of the apparent paradox $\sharp 3$,
that we can state in rough mathematical terms as follows:
\emph{the Kelvin-Helmholtz instabilities appear above the frequency threshold for
which surface tension is relevant, while the main (observable) part of the
wave involves low frequencies located below this frequency threshold}; consequently, the Kelvin-Helmholtz
instabilities are regularized by surface tension, while the main part of the wave
is unaffected by it. The role of criterion (\ref{SCintro}) is to ensure that
the frequency threshold mentioned above is high enough.
For low frequencies, we show that gravity can stabilize the inertial
effects of the shear.   The main task in this paper is to prove rigorously this scenario.

\begin{remark}[Rayleigh-Taylor stability vs. Kelvin-Helmholtz instability] As
already mentioned, surface tension has been shown to regularize 
Kelvin-Helmholtz instabilities (that are due to the shear). As a matter of fact, it
also stabilizes Rayleigh-Taylor instability in the sense that it allows for a local
solution even if the heavier fluid is placed above the lighter one. Such a solution is
an example of what would be an unstable solution for us: even though the initial value problem is locally well-posed, it does not satisfy (\ref{SCintro}) and the existence time is very small. In the stable configuration (heavier fluid below), we said that we use gravity
to control the destabilizing effects of the shear; equivalently, Rayleigh-Taylor \emph{stability} controls the Kelvin-Helmholtz \emph{instabilities} at low frequencies.
\end{remark}

Having proved that (\ref{SCintro}) determines the stability of interfacial waves
(and constructed such solutions to the two-fluid equations),
one practical question arises: how can we have access to the quantities
$\jump{-\dz P^\pm}$ and $\abs{\omega}_\infty$ that we need to
know to determine whether (\ref{SCintro}) is satisfied or not?
Indeed, one would like to use this criterion to assess  
the stability of some given configuration: for instance, 
knowing the depth and density of both fluid layers, is it possible 
to have a stable perturbation of the interface of amplitude $a$ and wavelength 
$\lambda$?\\
 It is quite difficult to know $\jump{-\dz P^\pm}$ and $\abs{\omega}_\infty$
experimentally (see however \cite{Grue} for measurements of $\omega$). A useful
approach is therefore to perform an asymptotic analysis of the problem. Of particular
importance for applications is the shallow water regime (the wave length of the
perturbation is large compared to the depths $H^+$ and $H^-$
of the fluid layers). In this regime,
the pressure satisfies the hydrostatic approximation with a fairly good precision,
and one has therefore $\jump{-\dz P^\pm}\sim (\rho^+-\rho^-)g$, it is also
possible to show that $\omega$ has a typical size of order $\frac{a}{H}\sqrt{g'H}$
(with $g'=(\urp-\urm)g$ and $H=\frac{H^+ H^-}{\urp H^-+\urm H^+}$), a prediction
that is consistent with the measurements of \cite{Grue} for instance. Plugging
these approximations into (\ref{SCintro}) shows the relevance of the dimensionless
parameter $\Upsilon$ defined as
$$
\Upsilon=(\urp\urm)^2\frac{a^4}{H^2}\frac{(\rho^++\rho^-)g'}{4\sigma},
$$
and suggests a very simple practical stability criterion in shallow water
\begin{equation}\label{practSCintro}
\Upsilon\ll 1: \mbox{ Stable configuration; }
\quad
\Upsilon\gg 1: \mbox{ Unstable configuration}.
\end{equation}
When $\Upsilon\sim 1$, stability/instability is critical and
it is necessary to look at the exact criterion (\ref{SCintro}).

\medbreak

In order to prove rigorously the relevance of the practical criterion (\ref{practSCintro}), one must be able to handle the shallow water limit
in the construction of the stable solutions allowed by (\ref{SCintro}). Unfortunately,
this limit is singular, which complicates the proof. In particular, it restricts the number of tools at our disposal since standard symbolic analysis is not adapted to shallow water regimes. Roughly speaking, symbolic analysis neglects the information coming from 
the bottom since it is infinitely smoothing (this is the argument used in \cite{ABZ} 
to prove local well posedness of the water waves equations over very exotic bottoms). However, in the shallow limit, the influence of the bottom is the main factor in the evolution
of the wave: though infinitely smooth, its contribution becomes quantitatively very large. In the water waves case, various techniques allow one to bypass symbolic analysis and
justify some shallow water models \cite{AL,Iguchi}. In the two-fluid case under investigation here, it is not clear whether these techniques can be adapted. We therefore appealed to
a technical result of independent interest, namely, a symbolic analysis ``with tail'' of the Dirichlet-Neumann operator, the tail corresponding to the infinitely smoothing component of this operator taking into account the influence of the bottom. This allows us to use some (rudimentary) symbolic analysis without suffering from the shallow water singularity.

\medbreak

Once this analysis is done, we know that interfacial wave are well behaved in the
shallow water limit when the stability criterion (\ref{SCintro}) --- or its practical
version (\ref{practSCintro}) --- is satisfied. We then turn to apply this result to
the two phenomena $\sharp 1$ and $\sharp 2$ described above. The good agreement
with experimental data allow us to confirm the relevance of one-fluid asymptotic models for air-water interfaces, and to give the first rigorous justification (on the relevant time scale)
of two-fluid asymptotic models commonly used to describe internal waves. We also indicate how
some singularities of these models can be related to Kelvin-Helmholtz instabilities.
We finally discuss
two phenomena:
\begin{itemize}
\item[$\sharp 1$] Air-Water interfaces. 
Simple physical examples show that 
close to wave breaking, it is likely that (\ref{SCintro}) fails
to be satisfied for certain configurations. This would mean that \emph{the first singularity observed in wave breaking may not always be a singularity of the water waves equations, but sometimes a two-fluid singularity}. 
\item[$\sharp 2$] The description of water-brine interfaces (as for oceanic internal waves) does not fall into the range of equations (\ref{eqI-1})-(\ref{eqI-6}) because
water and brine are not immiscible. \emph{We use the stability criterion(s) derived in this
paper to propose a two-fluid description of the water-brine interface} that includes
an artificial surface tension $\sigma$. Using comparison with laboratory experiments we are able to propose a numerical value for $\sigma$.
\end{itemize}

\medbreak

In order to make a more detailed description of the results described above, we
first need to write the equations, and to derive their
nondimensionalized version.  This is done in the following two subsections.

\subsection{The equations}\label{secttheeq}

We assume that the interface is parametrized by a function $\zeta(t,X)$ ($X\in \R^d$) and denote by $\Omega_t^+$ and $\Omega_t^-$ the volume occupied by the lower and upper fluids
respectively at time $t$. Choosing the origin of the vertical axis to correspond with the interface between the two fluids at rest, we assume that $\Omega_t^+$ (resp. $\Omega_t^-$) is bounded below (resp. above) by an horizontal wall located at $z=-H^+$ (resp. $z=H^-$).
We also denote by $\Gamma_t$ the interface $\Gamma_t=\{(X,z),z=\zeta(t,X)\}$ and
by $\Gamma^\pm$ the upper and lower boundaries $\Gamma^\pm=\{z=\mp H^\pm\}$.

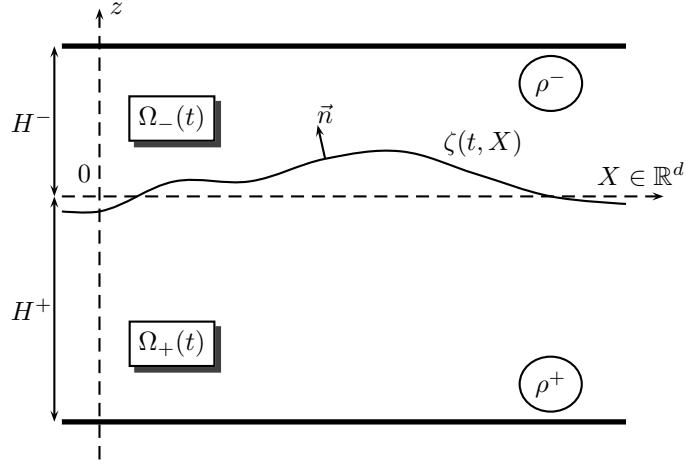
\begin{center}
\begin{pspicture}(-1,-3.5)(7.5,2.5)
\pscurve[showpoints=false](-0.5,-0.2)(0,-0.2)(1,0.2)(2,0.2)(3,0.5)(4,0.6)(5,0.3)(6,0)(7,-0.1)
\psline[linestyle=dashed]{->}(-0.5,0)(7.5,0)
\psline[linewidth=2pt](-0.5,-3)(7,-3)
\psline[linewidth=2pt](-0.5,2)(7,2)
\psline[linestyle=dashed]{->}(-0,-3.5)(-0,2.5)
\rput(3,1.1){$\vec n$}
\psline{->}(3,0.5)(2.9,0.95)
\rput(7.2,0.3){$X\in\R^d$}
\rput(0.23,2.5){$z$}
\psline{<->}(-0.6,0)(-0.6,2)
\rput(-0.9,1){$H^-$}
\rput(-0.2,0.3){$0$}
\psline{<->}(-0.6,0)(-0.6,-3)
\rput(-0.9,-1.5){$H^+$}
\rput(1,-2){\psshadowbox{$\Omega_+(t)$}}
\rput(1,1){\psshadowbox{$\Omega_-(t)$}}
\rput(6,1.5){\psovalbox{$\rho^-$}}
\rput(6,-2.5){\psovalbox{$\rho^+$}}
\rput(5.1,0.7){$\zeta(t,X)$}
\end{pspicture}
\end{center}

Finally, we denote by $\bU^\pm$ the velocity field in $\Omega_t^\pm$; the horizontal component of $\bU^\pm$ is written $V^\pm$ and its vertical one $w^\pm$. The pressure is denoted
by $P^\pm$.\\
For the sake of clarity, it is also convenient to introduce some notation to express the
difference and average of these quantity across the interface.
\begin{notation}
If $A^+$ and $A^-$ are two quantities (real numbers, functions, etc.), the notations $\jump{A^\pm}$ and $\av{A^\pm}$ stand for
$$
	\jump{A^\pm}=A^+-A^-\quad\mbox{ and }\quad \av{A^\pm}=\frac{A^++A^-}{2}.
$$
\end{notation}
We can now state the equations of motion.
\begin{itemize}
\item Equations in the fluid layers. In both fluid layers, the velocity field $\bU^\pm$ and the pressure $P^\pm$ satisfy the equations
\begin{equation}\label{eqI-1}
\dive \bU^\pm(t,\cdot)=0,\qquad \curl \bU^\pm(t,\cdot)=0,\qquad \mbox{ in } \Omega^\pm_t \quad (t\geq 0),
\end{equation}
which express the incompressibility and irrotationality assumptions, and
\begin{equation}\label{eqI-2}
\rho^\pm\big(\dt \bU^\pm+(\bU^\pm\cdot\grad)\bU^\pm\big)=-\grad P^\pm -g{\bf e}_z\qquad \mbox{ in } \Omega^\pm_t \quad (t\geq 0),
\end{equation}
which expresses the conservation of momentum (Euler equation).
\item Boundary conditions at the rigid bottom and lid. Impermeability of these two boundaries is classically rendered by
\begin{equation}\label{eqI-3}
w^\pm(t,\cdot)_{\vert_{\Gamma^\pm}}=0,\qquad (t\geq 0).
\end{equation}
\item Boundary conditions at the moving interface. The fact that the interface is a bounding surface (the fluid particles do not cross it) 
yields the equations
\begin{equation}\label{eqI-4}
\dt \zeta - \sqrt{1+\vert\nabla\zeta\vert^2}\bU^\pm_n=0,\qquad (t\geq 0),
\end{equation}
where $\bU^\pm_n:={\bU^\pm}_{\vert_{\Gamma_t}}\cdot{\bf n}$, ${\bf n}$ being the upward unit normal vector to the interface $\Gamma_t$. A direct consequence of (\ref{eqI-4}) is that
\emph{there is no jump of the normal component of the velocity at the interface}.
Finally, the continuity of the stress tensor at the interface gives in our particular case
\begin{equation}\label{eqI-5}
	\jump{P^\pm(t,\cdot)_\interft} =\sigma \mfk(\zeta) ,\qquad(t\geq 0),
\end{equation}
where $\sigma$ is the surface tension coefficient and $\mfk(\zeta)$ denotes the mean curvature of the
interface,
$$
\mfk(\zeta)=-\nabla\cdot\big(\frac{\nabla\zeta}{\sqrt{1+\abs{\nabla\zeta}^2}}\big).
$$
\end{itemize}

It is quite convenient for the mathematical analysis of the equations (\ref{eqI-1})-(\ref{eqI-5}) to transform them into a set of scalar equations on the interface.
 From the irrotationality assumption, we can write the velocities $\bU^\pm$ in terms of velocity 
potentials $\Phi^\pm$, namely,
$$
	\bU^\pm(t,\cdot)=\grad \Phi^\pm(t,\cdot)
	\qquad\mbox{ in }\Omega^\pm_t\quad (t\geq 0).
$$
Defining the trace of $\Phi^\pm$ at the interface by
$$
	\psi^\pm(t,\cdot )=\Phi^\pm(t,\cdot)_\interft \qquad (t\geq 0),
$$
we can now reduce the equations (\ref{eqI-1})-(\ref{eqI-5}) to a set of equations on $\zeta$ and $\psi^\pm$, namely,
\begin{eqnarray}
  \label{eqI-6}
	\dt \zeta -\Gp\psi^+=0,\\
  \label{eqI-7}
	\Gp\psi^+=\Gm\psi^-,\\
  \label{eqI-8}
	\rho^\pm\Big(\dt \psi^\pm +g\zeta+\frac{1}{2}\vert\nabla\psi^\pm\vert^2
        -\frac{(\Gpm\psi^\pm+\nabla\zeta\cdot\nabla\psi^\pm)^2}{2(1+\vert\nabla\zeta\vert^2)}\Big)=-P^\pm_\interft,\\
  \label{eqI-9}
	\jump{P^\pm(t,\cdot)_\interft}=\sigma \mfk(\zeta),
\end{eqnarray}
where $\Gpm$ are the Dirichlet-Neumann operators corresponding to the two fluid layers and defined 
(under reasonable assumptions on $\zeta$ and $\psi^\pm$) by
$$
	\Gpm\psi^\pm=\sqrt{1+\vert\nabla\zeta\vert^2}\dn \Phi^\pm\,_\interft,
$$
where $\Phi^\pm$ solve
\begin{equation}\label{eqI-10}
\left\lbrace
\begin{array}{l}
\dsp \Delta_{X,z}\Phi^\pm=0\quad\mbox{ in }\Omega_t^\pm,\\
\dsp \Phi^\pm\,_\interft=\psi^\pm,\qquad \dz \Phi^\pm\,_\interfpm=0
\end{array}\right.
\end{equation}
 and $\partial_n\Phi^\pm_\interft$ stands for the 
\emph{upward} partial derivative of $\Phi^\pm$ at the interface.

We now show\footnote{We do not justify the different steps of the derivation here; this will be done in Section \ref{sectPR}} that it is possible to reduce the two-fluid equations to a set of two equations on the surface elevation $\zeta$ 
and of the quantity $\psi$ defined as
$$
\psi:=\urp\psi^+-\urm\psi^-,
$$
where $\urpm$ stands for the relative density,
$
\urpm=\frac{\rho^\pm}{\rho^++\rho^-}
$ 
(in particular, $\urp+\urm=1$ and $\urp-\urm$ is the Atwood number).
Defining the operator $\G$ by
\begin{equation}\label{eqI-14}
	\G=\Gm\big(\urp \Gm-\urm \Gp\big)^{-1}\Gp,
\end{equation}
one can write $\Gp\psi^+=\Gm\psi^-$ in terms of $\zeta$ and $\psi$ only,
\begin{equation}\label{eqI-15}
\G\psi=\Gp\psi^+=\Gm\psi^-
\end{equation}
and one can also get $\psi^\pm$ in terms of $\zeta$ and $\psi$,
\begin{equation}\label{defpsipm}
\psi^\pm=\Gpm^{-1}\circ\G \psi.
\end{equation}
A formulation of the two-fluid equations as a system of two scalar evolution equations on $\zeta$ and $\psi$ can then be given,
\begin{equation}\label{eqI-16}
\left\lbrace
\begin{array}{lcl}
\dsp \dt \zeta - \G\psi&=&0,\vspace{.5mm}\\
\dsp \dt \psi+g'\zeta &+&
\dsp \frac{1}{2}\jump{\urpm\babs{\nabla\psi^\pm}^2} \vspace{.5mm}\\
\dsp&-&\dsp\frac{1}{2} \frac{\jump{\urpm(\Gpm\psi^\pm+\nabla\zeta\cdot\nabla\psi^\pm)^2}}{1+\abs{\nabla\zeta}^2}
= -\frac{\sigma}{\rho^++\rho^-} \mfk(\zeta),
\end{array}\right.
\end{equation}
where $\psi^\pm=\Gpm^{-1}\circ\G \psi$ and $g'$ stands for the reduced gravity,
$$
g'=(\urp-\urm) g.
$$
\begin{remark}
  In the case $\rho^-=0$, one has $\psi=\psi^+$ and $\G\psi=\Gp\psi^+$ and (\ref{eqI-16}) reduces to the formulation of the 
water waves equations in terms of the surface elevation $\zeta$ and the velocity potential at the surface $\psi^+$ due to Zakharov \cite{Zakharov}
 and Craig-Sulem \cite{Craig-Sulem}. In the case $\rho^-\neq 0$, $\zeta$ and $\psi$ are the 
two canonical variables of the Hamiltonian formulation exhibited by Benjamin and Bridges 
\cite{BB} (see also \cite{CGK})
\end{remark}

\subsection{Nondimensionalized equations}\label{sectND}

The order of magnitude of the existence time of the solution is our main concern in this paper. 
It is therefore crucial to get some information on the size of the different quantities appearing
in the equation, and in particular of those that play a role in our stability 
criterion (\ref{SCintro}).
These informations can be obtained by experimental measurements, 
but it is of course preferable to now them \emph{a priori} in terms of the
physical characteristics of the flow under consideration, namely,\\
\indent- The typical size $a$ of the interfacial waves\\
\indent- The typical wavelength $\lambda$ of these waves\\
\indent- The depth $H^\pm$ of the upper and lower fluid layers at rest\\
\indent- The value of the surface tension coefficient $\sigma$\\
\indent- The densities $\rho^\pm$ of the two fluids\\

 We show in Appendix \ref{apnd} how the linear theory can be used to obtain such an \emph{a priori}
estimate of the order of magnitude of the different unknowns. The best way to exploit this information
is to nondimensionalize the equations, that is, to perform a linear change of unknowns and variables 
such that all the unknowns are now dimensionless and of size $O(1)$ for typical configurations.

Writing (\ref{eqI-16}) in such a dimensionless form requires the introduction of various parameters;
the first two are given by
$$
\eps=\frac{a}{H},\qquad \mu=\frac{H^2}{\lambda^2},\quad\mbox{with}\quad
H=\frac{H^+ H^-}{\urp H^-+\urm H^+};
$$
in the water waves configuration, i.e. when $\urm=0$, $\eps$ is called the amplitude or nonlinearity parameter and $\mu$ the shallowness parameter. Throughout this article, $\mu$ is assumed
to remain bounded, as well as $\eps$. For notational convenience, and without loss of 
generality, we take
$$
0\leq  \mu\leq 1,\quad 0\leq \eps\leq 1.
$$
 The above definition of $\eps$ and $\mu$ takes into account both fluid layers at the same time. For a more specific description of each fluid layer, it is convenient
to introduce
$$
\eps^\pm=\frac{a}{H^\pm},\qquad \mu^\pm=\frac{(H^\pm)^2}{\lambda^2}.
$$
 We also need to
define the relative depth $\uH^\pm$ and the Bond number\footnote{The Bond number measures the ratio of gravity forces over capillary forces. For some reason, it often appears as the inverse of this quantity in the mathematics literature} $\Bo$ as
$$
\uH^\pm=\frac{H^+}{H},\qquad \Bo=\frac{(\rho^++\rho^-)g'\lambda^2}{\sigma}.
$$
\begin{remark}\label{remarkfullgen}
In full generality, one should allow $H^+$ to be much larger or much smaller
than $H^-$ (as in \cite{BLS} for instance). For the sake of clarity,
we assume throughout this paper that $H^+$ and $H^-$ are of same order (so that 
$\eps^+\sim\eps^-\sim\eps$ and $\mu^+\sim\mu^-\sim\mu$). There does not
seem to be any obstruction other than technical to generalize our method to more
general configurations.
\end{remark}
The next step is to introduce dimensionless Dirichlet-Neumann operators,
$$
	\Gpmhpm\psi^\pm=\sqrt{1+\vert\eps\nabla\zeta\vert^2}\dn \Phi^\pm_\inteps,
$$
where $\Phi^\pm$ solve
\begin{equation}\label{eqI-10nd}
\left\lbrace
\begin{array}{l}
\dsp (\mu\Delta+\dz^2)\Phi^\pm=0\quad\mbox{ in } -\uH^\pm<\pm z<\pm\eps\zeta,\\
\dsp \Phi^\pm\,_\inteps=\psi^\pm,\qquad \dz \Phi^\pm\,_{\vert_{z=\mp\uH^\pm}}=0,
\end{array}\right.
\end{equation}
 and $\partial_n\Phi^\pm\,_\inteps$ stands for the 
\emph{upward} conormal derivative of $\Phi^\pm$ at the interface (that is, $\Gpmu\psi^\pm=\dz\Phi^\pm\,_\inteps-\eps\mu\nabla\zeta\cdot\nabla\phi^\pm\,_\inteps$).
The operator $\G$ defined
in (\ref{eqI-14})  has similarly a dimensionless version given by
\begin{equation}\label{14nd}
	\Gmu={\mathpzc G}^-_\mu[\eps\zeta,\uH^-]\Big(\urp {\mathpzc G}^-_\mu[\eps\zeta,\uH^-]-\urm {\mathpzc G}^+_\mu[\eps\zeta,\uH^+]\Big)^{-1}{\mathpzc G}^+_\mu[\eps\zeta,\uH^+].
\end{equation}

We show in Appendix \ref{apnd} that the equations (\ref{eqI-16}) can then be written in dimensionless form 
 as
\begin{equation}\label{eqI-16nd}
\left\lbrace
\begin{array}{lcl}
\dsp \dt \zeta - \frac{1}{\mu}\Gmu\psi=0,\vspace{.5mm}\\
\dsp \dt \psi+\zeta +
\dsp \eps\frac{1}{2}\jump{\urpm\babs{\nabla\psi^\pm}^2} \vspace{.5mm}\\
\hspace{1.5cm}\dsp -\dsp \frac{1}{2}\frac{\eps}{\mu}\frac{\jump{\urpm\big(\Gpmhpm\psi^\pm+\eps\mu \nabla\zeta\cdot\nabla\psi^\pm
\big)^2}}{1+\eps^2\mu\abs{\nabla\zeta}^2}
= -\frac{1}{\Bo} \frac{1}{\eps\sqrt{\mu}}\mfk(\eps\sqrt{\mu}\zeta),
\end{array}\right.
\end{equation}
with $\psi^\pm={\mathpzc G}^+_\mu[\eps\zeta,\uH^+]^{-1}\circ\Gmu\psi$.
\subsection{Description of the results and organization of the paper}

We give in Section \ref{sectPR} some preliminary results that give a rigorous basis to 
the manipulations performed in \S\S \ref{secttheeq} and \ref{sectND} 
above to derive the equations (\ref{eqI-16}) and (\ref{eqI-16nd}). General notations
and definitions are given in \S \ref{sectionII1} and the Dirichlet-Neumann
operators ${\mathpzc G}_\mu^\pm[\eps\zeta,\uH^\pm]$ are studied in 
\S \ref{sectDN}. Their invertibility properties are investigated in (\S \ref{sectinv});
a technical difficulty is that the inverse of the DN operator takes values in Beppo-Levi spaces
(i.e. functions with gradient in Sobolev spaces, but not necessarily in $L^2$ since we
consider unbounded domains).
The structure of the shape derivatives of the DN operators are then discussed in \S \ref{sectSD}. We are then able to address in \S \ref{secttransm} a transmission problem that is
crucial to derive rigorously the two-fluid equations (\ref{eqI-16nd}) since it
states that the trace of the potential at the upper and lower part of the interface,
$\psi^-$ and $\psi^+$, are fully determined by $\zeta$ and $\psi=\urp\psi^+-\urm \psi^-$. This allows us to give a rigorous definition to the operator
$\Gb=\Gb_\mu[\eps\zeta]$ introduced in (\ref{14nd}). Since we want to be able to handle the shallow
water limit, we need to give two versions for (almost) all the estimates given in this 
section: one that is sharp with respect to regularity but not optimal with respect to the dependence on $\mu$, and one that is sharp with respect to $\mu$ but not optimal with respect to regularity. For instance, $\Gb_\mu^\pm[\eps\zeta,\uH^\pm]$ has an operator norm of size $O(\sqrt{\mu})$ if it
is seen, as usual, as a first order operator; but if we consider it as a second order operator,
we get a better control of size $O(\mu)$ on its operator norm (see Remark \ref{remsecond} below).

\medbreak

Section \ref{sectsymbolic} is devoted to the symbolic analysis of the Dirichlet-Neumann
operator $\Gpb=\uH^+\Gb_\mu^+[\eps\zeta,\uH^+]$. It is standard \cite{Taylor2,LannesJAMS,ShatahZeng,AlazardMetivier} that the principal symbol of $\Gpb$
can be written in terms of the Laplace-Beltrami associated to the surface.
More precisely, if we define the symbol $g(X,\xi)$ as
$$
g(x,\xi)=\sqrt{\abs{\xi}^2+\eps^2\mu(\abs{\nabla\zeta}^2\abs{\xi}^2-(\nabla\zeta\cdot\xi)^2)},
$$
then the following holds,
\begin{equation}\label{introsymb}
\forall 0\leq s\leq t_0,\quad 
\abs{\Gpb\psi-\sqrt{\mu^+}g(x,D)\psi}_{H^{s+1/2}}\leq M(t_0+3)\abs{\nabla\psi}_{H^{s-1/2}},
\end{equation}
where $t_0>d/2$ and $M(t_0+3)$ is a constant depending on $\abs{\zeta}_{H^{t_0+3}}$ and the minimal depth of the lower fluid. Since $\Gpb$ and $g(x,D)$ are
first order operator, this identity shows that $\Gpb$ can be replaced by
$g(x,D)$ up to a more regular (zero order) operator. Unfortunately,
such a substitution is not possible in the shallow water limit since this
zero order operator is of size $O(1)$ with respect to $\mu$, while
both $\Gpb$ and $g(x,D)$ have operator norms of size $O(\sqrt{\mu})$. Using
symbolic analysis induces therefore a singularity of order $O(\mu^{-1/2})$
when the shallow water limit is considered. The explanation of this behavior is that the shallow water regime corresponds to configurations where the bottom
plays a very important role; however, the contribution of the bottom to the
Dirichlet-Neumann operator is analytic and therefore neglected (at any order)
by symbolic analysis. In order to handle this difficulty, we propose a symbolic analysis ``with tail'' of the DN operator, that takes into account the infinitely smoothing (but very large) contribution of the bottom. This leads us to
replace the symbol $g(X,\xi)$ by $g(X,\xi)t^+(X,\xi)$, where
$$
t^+(X,\xi)=(1+\eps^+\zeta)\int_{-1}^0\frac{\sqrt{\abs{\xi}^2+\eps^2\mu(z+1)^2(\abs{\nabla\zeta}^2\abs{\xi}^2-(\nabla\zeta\cdot\xi)^2)}}{1+\eps^2\mu(z+1)^2\abs{\nabla\zeta}^2}dz
$$
(when $d=1$, this simplifies into 
$\dsp t^+(x,\xi)=(1+\eps^+\zeta)\frac{\arctan(\eps\sqrt{\mu}\partial_x\zeta)}{\eps\sqrt{\mu}\partial_x\zeta}\abs{\xi}$).
We then show in Theorem \ref{theotail} that (\ref{introsymb}) can be improved
into
$$
\babs{\Gpb\psi-\sqrt{\mu^+}\Op\big[g \tanh\big(\sqrt{\mu^+}t^+\big)\big]\psi}_{H^{s+1/2}}\leq \eps\sqrt{\mu}M(t_0+3)\abs{\nabla\psi}_{H^{s-1/2}}.
$$
The cost in terms of derivatives of $\psi$ is the same as in (\ref{introsymb}),
but the behavior with respect to the parameters $\eps$ and $\mu$ is much
better since there is a gain of size $O(\eps\sqrt{\mu})$. This allows us
to handle the shallow water limit. We also give in Section \ref{sectsymbolic} the
symbolic analysis (with tail) of $(\Gmb)^{-1}\Gpb$, $(\Gmb)^{-1}\Gb$
and $\Gb(\Gmb)^{-1}\partial_j$ that are needed for the analysis of the two-fluid equations.

\medbreak

In Section \ref{sectquasi}, we show that it is possible to quasilinearize the 
two-fluid equations (\ref{eqI-16nd}) by differentiating them and writing them in terms
of the relevant unknowns. We first introduce and study in \S \ref{sectdefnew} some
new operators that are needed to write the quasilinearized equations. An important step
is a linearization formula for the operator $\Gb$ that is given in \S \ref{sectlin}; the main step to establish this linearization formula is an explicit formula for the derivatives
of the mapping $\zeta\mapsto \Gb_\mu[\eps\zeta]\psi$ (shape derivative). This linearization formula also suggests what the good unknown should be. More precisely, the
equations (\ref{eqI-16nd}) differentiated $\alpha$-times should not be written in
terms of $\partial^\alpha\zeta$ and $\partial^\alpha\psi$ but in terms
of $\zetaa$ and $\psia$ defined as
$$
\zetaa=\partial^\alpha\zeta,\qquad \psia=\jump{\urpm\partial^\alpha\psi^\pm}
-\jump{\urpm\uwpm}\partial^\alpha\zeta,
$$
where $\uw^\pm$ is the vertical component of the velocity at the interface in the $\pm$ fluid, while $\psi^\pm$ is as in (\ref{eqI-16nd}). When $\urm=0$, these unknowns coincide
of course with the ones used in \cite{Iguchi,RoussetTzvetkov} to write the quasilinearized
water waves equations.
The quasilinearized equations are then derived in \S \ref{sectQL}. Differentiating
$\alpha$ times the second equation of (\ref{eqI-16nd}) and rewriting it in terms
of $\zetaa$ and $\psia$, one gets (without surface tension)
$$
\dt \psi_{(\alpha)}+\mfa \partial^\alpha\zeta+\eps\av{\uVpm}\cdot\nabla\psia
+\eps\jump{V^\pm}\cdot\nabla{\av{\urpm\psiapm}}
\sim 0
$$
(the symbol $\sim 0$ means that harmless terms are omitted); in this expression, 
$\mfa$ can be related to the (dimensionless) jump of the vertical derivative of the
pressure at the interface, and $\uVpm$ stands for the horizontal velocity
at the interface in the fluid $\pm$. In order to write this equation
in terms of the good unknowns $\zetaa$ and $\psia$, it is necessary to write
$\av{\urpm\psia^\pm}$ in terms of $\jump{\urpm\psia^\pm}=\psia$ and $\zetaa$. This is completely trivial in the water waves case $\urm=0$. The dependence of this term on $\zetaa$ is therefore specific to the two-fluid system, \emph{and it is responsible for the Kelvin-Helmholtz instabilities} --- through the operator $\E$ in the quasilinear system given in 
Proposition \ref{propIII-1}.

\medbreak

The main existence and stability results are then given in Section \ref{sectmain}. We first 
state in \S \ref{sectSC} several versions of the stability criterion. The first one (see \S \ref{sectSC1}) differs slightly from (\ref{SCintro}) because it also involves the $L^\infty$-norm of first order space-time derivatives of the jump of the horizontal velocity. In \S \ref{sectSC2}, we show how to recover (\ref{SCintro}) when the jump of velocity is nonzero. The practical criterion (\ref{practSCintro}) is then derived in
\S \ref{sectSC3}. After some considerations on the conditions we impose on the initial data
(see \S \ref{sectinit}), we state in \S \ref{sectlocal} the main results of this paper, which show that
\emph{``stable'' solutions exist under the criterion given in \S\S \ref{sectSC1} and \ref{sectSC2}} (we give two versions of the theorem). By stable, we mean that the existence time depends on 
the surface tension $\sigma$ through the stability criterion (\ref{SC}) only. In particular,
it is not necessarily small if $\sigma$ is small. We want to emphasize on the
fact that the theorems are given under a nondimensionalized form and that the existence time
is uniform with respect to the parameters $\eps$ and $\mu$. In particular, this allows
one to handle the shallow water limit (or other regimes) straightforwardly. Other kinds of
limits (zero density and zero surface tension) are also direct corollaries. We then prove in \S 
\ref{sectpersist} that these stable solutions persist over large times (i.e. of size $O(1/\eps)$) if a stronger
version of the stability criterions is satisfied.\\
 The proofs of the theorems are then given in \S \ref{sectproofth}. The stability criterion appears in the key Lemma \ref{lemmequive} 
where
it is used to control the destabilizing effects of the shear by the surface tension and gravity terms.
As in the water waves case, some technical difficulties arise because of the second order surface tension term that requires special care in the treatment of the subprincipal terms in the symmetrization process. We chose to adapt the technique introduced in \cite{RoussetTzvetkov} (see also \cite{RoussetTzvetkov2} for a synthetic presentation), which is probably not the sharpest one in terms of regularity (see for instance 
\cite{ShatahZeng,AM2,AM3,MingZhang,ABZ,CHS} for alternatives methods) but seems to be the most robust with respect to the shallow water limit, in particular because it does not require symbolic analysis. This technique requires that time derivative be treated as space derivatives, and this is the reason why this is done throughout this paper.

\medbreak

We finally sketch in Section \ref{sectappl} some applications for the stability criterion (\ref{SCintro})
and its practical version (\ref{practSCintro}). We consider in \S \ref{sectairwater}
the case of air-water interfaces characterized by a very small density of the upper fluid. We first confirm
in \S \ref{sectvalidity}
that the density of the air can be neglected in the asymptotic models used in coastal
oceanography to describe the propagation of waves (we
consider here the example of a so called long wave). For waves close to the breaking point, 
we show in \S \ref{sectKHI} that this might not be true anymore and that Kelvin-Helmholtz
instabilities may appear. We therefore suggest that wave breaking may sometimes be a two-fluid singularity rather than a singularity of the water waves (one-fluid) equations. Links
with physical phenomena such as spilling breakers and white caps are also commented.\\
We then consider in \S \ref{sectinternal} the case of internal waves at the interface of two
fluids of comparable density (stability of the fluid is therefore ensured by a small shear velocity).  We first check that some ``stable'' configurations
reported in experimantal works do satisfy our stability criterion. We then consider the case of water-brine interfaces, of interest for the study of oceanographic internal waves. There
is no natural value for the surface tension for such interface (water and brine are not immiscible), but investigating the occurrence of Kelvin-Helmholtz instabilities thanks to
the experiments of \cite{Grue}, we show how to use our stabiliy criterion to propose
a value for the artificial surface tension that must be introduced to use the
two-fluid formalism.\\
We end this section by showing how to justify rigorously the two-fluid asymptotic models
used in applications, and on a time scale consitent with the observations. We give the
details for a shallow-water/shallow-water model and explain how one of its singularities
can be related to Kelvin-Helmholtz instabilities. 

\subsection{Basic notations}

- We denote by $C(\lambda_1,\lambda_2,\dots)$ a constant
depending on the parameters $\lambda_1,\lambda_2,\dots$ and \emph{whose dependence
on the $\lambda_j$ is always assumed to be nondecreasing}.\\
- We denote by $X=(X_1,\dots,X_d)\in\R^d$ the horizontal variables, by $z$ the vertical one, and by $t$ the time variable.\\
- We denote by $\partial_j$ ($1\leq j\leq d$) partial differentiation with respect to $X_j$; partial differentiation with respect to $t$ and $z$ is denoted by $\dt$  and $\dz$.\\
- We denote by $\nabla$ and $\Delta$ the standard gradient and Laplace operators in the horizontal coordinates. These operators are denoted by $\nabla_{X,z}$ and $\Delta_{X,z}$ when they also take into account the vertical variable.\\
- We denote by $\nam=(\sqrt{\mu}\nabla^T,\dz)^T$, that is $\nam$ corresponds to $\grad$ with
a factor $\sqrt{\mu}$ in front of each \emph{horizontal} derivative.\\
- We denote $\Lambda:=(1-\Delta)^{1/2}$ and $H^s(\R^d)$ ($s\in\R$)
the usual Sobolev space 
$H^s(\R^d)=\{u\in {\mathcal S}'(\R^d),\vert u\vert_{H^s}<\infty\}$, 
where $\vert u\vert_{H^s}=\vert \Lambda^s u\vert_{L^2}$.
We keep this notation if $u$ is a vector or matrix with coefficients
in $H^s(\R^d)$.\\
- We denote by $(\cdot,\cdot)$ the standard $L^2$-scalar product.\\
- The notation $a\vee b$ stands for $\max\{a,b\}$.\\
- We denote by $\Op(\sigma)$ or $\sigma(x,D$ the pseudodifferential operator of symbol $\Sigma(x,\xi)$,
$$
\Op(\sigma)u(x)=(2\pi)^{-d}\int_{\R^d}e^{ix\cdot\xi}\sigma(x,\xi)\widehat{u}(\xi)d\xi.
$$ 

\section{Preliminary results}\label{sectPR}

\subsection{Notations and definitions}\label{sectionII1}

\subsubsection{General notations}

Throughout this section, we always assume that the interface deformation $\zeta\in H^{t_0+1}(\R^d)$ ($t_0>d/2$) does not touch the
bottom nor the lid. In dimensionless variables, this condition reads
\begin{equation}\label{sectII1}
\exists h_{min}^\pm>0,\qquad \inf_{X\in\R^d}(1\pm\eps^\pm\zeta(X))\geq h_{min}^\pm.
\end{equation}
We also denote by $\Omega^\pm$ the domains
\begin{eqnarray*}
\Omega^+&=&\{(X,z)\in\R^{d+1},-1<z<\eps^+\zeta(X)\}\\
\Omega^-&=&\{(X,z)\in\R^{d+1}, \eps^-\zeta(X)<z<1\}.
\end{eqnarray*}

We finally introduce, for notational convenience, a constant  $M$  defined as\footnote{the dependence on $H^+/H^-$ and $H^-/H^+$ is harmless because we assumed that $H^+$ and $H^-$
are of same order (see Remark \ref{remarkfullgen}). For more general configurations,
a finer analysis on the dependence on $H^\pm$ of the solution is needed}
\begin{equation}\label{eqM}
M=C(\frac{1}{h_{min}^\pm},\frac{H^+}{H^-},\frac{H^-}{H^+},\abs{\zeta}_{H^{t_0+2}}),
\end{equation}
as well as
\begin{equation}\label{eqMbis}
M(s)=C(M,\abs{\zeta}_{H^s}),
\end{equation}
where we recall that $C(\cdot)$ denotes generically a nondecreasing, positive, function of its arguments.

\subsubsection{Diffeomorphisms}\label{sectdiff}

It is often convenient to transform boundary value problems on the fluid domains $\Omega^\pm$
into boundary value problems on the flat strips $\cS^\pm$ defined as
$$
	\cS^+=\{(X,z),-1<z<0\}
	\quad\mbox{ and }\quad
	\cS^-=\{(X,z),0<z<1\},
$$
so that $\cS^\pm$ correspond do the two fluid domains at rest. Various diffeomorphisms can be used for
such domain transformations. We only use here \emph{admissible} diffeomorphisms in the following sense:
\begin{definition}[Admissible diffeomorphisms] \label{defII1-1}
Let $t_0>d/2$ and $\zeta\in H^{t_0+2}(\R^d)$
be such that (\ref{sectII1}) is satisfied.\\
We say that $\Sigma^\pm:\cS^\pm\to\Omega^\pm$ is an \emph{admissible diffeomorphism}
if
\begin{enumerate}
\item $\Sigma^\pm$ can be extended to the boundaries in such a way that 
$$
\Sigma^\pm(\{z=0\})=\{z=\eps^\pm\zeta\}
\quad\mbox{ and }\quad
\Sigma^\pm(\{z=\mp 1\})=\{\zeta=\mp 1\}.
$$
\item The  coefficients of the Jacobian matrix $J_{\Sigma^\pm}$ are bounded on $\cS^\pm$, and
$
\abs{J_{\Sigma^\pm}}_{L^\infty(\cS^\pm)}
\leq M.
$
\item The determinant $\abs{\det J_{\Sigma^\pm}}$ is uniformly
bounded from below on $\cS^\pm$ by a nonnegative constant $c^\pm$ such that
$
\frac{1}{c^\pm}  \leq M.
$
\end{enumerate}
\end{definition}

If $\Phi^\pm$ is defined on $\Omega^\pm$, 
then one can define $\phi^\pm=\Phi^\pm\circ\Sigma^\pm$ on $\cS^\pm$; a direct application of
the chain rule shows that
$$
	\nampm\Phi^\pm\circ\Sigma^\pm=I^{\mu^\pm}(J_{\Sigma^\pm}^{-1})^T(I^{\mu^\pm})^{-1}\nampm\phi^\pm,
$$
where $I^{\mu^\pm}$ is the $(d+1)\times(d+1)$ diagonal matrix with entries $\sqrt{\mu^\pm}$ on the $d$-first diagonal coefficients, and $1$ on the last one.\\
In particular, the (variational formulation of the) equation $(\dz^2+\mu^\pm\Delta)\Phi^\pm=0$ in $\Omega^\pm$ is transformed by
$\Sigma^\pm$ into the (variational formulation of the) 
following variable coefficient elliptic equation on $\cS^\pm$,
\begin{equation}\label{eqell}
	\nampm\cdot P(\Sigma^\pm)\nampm\phi^\pm=0
	\quad\mbox{ in }\quad \cS^\pm,
\end{equation}
where
\begin{equation}\label{PS}
	P(\Sigma^\pm)=\abs{\det J_{\Sigma^\pm}}(I^{\mu^\pm})^{-1}J_{\Sigma^\pm}^{-1}(I^{\mu^\pm})^{2}(J_{\Sigma^\pm}^{-1})^T(I^{\mu^\pm})^{-1};
\end{equation}
it is easy to check (Prop. 2.3 of \cite{AL}) that the matrix  $P(\Sigma^\pm)$ is 
uniformly coercive with coercivity constant $k(\Sigma^\pm)$. We assume further that
\begin{equation}\label{eqk}
	\Abs{P(\Sigma^\pm)}_\infty\leq M
        \quad\mbox{ and }\quad
        \frac{1}{k(\Sigma^\pm)}\leq M. 
\end{equation}
It is also convenient to assume that $P(\Sigma^\pm)$ is a Sobolev perturbation of the identity matrix (this is in general a direct consequence of the fact that $\zeta$ has Sobolev regularity) in the
sense that
\begin{equation}\label{Psob}
\Abs{\Lambda^{t_0+1}(\Id-P(\Sigma^\pm))}_{L^\infty_zL^2_X}\leq M.
\end{equation}

Let us also mention that we use the notation $\dn$ for the transformed equations to denote the
upward \emph{conormal} derivative naturally associated to any elliptic equation. For instance,
the upward conormal derivative associated to (\ref{eqell}) is
\begin{equation}\label{conormal}
	\dn\phi^\pm={\bf e_z}\cdot P(\Sigma^\pm)\nampm\phi^\pm.
\end{equation}

We now introduce as in \cite{LannesJAMS} (see also \cite{BeyerGunther})
a very useful class of admissible diffeomorphisms called
\emph{regularizing} diffeomorphisms.
\begin{definition}[Regularizing diffeomorphisms]\label{defII1-2}
Let $t_0>d/2$, $s\geq 0$ and $\zeta\in H^{t_0+2}\cap H^s(\R^d)$
be such that (\ref{sectII1}) is satisfied.\\
We say that an admissible diffeomorphism $\Sigma^\pm:\cS^\pm\to\Omega^\pm$ is
\emph{regularizing} if 
it is of the form $\Sigma^\pm(X,z)=(X,z+\sigma^\pm(X,z))$ with
\begin{equation}\label{eqreg0}
\Abs{\Lambda^s\nam \sigma^\pm}_{L^2(\cS^\pm)}+\Abs{\Lambda^{s-1}\dz \nam \sigma^\pm}_{L^2(\cS^\pm)}\leq M(s+1/2).
\end{equation}
\end{definition}
Note that if $\Sigma^\pm$ is regularizing, then one automatically has from (\ref{PS}) that
\begin{equation}\label{eqreg}
	\Abs{\Lambda^s(P(\Sigma^\pm)-\Id)}_{L^2(\cS^\pm)}\leq
        M(s+1/2).
\end{equation}
\begin{example}\label{exII1-1}
The most simple example of admissible diffeomorphism is given by
$$
	\Sigma^\pm(X,z)=(X,\eps^\pm(1\pm z)\zeta(X)+z).
$$
It satisfies (\ref{eqk}) and (\ref{Psob}), but not (\ref{eqreg0}) (the r.h.s. should be $M(s+1)$ in this inequality).
\end{example}
\begin{example}\label{exII1-2}
For small enough $\delta>0$ (and with $\chi:\R\to\R$ a smooth, positive, compactly supported function such that $\chi(0)=1$), the functions $\Sigma^\pm$ defined as
$$
	\Sigma^\pm(X,z)=(X,z+\sigma^\pm(X,z))
        \quad \mbox{ with }\quad
        \sigma^\pm(\cdot,z)=\eps^\pm(1\pm z)\chi(\delta z\abs{D})\zeta
$$
are regularizing diffeomorphisms  (see \cite{LannesJAMS}) also satisfying (\ref{eqk}) and (\ref{Psob}).
\end{example}

\subsubsection{Functional spaces}

We introduce here two functional spaces that are closely related to the energy norm. The space $H_\sigma^{N+1}(\R^d)$ is introduced to measure the extra-control (with respect to the $H^N$-norm) provided by the surface tension term and is defined for $N\in\N$ as
\begin{equation}\label{defHsig}
H^{N+1}_\sigma(\R^d)=\{u\in H^{N}(\R^d),\frac{1}{\sqrt{\Bo}}\nabla u\in H^{N}(\R^d)^d\},
\end{equation}
endowed with the norm $\dsp \abs{u}_{H^{N+1}_\sigma}^2=\abs{u}_{H^N}^2+\frac{1}{\Bo}\abs{\nabla u}_{H^N}^2$.\\
Since the velocity potential does not necessary decay at infinity (while the velocity does), we
have to work in Beppo-Levi spaces \cite{DenyLions} on open subsets $\Omega\subset \R^{d+1}$
defined for $N\geq 0$ as
$$
\dot{H}^{N+1}(\Omega)=\{u\in L^2_{loc}(\Omega),\nabla_{X,z} u\in H^{N}(\Omega)^{d+1}\},
$$
endowed with the (semi)-norm $\Abs{u}_{\dot{H}^{N+1}}=\Abs{\nabla u}_{H^{N}}$. We also
 define similar spaces over $\R^d$,
$$
\dot{H}^{s+1/2}(\R^d)=\{u\in L^2_{loc}(\R^d),\nabla u\in H^{s-1/2}(\R^d)^{d}\}\qquad
(s\in\R),
$$
endowed with the (semi)-norm $\abs{u}_{\dot{H}^{s+1/2}}=\abs{\nabla u}_{H^{s-1/2}}$. In order to capture the shallow-water dynamics, it is convenient to introduce a variant of this space that depends on the shallowness parameter $\mu$,
\begin{equation}\label{eqHm}
\dot{H}_\mu^{s+1/2}(\R^d)=\dot{H}^{s+1/2}(\R^d) \quad\mbox{ endowed with }\quad
\abs{u}_{\Hm^{s+1/2}}=\abs{\Pp u}_{H^s},
\end{equation}
where $\Pp$ is the nonhomogeneous Fourier multiplier of order $1/2$ defined as
\begin{equation}\label{eqPpm}
\Pp=\frac{\abs{D}}{(1+\sqrt{\mu}\abs{D})^{1/2}}.
\end{equation}
\subsection{A few results on Dirichlet-Neumann operators}\label{sectDN}

\subsubsection{Two equivalent definitions}

If $\zeta\in H^{t_0+2}(\R^d)$ ($t_0>d/2$) satisfies (\ref{sectII1}) then it is well known that for all $\psi^\pm\in \dot{H}^{3/2}(\R^d)$, there 
exists a unique solution $\Phi^\pm\in \dot{H}^2(\Omega^\pm)$ to the elliptic boundary value problem
\begin{equation}\label{eqsectII1-1}
\left\lbrace
\begin{array}{l}
(\dz^2+\mu^\pm\Delta)\Phi^\pm=0\quad\mbox{ in }\quad \Omega^\pm,\\
\Phi^\pm_{\vert_{z=\eps^\pm\zeta}}=\psi^\pm,\qquad
\dz \Phi_{\vert_{z=\mp 1}}=0
\end{array}\right.
\end{equation}
(since the interface  is a graph, it is possible to view the trace of $\Phi^\pm$ on the interface
 as a function defined on $\R^d$)
so that the following definition of the Dirichlet-Neumann operators ${\mathpzc G}^\pm_{\mu^\pm}[\eps^\pm\zeta,1]$ makes sense,
$$
{\mathpzc G}^\pm_{\mu^\pm}[\eps^\pm\zeta,1]:
\begin{array}{lcl}
  \dot{H}^{3/2}(\R^d) & \to & H^{1/2}(\R^d)\\
  \psi^\pm  &\mapsto& \sqrt{1+\abs{\eps^\pm\nabla\zeta}^2}\dn \Phi^\pm_{\vert_{z=\eps^\pm\zeta}},
\end{array}
$$
where $\Phi^\pm$ solves (\ref{eqsectII1-1}) and $\dn$ denotes the \emph{upwards} conormal derivative.
\begin{remark}
The operators ${\mathpzc G}^\pm_{\mu}[\eps\zeta,\uH^\pm]$ that appear in the introduction are related
to ${\mathpzc G}^\pm_{\mu^\pm}[\eps^\pm\zeta,1]$ through the scaling law
\begin{equation}\label{scalGN}
\Gpmhpm=\frac{1}{\uH^\pm}{\mathpzc G}^\pm_{\mu^\pm}[\eps^\pm\zeta,1].
\end{equation}
\end{remark}
\begin{notation}
For the
sake of simplicity, we simply write
$$
\Gpmb={\mathpzc G}^\pm_{\mu^\pm}[\eps^\pm\zeta,1],
$$
when no confusion is possible.
\end{notation}
From the discussion in \S \ref{sectdiff}, $\Gpmb\psi^\pm$ can be equivalently defined as
\begin{equation}\label{eqsectII1-2}
\Gpmb\psi^\pm=\dn{\phi^\pm}_\interff
\end{equation}
where $\dn$ now stands for the upwards conormal derivative associated to the elliptic operator in the
flat strip (see (\ref{conormal})) solved by $\phi^\pm$, 
\begin{equation}\label{eqsectII1-3}
\left\lbrace
\begin{array}{l}
\nampm\cdot P(\Sigma^\pm)\nampm\phi^\pm=0 \quad\mbox{ in }\quad \cS^\pm,\\
{\phi^\pm}_\interff=\psi^\pm,\qquad \dn{\phi^\pm}_{\vert_{z=\mp 1}}=0
\end{array}\right.
\end{equation}
($\Sigma^\pm$ being an admissible diffeomorphism in the sense of Definition \ref{defII1-1}).

\subsubsection{Basic estimates}

Let $t_0>d/2$ and $\zeta\in H^{t_0+2}(\R^d)$. 
If $\Sigma^\pm$ is an admissible diffeomorphism in the sense of Definition \ref{defII1-1} then
for all $\psi^\pm\in\dot{H}^{3/2}(\R^d)$, there exists a unique solution $\phi^\pm\in \dot{H}^2(\cS^\pm)$ to (\ref{eqsectII1-3}). Moreover, one has (see \cite{LannesJAMS,Iguchi}),
\begin{equation}\label{estreg}
\forall 0\leq s\leq t_0+1,\qquad
\Abs{\Lambda^s\nampm\phi^\pm}_{L^2(\cS^\pm)}\leq M \sqrt{\mu}\abs{\psi^\pm}_{\Hm^{s+1/2}},
\end{equation}
where $M$ is as in (\ref{eqM}) and $\Hm^{s+1/2}(\R^d)$ as in (\ref{eqHm}).
The reverse inequality also holds (see Appendix \ref{appDNstar}),
\begin{equation}\label{DNstar}
\sqrt{\mu}\abs{\psi^\pm}_{\Hm^{s+1/2}}\leq M \Abs{\Lambda^s\nampm\phi^\pm}_{L^2(\cS^\pm)}.
\end{equation}

For all $\zeta\in H^{t_0+2}(\R^d)$ ($s\geq 0$), (\ref{eqsectII1-2}) and (\ref{estreg})
allow one to
extend $\Gpmb$ as a mapping
$$
\forall 0\leq s\leq t_0+1,\qquad \Gpmb: \dot{H}^{s+1/2}(\R^d)\to H^{s-1/2}(\R^d).
$$
Moreover, the following estimates hold (they correspond to the particular case $j=0$
of (\ref{estder0})-(\ref{estder0k}) below)
\begin{eqnarray}\label{eqsectII1-4sam}
\forall 0\leq s\leq t_0+1, & &\abs{\Gpmb\psi}_{H^{s-1/2}}\leq {\mu}^{3/4}
M \abs{\Pp\psi}_{H^{s}},\\
\label{eqsectII1-4sambis}
\forall 0\leq s\leq t_0+1/2, & &\abs{\Gpmb\psi}_{H^{s-1/2}}\leq {\mu}
M \abs{\Pp\psi}_{H^{s+1/2}}.
\end{eqnarray}
\begin{remark}\label{remsecond}
From the definition (\ref{eqPpm}) of $\Pp$, it follows easily that $\frac{1}{{\sqrt{\mu}}}\Gpmb$ remains
uniformly bounded  as a \emph{first} order  operator as $\mu\to 0$. However, the quantity
$\frac{1}{{{\mu}}}\Gpmb$ involved in the first equation of (\ref{eqI-16nd}) remains
bounded only if we consider $\Gpmb$ as a \emph{second} order operator. This fact makes the
shallow limit singular and requires some specific attention.
\end{remark}

It also follows upon integrating by parts in (\ref{eqsectII1-3}) that 
for all $\psi_1,\psi_2\in\dot{H}^{1/2}(\R^d)$,
\begin{equation}\label{eqsectII1-5}
\int_{\cS^\pm}\nampm\phi_1^\pm\cdot P(\Sigma^\pm)\nampm\phi_2^\pm=\pm\int_{\R^d}\psi_1(\Gpmb\psi_2),
\end{equation}
where $\phi^\pm_1$ and $\phi^\pm_2$ the (variational) solutions to (\ref{eqsectII1-3}) with Dirichlet condition $\psi_1$ and $\psi_2$.  This identity is the key ingredient to prove the following estimate
(see Appendix \ref{proofII1-6}),
\begin{equation}\label{eqsectII1-6}
\forall 0\leq s\leq t_0+1,\qquad \babs{(\Lambda^s\Gpmb\psi_1,\Lambda^s\psi_2)}\leq 
M \mu \abs{\psi_1}_{\Hm^{s+1/2}}\abs{\psi_2}_{\Hm^{s+1/2}},
\end{equation}
for all $\psi_1,\psi_2\in\dot{H}^{s+1/2}(\R^d)$. Finally, we will also need the 
following commutator estimate (see Appendix \ref{proofcommut} for a proof)
\begin{equation}\label{DNcommut}
\forall 0\leq s\leq t_0+1,\qquad
\babs{([\Lambda^s,\Gpmb]\psi_1,\Lambda^s\psi_2)}\leq \mu M \abs{\psi_1}_{\Hm^{s-1/2}}\abs{\psi_2}_{\Hm^{s+1/2}},
\end{equation}
and the inequality (see Proposition 3.7 of \cite{AL}),
\begin{equation}\label{idcuba}
\big({\bf v}\cdot\nabla u,\Gpmb u\big)\leq \mu M\abs{{\bf v}}_{W^{1,\infty}}\abs{\Pp u}_2^2,
\end{equation}
that holds for all ${\bf v}\in W^{1,\infty}(\R^d)^d$ and
$u\in\Hm^{1/2}(\R^d)$.

\subsubsection{Invertibility}\label{sectinv}

It follows quite easily from (\ref{eqsectII1-5}) 
that $\Gpb$ (resp. $\Gmb$) is a symmetric, positive (resp. negative) operator. We can also
deduce that they are injective (up to constants of course) so that the inverse $(\Gpmb)^{-1}$ is well defined
on the range of $\Gpmb$ and with values in $\dot{H}^{1/2}(\R^d)$.
The next proposition states that $(\Gmb)^{-1}\circ\Gpb$ is well defined and uniformly
bounded (with respect to $\eps$ and $\mu$) as a family of operators mapping $\Hm^{s+1/2}(\R^d)$ into itself.
\begin{proposition}\label{prop1}
Let $t_0>d/2$  and $\zeta\in H^{t_0+2}(\R^d)$ be such that (\ref{sectII1}) 
is satisfied. For all $0\leq s\leq t_0+1$, the mapping
$$
	(\Gmb)^{-1}\circ \Gpb:
	\begin{array}{lcl}
	\Hm^{s+1/2}(\R^d)&\to&\Hm^{s+1/2}(\R^d)\\
	\psi &\mapsto & (\Gmb)^{-1}(\Gpb\psi)
	\end{array}
$$
is well defined and one has, 	for all $\psi\in\Hm^{s+1/2}(\R^d)$,
$$
	\abs{(\Gmb)^{-1}\circ\Gpb\psi}_{\Hm^{s+1/2}}\leq 
        M
	\abs{\psi}_{\Hm^{s+1/2}},
$$
where $M$ is as defined in (\ref{eqM}). Moreover, one also has the
estimate
$$
\forall 0\leq s\leq t_0+1/2,\qquad
\abs{\nabla (\Gmb)^{-1}\circ \Gpb\psi}_{H^s}\leq M\abs{\nabla\psi}_{H^s}.
$$
\end{proposition}
\begin{remark}\label{remGpGm}
It is straightforward to check that the proposition still holds if one switches the $+$ and $-$ signs
everywhere.
\end{remark}
\begin{remark}\label{remGder}
Let $A: H^{s+1/2}(\R^d)\mapsto H^{s-1/2}(\R^d)$ be a linear operator such that
$$
\forall 0\leq s\leq t_0+1, \qquad
(\Lambda^s A\psi_1,\Lambda^s\psi_2)\leq \mu M_A(\psi_1)\abs{\psi_2}_{\Hm^{s+1/2}},
$$
for some positive constant $M_A(\psi_1)$. Replacing (\ref{eqsectII1-6}) by this inequality everywhere in the proof of the Proposition shows that $(\Gmb)^{-1}\circ A:{H}^{s+1/2}\to\Hm^{s+1/2}$ is well defined and
$$
\forall 0\leq s\leq t_0+1, \qquad 
\abs{ (\Gmb)^{-1}\circ A\psi}_{\Hm^{s+1/2}}\leq M\,M_A(\psi).
$$
For instance, one has, for all $1\leq j\leq d$, 
$$
\forall 0\leq s\leq t_0+1, \qquad 
\abs{ (\Gmb)^{-1}\circ \partial_j\psi}_{\Hm^{s+1/2}}\leq \frac{1}{\mu}M \abs{(1+\sqrt{\mu}\abs{D})^{1/2}\psi}_{H^s}.
$$
\end{remark}
\begin{remark}
With a little more work, it is possible to show (using regularizing diffeomorphisms in the proof) that for all $s\geq 0$, the mapping $(\Gmb)^{-1}\circ\Gpb:\Hm^{s+1/2}\to \Hm^{s+1/2}$ is well defined and that
$$
	\abs{(\Gmb)^{-1}\circ \Gpb\psi}_{\Hm^{s+1/2}}\leq 
        M(s+1/2)
	\abs{\psi}_{\Hm^{s+1/2}},
$$
with $M(s+1/2)$ as in (\ref{eqMbis}), 
but since we do not need such a high order estimate here, we do not prove it.
\end{remark}
\begin{proof}
Let us prove that $(\Gmb)^{-1}\Gpb\psi$ is well defined in 
$\dot{H}^{1/2}(\R^d)$. We first prove that there exists a unique
variational solution $\Phi^-\in \dot{H}^{1}(\Omega^-)$ to the boundary value problem
\begin{equation}\label{eqsectII1-7}
\left\lbrace
\begin{array}{l}
\dsp (\dz^2+\mu^-\Delta)\Phi^-=0\quad\mbox{ in }\quad \Omega^-,\vspace{.5mm}\\
\dsp \sqrt{1+\abs{\eps^-\nabla\zeta}^2}\dn{\Phi^-}_{\vert_{z=\eps^-\zeta}}=\Gpb\psi,\qquad\dz{\Phi^-}_{\vert_{z= 1}}=0,
\end{array}\right.
\end{equation}
or equivalently (by \S \ref{sectdiff}), we seek a unique variational solution $\phi^-\in\dot{H}^1(\cS^-)$
to
\begin{equation}\label{eqsectII1-8}
\left\lbrace
\begin{array}{l}
\namm\cdot P(\Sigma^-)\namm\phi^-=0 \quad\mbox{ in }\quad \cS^-,\\
\dn {\phi^-}_{\vert_{z=0}}=\Gpb\psi,\qquad \dn{\phi^-}_{\vert_{z=1}}=0,
\end{array}\right.
\end{equation}
where $\Sigma^-$ is an admissible diffeomorphism in the sense of Definition \ref{defII1-1}
and satisfying (\ref{Psob}).\\
For all  $\varphi\in C^\infty(\overline{\cS^-})\cap \dot{H}^1(\cS^-)$, one deduces from
(\ref{eqsectII1-6}) that
$$
	(\Gpb\psi,\varphi_\interff)\leq
	M \mu \abs{\psi}_{\Hm^{1/2}}\abs{\varphi_\interff}_{\Hm^{1/2}}.
$$
Since moreover $\abs{\varphi_\interff}_{\Hm^{1/2}}\lesssim \mu^{-1/2}\Vert \namm\varphi\Vert_{L^2(\cS^-)}$ (this is a variant of the trace lemma, see the proof of Proposition 3.4 of \cite{AL}), one has
$$
(\Gpb\psi,\varphi_\interff)
	\leq M \sqrt{\mu}
	\abs{\psi}_{\Hm^{1/2}}\Vert \namm\varphi\Vert_{L^2(\cS^-)}.
$$
The linear form $\varphi\mapsto (\Gpb\psi,\varphi_\interff)$ is thus continuous on $\dot{H}^1(\cS^-)$; moreover, $\namm\cdot P(\Sigma^-)\namm$ is obviously coercive on this space, so that  existence/uniqueness of a variational solution to
(\ref{eqsectII1-8}) follows classically from Lax-Milgram theorem.\\
We can thus define $\psi^-={\phi^-}_\interff$ so that one obviously has $\Gmb\psi^-=\Gpb\psi$
and  $\psi^-=(\Gmb)^{-1}\circ\Gpb\psi$. Since such a solution is obviously unique,
the operator $(\Gmb)^{-1}\circ \Gpb$ is well defined.

We now turn to prove the estimates given in the statement of the proposition. Remarking that
$\abs{(\Gmb)^{-1}\circ\Gpb\psi}_{\Hm^{s+1/2}}=\abs{\psi^-}_{\Hm^{s+1/2}}$,
with $\psi^-={\phi^-}_\interff$ as constructed above, we deduce from (\ref{DNstar}) that
\begin{equation}\label{eqsectII1-9}
	\sqrt{\mu}\abs{(\Gmb)^{-1}\circ\Gpb\psi}_{{\Hm}^{s+1/2}}\leq M \Abs{\Lambda^s\namm\phi^-}_{L^2(\cS^-)}.
\end{equation}
In order to get an estimate on the r.h.s. of this inequality, let us apply $\Lambda^s$ to the l.h.s. of (\ref{eqsectII1-8}),
multiply it  by $\Lambda^s\phi^-$ and integrate by parts, to get
\begin{eqnarray*}
	\lefteqn{\int_{\cS^-}P(\Sigma^-)\namm\Lambda^s\phi^-\cdot\namm\Lambda^s\phi^-
	=-\int_{\R^d}\Lambda^s(\Gpb\psi)\Lambda^s\psi^-}\\
	&+& \int_{\cS^-}[\Lambda^s,P(\Sigma^-)]\namm\phi^-\cdot\namm\Lambda^s\phi^-.
\end{eqnarray*}
We can therefore deduce from the coercivity of $P(\Sigma^-)$, (\ref{eqsectII1-6}) and Cauchy-Schwarz 
inequality that
\begin{eqnarray*}
	k(\Sigma^-)\Abs{\Lambda^s\namm\phi^-}_{L^2(\cS^-)}^2&\leq&
	M\mu\abs{\psi}_{\Hm^{s+1/2}}\abs{\psi^-}_{\Hm^{s+1/2}}\\
	&+&
	\Abs{[\Lambda^s,P(\Sigma^-)]\namm\phi^-}_{L^2(\cS^-)}
	\Abs{\Lambda^s\namm\phi^-}_{L^2(\cS^-)}.
\end{eqnarray*}
Since $\sqrt{\mu}\abs{ \psi^-}_{\Hm^{s+1/2}}\leq M \Abs{\Lambda^s\namm\phi^-}_{L^2(\cS^-)}$
by (\ref{DNstar}), one can deduce, recalling that $k(\Sigma^-)$ satisfies (\ref{eqk}),
that
$$
	\Abs{\Lambda^s\namm\phi^-}_{L^2(\cS^-)}\leq
	M\sqrt{\mu}\abs{\psi}_{\Hm^{s+1/2}}
	+M
	\Abs{[\Lambda^s,P(\Sigma^-)]\namm\phi^-}_{L^2(\cS^-)}.
$$
Using the classical commutator estimate 
$$
\forall 0\leq s\leq t_0+1,\qquad  \abs{[\Lambda^s,f]g}_2\lesssim  \abs{f}_{H^{t_0+1}}\abs{g}_{H^{s-1}},
$$
it is easy to deduce from (\ref{Psob}) 
that 
\begin{equation}\label{nouv}
\Abs{[\Lambda^s,P(\Sigma^-)]\namm\phi^-}_{L^2(\cS^-)}\leq M \Abs{\Lambda^{s-1}\namm\phi^-}_2
\end{equation}
and therefore
$$
	\Abs{\Lambda^s\namm\phi^-}_{L^2(\cS^-)}\leq
	M\big(\sqrt{\mu}\abs{\psi}_{\Hm^{s+1/2}}+
	 \Abs{\Lambda^{s-1}\namm\phi^-}_2\big).
$$
The first estimate of the proposition follows therefore from a continuous induction on $s$ and (\ref{eqsectII1-9}). The second estimate is then a direct consequence of the following lemma.
\begin{lemma}\label{lemestim}
Let $A$ be a linear operator mapping $\Hm^{s+1/2}$ into itself for all $0\leq s\leq t_0+1$ and such that $\Abs{A}_{\Hm^{s+1/2}\to \Hm^{s+1/2}}\leq M$.
Then one has
$$
\forall 0\leq s\leq t_0+1/2,\qquad  \abs{\nabla (A\psi)}_{H^s}\leq M \abs{\nabla\psi}_{H^s}.
$$
\end{lemma}
\begin{proof}[Proof of the lemma]
From the definition of $\Pp$, one deduces that
\begin{eqnarray*}
\abs{\nabla(A\psi)}_{H^s}&=& \abs{(1+\sqrt{\mu}\abs{D})^{1/2}\Pp \psi}_{H^s}\\
&\leq& \abs{\Pp(A\psi)}_{H^s}+\mu^{1/4}\abs{\Pp (A\psi)}_{H^{s+1/2}}.
\end{eqnarray*}
Using the assumption made on $A$, we then get
\begin{eqnarray*}
\abs{\nabla(A\psi)}_{H^s}&\leq & M\big(\abs{\Pp \psi}_{H^s}+\mu^{1/4}\abs{\Pp \psi}_{H^{s+1/2}}\big)\\
&\leq& M\babs{\frac{1+\mu^{1/4}\abs{D}^{1/2}}{(1+\sqrt{\mu}\abs{D})^{1/2}} \abs{D}\psi}_{H^s},
\end{eqnarray*}
and the result follows  from the observation that $\frac{1+\mu^{1/4}\abs{D}^{1/2}}{(1+\sqrt{\mu}\abs{D})^{1/2}}$
is uniformly bounded (as a zero order operator) for $\mu\in (0,1)$.
\end{proof}
\end{proof}

\subsubsection{Shape derivatives}\label{sectSD}

It is known (e.g. \cite{LannesJAMS,Iguchi}) that for all $0\leq s\leq t_0+1$ and
$\psi^\pm\in \dot{H}^{s+1/2}(\R^d)$, the mapping
$$
\begin{array}{lcl}
H^{t_0+2}(\R^d)& \to & H^{s-1/2}(\R^d)\\
\zeta &\mapsto & \Gpmb\psi^\pm={\mathpzc G}^\pm_{\mu^\pm}[\eps^\pm\zeta,1]\psi^\pm
\end{array}
$$
is smooth in a neighborhood of any ${\zeta}\in H^{t_0+2}(\R^d)$ satisfying (\ref{sectII1}). Let us 
denote by $d^j\Gpmb(\bh)\psi^\pm$ ($j\in\N^*$, $\bh=(h_1,\dots,h_j)\in H^{t_0+2}(\R^d)^{j}$) its $j$-th derivative at $\zeta$
and in the direction $\bh=(h_1,\dots,h_j)$. Such derivatives are called \emph{shape derivatives}.

A first important result is that there exists an exact formula
for the first order shape derivative (Theorem 3.20 of \cite{LannesJAMS}),
\begin{equation}\label{formuleder}
d \Gpmb(h)\psi^\pm=-\eps\Gpmb\big(h \uwpm)-\eps^\pm\mu^\pm\nabla\cdot(h \uVpm) ,
\end{equation}
where 
$$
\uwpm=\frac{1}{1+\eps^2\mu\abs{\nabla\zeta}^2}\big(\frac{1}{\uH^\pm}\Gpmb\psi^\pm+\eps\mu\nabla\zeta\cdot\nabla\psi^\pm\big),
\qquad \uVpm=\nabla\psi^\pm-\eps \uwpm\nabla\zeta.
$$
\begin{remark}
Consistently with the notations used in the introduction, $\uVpm$ and $\uwpm$ stand for the horizontal
and vertical components of the (nondimensionalized) velocity fields in the two layers, evaluated at the interface.
\end{remark}

The formula (\ref{formuleder}) is not convenient to give control of the operator norm
of the shape derivatives of $\Gpmb$ because it looks more singular than expected
(in fact, the identity $\eps\Gpmb[\eps\zeta]\uwpm=-\eps^\pm\mu^\pm\nabla\cdot \uVpm$ \cite{BLS,ABZ} shows that the most singular terms cancel one each other).
Using direct methods, it is however possible to obtain some estimates. Let us first set some 
notation.
\begin{notation}\label{notah}
For all $\bh=(h_1,\dots,h_j)\in X^j$ ($X$ Banach space), and $1\leq l\leq j$, we write
$$
\langle{\bh}\rangle_X=\prod_{m=1}^j \abs{h_j}_X
\quad \mbox{ and }\quad
\langle{\check{\bh}_l}\rangle_X=\prod_{m\neq l}^j \abs{h_j}_X
$$
\end{notation}
The first estimates are given by  (see Proposition 3.3 of \cite{AL} or adapt the proof of (\ref{estder0bis}) given in the Appendix)
\begin{equation}
\label{estder0}
\forall 0\leq s \leq t_0+1,\quad \abs{d^j\Gpmb(\bh)\psi}_{H^{s-1/2}}
 \leq M \eps^j{\mu}^{3/4}\langle\bh\rangle_{H^{s\vee t_0+1}}\abs{\psi}_{\Hm^{s+1/2}},
\end{equation}
and (see Remark 3.3 of \cite{AL} or adapt the proof of 
(\ref{estder0bisk}) given in the Appendix)
\begin{equation}
\label{estder0k}
\forall 0\leq s\leq t_0+1/2,\quad\abs{d^j\Gpmb(\bh)\psi}_{H^{s-1/2}}
 \leq M \eps^j\mu\langle{\bh}\rangle_{H^{(s+1/2)\vee t_0+1}}\abs{\psi}_{\Hm^{s+1}}.
\end{equation}
For small values of $s$, (\ref{estder0}) and (\ref{estder0k}) require much more regularity
on the $h_j$ than on $\psi$. 
This is the reason why we also need the following two estimates,
where the $s$-depending norm in the r.h.s. is on one of the $h_j$ rather than on $\psi^\pm$ (Appendix \ref{apestpouet} for a proof): for all $1\leq l\leq j$,
\begin{equation}
 \abs{d^j\Gpmb(\bh)\psi}_{H^{s-1/2}}
\leq M\eps^j{\mu}^{3/4}\abs{h_l}_{H^{s+1/2}}
\label{estder0bis}
\langle \check{\bh}_l\rangle_{H^{s\vee t_0+3/2}}\abs{\psi}_{\Hm^{s\vee t_0+1}},
\end{equation}
for all $0\leq s\leq t_0+1/2$, and
\begin{equation}
\abs{d^j\Gpmb(\bh)\psi}_{H^{s-1/2}}
 \leq M\eps^j{\mu}\abs{h_l}_{H^{s+1}}
\label{estder0bisk}
\langle \check{\bh}_l\rangle_{H^{(s+1/2)\vee t_0+3/2}}\abs{\psi}_{\Hm^{(s+1/2)\vee t_0+1}},
\end{equation}
 for all $0\leq s\leq t_0$.
\begin{remark}
Formulas (\ref{estder0}) and (\ref{estder0bis}) require less regularity on $\psi$ or $\bh$
than (\ref{estder0k}) and (\ref{estder0bisk}) respectively, but they
give a control of size $O({\mu}^{3/4})$ while (\ref{estder0k}) and (\ref{estder0bisk}) give a $O(\mu)$ control. This is reminiscent of the comments made in Remark \ref{remsecond}.
\end{remark}

 In the same spirit, is also possible to give two generalizations of 
(\ref{eqsectII1-6}) to shape derivatives (see Appendix \ref{apestpouet} for a proof). One has
\begin{equation}
 \babs{(\Lambda^sd^j\Gpmb(\bh)\psi_1,\Lambda^s\psi_2)}\leq
M \eps^j {\mu}\abs{\psi_1}_{\Hm^{s+1/2}}\abs{\psi_2}_{\Hm^{s+1/2}}
\label{eqsectII1-6gen}
\langle \bh\rangle_{H^{s\vee t_0+1}},
\end{equation}
for all $0\leq s\leq t_0+1$, and, for $1\leq l\leq j$,
\begin{equation}
\nonumber
\babs{(\Lambda^sd^j\Gpmb(\bh)\psi_1,\Lambda^s\psi_2)}\leq
M \eps^j\mu\abs{h_l}_{H^{s+1/2}}\langle \check{\bh}_l\rangle_{H^{s\vee t_0+3/2}}
\label{eqsectII1-6genbis}
 \abs{\psi_1}_{\Hm^{s\vee t_0+1}}\abs{\psi_2}_{\Hm^{s+1/2}},
\end{equation}
for all $0\leq s\leq t_0+1/2$.
\begin{remark}\label{remcons}
A consequence of these estimates is that one can use Remark \ref{remGder} to prove that  $(\Gmb)^{-1}\circ d^j\Gpb(\bh)$ is well defined and give some estimate on its operator norm.
For instance, one gets with (\ref{eqsectII1-6genbis}) 
$$
	\abs{(\Gmb)^{-1} d^j\Gpb(\bh)\psi}_{\Hm^{s+1/2}}\leq \eps^j M \abs{h_l}_{H^{s+1/2}}
        \abs{\psi}_{\Hm^{s\vee t_0+1}}\langle \check{\bh}_l\rangle_{H^{s\vee t_0+3/2}}.
$$
\end{remark}

\subsection{A transmission problem}\label{secttransm}

Attention is given here to the following transmission problem, whose resolution ensures that the velocity potentials in both fluid layers can be recovered from the knowledge of $\zeta$ (i.e. the shape of the interface) and $\psi$.
$$
\left\lbrace
\begin{array}{l}
\dsp (\dz^2+\mu\Delta)\Phi^+=0\qquad \mbox{ for } -\uH^+<z<\eps\zeta,\\
\dsp (\dz^2+\mu\Delta)\Phi^-=0\qquad \mbox{ for  } \eps\zeta< z<\uH^-,\\
\dsp \big(\urp{\Phi^+}-\urm{\Phi^-}\big)_{\vert_{z=\eps\zeta}}=\psi,\\
\dsp \dn {\Phi^+}_{\vert_{z=\eps\zeta}}-{\dn \Phi^-}_{\vert_{z=\eps\zeta}}=0, \qquad \dz {\Phi^\pm}_{\vert_{z=\mp\uH^\pm}}=0.
\end{array}\right.
$$
From the discussion in \S \ref{sectdiff}, we can equivalently investigate the following straightened version,
\begin{equation}\label{eqsectII2-1}
\left\lbrace
\begin{array}{l}
\dsp \nampm\cdot P(\Sigma^\pm)\nampm\phi^\pm=0\qquad \mbox{ in }\cS^\pm,\\
\dsp \urp{\phi^+}_\interff-\urm{\phi^-}_\interff=\psi,\\
\dsp \frac{1}{\uH^+}\dn {\phi^+}_\interff-\frac{1}{\uH^-}{\dn \phi^-}_\interff=0, \qquad \dn {\phi^\pm}_{\vert_{z=\mp 1}}=0,
\end{array}\right.
\end{equation}
where $\dn$ denotes as usual the upwards conormal derivative (see (\ref{conormal})).\\
Transmission problems of this kind are usually studied with harmonic analysis tools and 
the solution is then 
given in terms of the single and double layer potentials (see for instance \cite{EM}).
Such an approach gives very sharp results in terms of regularity but is not adapted to our
problem (influence of the boundaries, asymptotics, ...). We thus propose an another approach, more
elementary, but more robust with respect to the applications we have in mind.
\begin{proposition}\label{prop2}
  Let $t_0>d/2$, $0\leq s\leq t_0+1$ and $\zeta\in H^{t_0+2}(\R^d)$ be such that (\ref{sectII1}) is satisfied. Then for all $\psi\in \dot{H}^{s+1/2}(\R^d)$, there exists a unique solution $\phi^\pm\in \dot{H}^{1}(\cS^\pm)$ to (\ref{eqsectII2-1}) such that $\Lambda^s\nampm\phi^\pm\in L^2(\cS^\pm)$ 
and
$$
\Abs{\Lambda^s\nampm \phi^\pm}_{L^2(\cS^\pm)}\leq M \sqrt{\mu}\abs{\psi}_{\Hm^{s+1/2}}.
$$
\end{proposition}
\begin{remark}
As with Proposition \ref{prop1} it is possible to extend the result to all $s\geq 0$ in Proposition
\ref{prop2} above, and Lemma \ref{lemma1} and Corollaries \ref{coro1} and \ref{coro2} below. The estimates given in 
the statement of these results still hold for all $s\geq 0$ provided that $M$ is replaced by
$M(s+1/2)$ in the right-hand-side.
\end{remark}
\begin{proof}
If the result of the proposition holds true, then, denoting $\psi^\pm={\phi^\pm}_{\vert_{z=0}}$, one has
$$
\urp\psi^+-\urm\psi^-=\psi
\quad\mbox{ and }\quad
\frac{1}{\uH^+}\Gpb\psi+=\frac{1}{\uH^-}\Gmb\psi^-,
$$
and thus $\psi^-=\frac{H^-}{H^+}(\Gmb)^{-1}\Gpb\psi^+$ (well defined in $\dot{H}^{s+1/2}$ by Proposition \ref{prop1}) and
$$
\urp\psi^+-\urm\frac{H^-}{H^+}(\Gmb)^{-1}\Gpb\psi^+=\psi.
$$
We now need the following lemma.
\begin{lemma}\label{lemma1}
For all $0\leq s\leq t_0+1$, the mapping 
$$
J[\zeta]:
\begin{array}{lcl}
\Hm^{s+1/2}(\R^d) &\to &\Hm^{s+1/2}(\R^d)\\
\psi&\mapsto& \big(\urp-\urm\frac{H^-}{H^+}(\Gmb)^{-1}\Gpb\big)\psi
\end{array}
$$
is one-to-one and onto; moreover, one has
$$
\abs{J[\zeta]^{-1}\psi}_{\Hm^{s+1/2}}\leq M\abs{\psi}_{\Hm^{s+1/2}}.
$$
\end{lemma}
\begin{proof}
Let us first consider the case $s=0$. 
For all $\psi^+\in \Hm^{1/2}(\R^d)$, let $\psi^-=\frac{H^-}{H^+}(\Gmb)^{-1}\circ\Gpb\psi^+\in\Hm^{1/2}(\R^d)$
(as provided by Proposition \ref{prop1}), and let
  $\phi^\pm\in\dot{H}^1(\cS^\pm)$ be the unique solution of (\ref{eqsectII1-3}).
Let us now multiply (\ref{eqsectII1-5})$_+$
by $\frac{1}{\uH^+}\urp$, (\ref{eqsectII1-5})$_-$ by $\frac{1}{\uH^-}\urm$ 
and sum up the two identities to get (with $P^\pm=P(\Sigma^\pm)$)
\begin{eqnarray}
\nonumber
\frac{\urp}{\uH^+}\int_{\cS^+}P^+\namp\phi^+\cdot\namp\phi^+&+&
\frac{\urm}{\uH^-}\int_{\cS^-}P^-\namm\phi^-\cdot\namm\phi^-\\
\label{eqnex}
&=&\int_{\R^d}\frac{1}{\uH^+}(\Gpb\psi^+)J[\zeta]\psi^+.
\end{eqnarray}

Using the coercivity of $P^\pm$ (with coercivity constants $k^\pm=k(\Sigma^\pm)$ as in (\ref{eqk}))
to get a lower bound on the l.h.s., and  (\ref{eqsectII1-6}) to derive an upper
bound for the r.h.s., one obtains
$$
\frac{\urp}{\uH^+} k^+\Abs{\namp\phi^+}^2_{L^2(\cS^+)}
+\frac{\urm}{\uH^-} k^-\Abs{\namm\phi^-}^2_{L^2(\cS^-)}
\leq \frac{M}{\uH^+}\mu
\abs{\psi^+}_{\Hm^{1/2}}\abs{J[\zeta]\psi^+}_{\Hm^{1/2}}.
$$
Now, since $\psi^+=(\Gpb)^{-1}\circ\Gmb \psi^-$, we can deduce from Remark \ref{remGpGm} that
$\abs{\psi^+}_{\Hm^{1/2}}\leq M\abs{\psi^-}_{\Hm^{1/2}}$; using (\ref{eqk}) and (\ref{DNstar})
we then deduce
$$
\abs{\psi^+}_{\Hm^{1/2}}
\leq M
\abs{J[\zeta]\psi^+}_{\Hm^{1/2}}.
$$
This tells us that for all $0\leq \urp\leq 1$ and $\urm=1-\urp$, $J[\zeta]$ is a closed, one-to-one operator; it is thus semi-Fredholm. 
Since moreover $J[\zeta]$ is clearly invertible (by a Neumann series expansion) for small enough values of $\urm$, we deduce from
the homotopic invariance of the index that it is Fredholm of index zero, and thus invertible (since it is one-to-one).
We now turn to prove the Lemma for the general case $0\leq s\leq t_0+1$.  It is easy to deduce from (\ref{eqsectII1-3}) the following generalization of (\ref{eqsectII1-5}),
\begin{eqnarray*}
\int_{\cS^\pm}P^\pm\nampm \Lambda^s\phi^\pm\cdot\nampm\Lambda^s\phi^\pm
&=&\pm\int_{\R^d}(\Lambda^s\Gpmb\psi^\pm)\Lambda^s\psi^\pm\\
&+&\int_{\cS^\pm}[\Lambda^s,P^\pm]\nampm\phi^\pm\cdot\Lambda^s\nampm\phi^\pm.
\end{eqnarray*}
Proceeding as for the case $s=0$, one gets 
\begin{eqnarray*}
\lefteqn{\urp\Abs{\Lambda^s\namp\phi^+}_{L^2(\cS^+)}+
{\urm}\Abs{\Lambda^s\namm\phi^-}_{L^2(\cS^-)}
\leq {M}\sqrt{\mu}
\abs{J[\zeta]\psi^+}_{\Hm^{s+1/2}}}\\
& &+M\Big({\urp}\Abs{[\Lambda^s,P^+]\namp\phi^+}_{L^2(\cS^+)}
+\urm\Abs{[\Lambda^s,P^-]\namm\phi^-}_{L^2(\cS^-)}\Big).
\end{eqnarray*}
Using (\ref{nouv}) and a continuous induction as in the proof of Proposition \ref{prop1}, we 
 then get
$$
{\urp}\Abs{\Lambda^s\namp\phi^+}_{L^2(\cS^+)}+
{\urm}\Abs{\Lambda^s\namm\phi^-}_{L^2(\cS^-)}
\leq {M}\sqrt{\mu}
\abs{J[\zeta]\psi^+}_{\Hm^{s+1/2}},
$$
and the result follows from (\ref{DNstar}) as in the case $s=0$.
\end{proof}
Thanks to the lemma (and Proposition \ref{prop1}), 
one can define $\psi^+,\psi^-\in\Hm^{s+1/2}$ by
\begin{equation}\label{manif}
\psi^+=J[\zeta]^{-1}\psi,\qquad \psi^-=\frac{H^-}{H^+}(\Gmb)^{-1}\circ\Gpb\psi^+.
\end{equation}
Taking $\phi^\pm$ as the solution to (\ref{eqsectII1-3}) with $\psi^\pm$ as above concludes the proof.
\end{proof}

As a first corollary to Proposition \ref{prop2}, we can prove that the quantities that appear in 
the equations (\ref{eqI-16nd}) are well defined.
\begin{corollary}\label{coro1}
Under the assumptions of Proposition \ref{prop2}, and denoting by $\psi^\pm$ the trace of $\phi^\pm$ at the interface $\{z=0\}$, the mappings
$$
	\Upm:
	\begin{array}{lcl}
	\dot{H}^{s+1/2}(\R^d)&\to & H^{s-1/2}(\R^d)\\
	\psi&\mapsto&\nabla \psi^\pm
	\end{array}
$$
and
$$
	\wpm:
	\begin{array}{lcl}
	\dot{H}^{s+1/2}(\R^d)&\to & H^{s-1/2}(\R^d)\vspace{1mm}\\
	\psi&\mapsto&\dsp \frac{1}{1+\eps^2\mu\abs{\nabla\zeta}^2} 
        \big(\frac{1}{\uH^\pm}\Gpmb\psi^\pm+\eps\mu \nabla\zeta\cdot\nabla\psi^\pm\big)
	\end{array}
$$
are well defined and one has, 
\begin{eqnarray*}
\forall 0\leq s\leq t_0+1/2,& & \abs{\Upm\psi}_{H^s}\leq M \abs{\nabla\psi}_{H^s},\\
\forall 0\leq s\leq t_0+1,& & \abs{\wpm\psi}_{H^{s-1/2}}\leq \mu^{3/4}M\abs{\psi}_{\Hm^{s+1/2}},\\
\forall 0\leq s\leq t_0+1/2,& & \abs{\wpm\psi}_{H^{s-1/2}}\leq \mu M\abs{\psi}_{\Hm^{s+1}}.
\end{eqnarray*}
\end{corollary}
\begin{proof}
This is an immediate consequence of (\ref{manif}), Lemma \ref{lemma1}, Proposition \ref{prop1} and Lemma \ref{lemestim} (and of (\ref{eqsectII1-4sam})-(\ref{eqsectII1-4sambis}) for 
the estimate on $\wpm\psi$).
\end{proof}

A second important corollary to Proposition \ref{prop2} concerns the operator $\Gmu$ defined in (\ref{14nd}).
\begin{corollary}\label{coro2}
Under the assumptions of Proposition \ref{prop2}, the mapping
$$
\Gmu:
\begin{array}{lcl}
\dot{H}^{s+1/2}(\R^d)&\to& {H}^{s-1/2}(\R^d)\\
\psi&\mapsto &\Gmb\big(\urp\uH^+ \Gmb-\urm \uH^-\Gpb\big)^{-1}\Gpb\psi
\end{array}
$$
is well defined 
for all $0\leq s\leq t_0+1/2$ and one has
$$
\forall \psi\in\dot{H}^{s+1/2}(\R^d),\qquad
\abs{\Gmu\psi}_{H^{s}}\leq \sqrt{\mu} M\abs{\nabla\psi}_{H^{s}}.
$$
\end{corollary}
\begin{proof}
Let us check first that $(\rho^+\uH^+\Gmb-\uH^-\rho^-\Gpb)^{-1}\circ \Gpb$ is well defined. 
We can deduce from the positivity of $(\rho^+\uH^+\Gmb-\rho^-\uH^-\Gpb)$ and the identity
$$
(\rho^+\uH^+\Gmb-\rho^-\uH^-\Gpb)\circ(\Gmb)^{-1}\circ\Gpb \circ J[\zeta]^{-1}=\uH^+\Gpb
$$
that  $(\rho^+\uH^+\Gmb-\rho^-\uH^-\Gpb)^{-1}\circ \Gpb$ is well defined and given by the formula
$$
(\rho^+\uH^+\Gmb-\rho^-\uH^-\Gpb)^{-1}\circ \Gpb=\frac{1}{\uH^+}(\Gmb)^{-1}\circ \Gpb\circ  J[\zeta]^{-1}.
$$
The estimate given in the statement of the corollary is thus a direct consequence of 
 (\ref{eqsectII1-4sam}) and Lemmas \ref{lemma1} and \ref{lemestim}. 
\end{proof}
\begin{remark}\label{remequiv}
Proceeding as in the proof of Corollary \ref{coro2}, one can derive several equivalent
expressions for $\Gb_\mu[\eps\zeta]$, for instance,
\begin{eqnarray*}
\Gmu&=&\frac{1}{\uH^+}\Gpb\circ J[\zeta]^{-1}\\
\urp\Gmu&=&\frac{1}{\uH^+}\big(1+\urm \uH^-\Gmu\circ(\Gmb)^{-1}\big)\Gpb,
\end{eqnarray*}
\end{remark}
\begin{remark}\label{remGNlin}
In the undisturbed case ($\zeta=0$), it is easy to check by Fourier analysis that
$\Gpmb=\pm{\mathpzc G}^\pm_{\mu^\pm}[0,1]=\pm\sqrt{\mu^\pm}\abs{D}\tanh(\sqrt{\mu^\pm}\abs{D})$. It follows therefore that
$$
{\mathpzc G}_\mu[0]=\sqrt{\mu}\abs{D}\frac{\tanh(\sqrt{\mu^+}\abs{D})\tanh(\sqrt{\mu^-}\abs{D})}{\urp\tanh(\sqrt{\mu^-}\abs{D})+\urm\tanh(\sqrt{\mu^+}\abs{D})}.
$$
\end{remark}

Some important features of the operator $\Gb_\mu[\eps\zeta]$ are gathered in the following proposition.
\begin{proposition}\label{propG}
Let $t_0>d/2$ and
$\zeta\in H^{t_0+2}(\R^d)$ satisfying
(\ref{sectII1}). Then
\item[(1)] The operator $\frac{1}{\mu}\Gb_\mu[\eps\zeta]$ is positive and uniformly coercive on $\Hm^{1/2}(\R^d)$,
$$
\forall \psi\in \dot{H}^{1/2}(\R^d),\qquad
\abs{\Pp \psi}_2^2\leq M\big(\psi,\frac{1}{\mu}\Gb_\mu[\eps\zeta] \psi\big).
$$
\item[(2)] The bilinear form $(\frac{1}{\mu}\Gb_\mu[\eps\zeta]\cdot,\cdot)$ is symmetric on 
$\Hm^{1/2}\times\Hm^{1/2}$,
$$
\forall \psi_1,\psi_2\in \dot{H}^{1/2}(\R^d),\qquad
 (\Gb_\mu[\eps\zeta] \psi_1,\psi_2)=(\Gb_\mu[\eps\zeta] \psi_2,\psi_1).
$$
\item[(3)] For all $0\leq s\leq t_0+1$ and $\psi_1,\psi_2\in \Hm^{s+1/2}$, one has
$$
 \babs{(\Lambda^s\Gb_\mu[\eps\zeta]\psi_1,\Lambda^s\psi_2)}\leq 
\mu M  \abs{\psi_1}_{\Hm^{s+1/2}}\abs{\psi_2}_{\Hm^{s+1/2}}.
$$
\end{proposition}
\begin{proof}
Thanks to Remark \ref{remequiv}, one has $(\Gb_\mu[\eps\zeta] \psi, \psi)=\frac{1}{\uH}(\Gpb\psi^+,J[\zeta]\psi^+)$, with $\psi^+=J[\zeta]^{-1}\psi$. It follows therefore from (\ref{eqnex}) and
(\ref{DNstar}) that one has
$\mu \abs{\Pp \psi^+}\leq M (\Gb_\mu[\eps\zeta]\psi,\psi)$. Since moreover $\abs{\Pp \psi}\leq M\abs{\Pp \psi^+}$ by Proposition \ref{prop1}, the first point of the proposition is proved.\\
To prove the second point, we proceed as for (\ref{eqnex}) to get
$$
(\Gb_\mu[\eps\zeta] \psi_1,\psi_2)=\frac{\urp}{\uH^+}\int_{\cS^+}P^+\namp\phi_1^+\cdot\namp\phi_2^+
+\frac{\urm}{\uH^-}\int_{\cS^-}P^-\namm\phi_1^-\cdot\namm\phi_2^-,
$$
with obvious notations for $\phi_j^\pm$ ($j=1,2$). The symmetry of $\Gb$ follows from the
symmetry of $P^\pm$.\\
The last point of the proposition follows from the first equivalent expression
of $\Gb_\mu[\eps\zeta]$ given in Remark \ref{remequiv}, 
together with (\ref{eqsectII1-6}) and Lemma \ref{lemma1}. 
\end{proof}

\section{Symbolic analysis ``with tail'' of the Dirichlet-Neumann operator and consequences}\label{sectsymbolic}

\subsection{Symbolic analysis of $\Gpmb$}

It is known that the principal symbol of the Dirichlet-Neumann operator can be expressed in terms of the Laplace-Beltrami operator associated to the surface. More precisely, if $\Gp$ denotes the Dirichlet-Neumann on the original (without nondimensionalization) problem (\ref{eqI-10}), one has typically
\begin{equation}\label{formsymb}
\babs{\Gpb[\zeta]\psi-\sqrt{\abs{D}^2+(\abs{D}^2\abs{\nabla\zeta}^2-(D\cdot\nabla\zeta)^2)}\psi}_{H^{s+1/2}}\leq C(s,\zeta)\abs{\nabla\psi}_{H^{s-1/2}},
\end{equation}
where $C(s,\zeta)$ depends on the norm of $\zeta$ in some Sobolev space depending on $s$.
We refer for instance to \cite{Taylor,LannesJAMS} for a proof with PDE tools, to \cite{ShatahZeng} for a more geometric approach, and to \cite{AlazardMetivier} for an interesting paradifferential extension.\\
The question that interests us here is the way this identity behaves asymptotically in the shallow water regime. Since this regime physically corresponds to configurations where the effect of the bottom is ``felt'' at the surface, it is possible to guess without computation that (\ref{formsymb}) becomes singular in shallow water; indeed, the influence of the bottom on  $\Gp$ is analytic (by standard elliptic theory) and thus not taken into account in the symbolic expansion of $\Gp\psi$ (even at the next orders).\\
In order to make more precise this singular behavior, let us look at the form taken by
(\ref{formsymb}) in the nondimensionalized setting, and in the case $\zeta=0$. We recall that in this case, one has ${\mathpzc G}^\pm_\mu[0]=\pm\sqrt{\mu^\pm}\abs{D}\tanh(\sqrt{\mu^\pm}\abs{D})$, whose principal symbol is $\pm\sqrt{\mu^\pm}\abs{D}$. It is thus straightforward to check that
\begin{equation}\label{formsymb2}
\abs{{\mathpzc G}^\pm_\mu[0]\psi\mp\sqrt{\mu^\pm}\abs{D}\psi}_{H^{s+1/2}}\lesssim \abs{\nabla\psi}_{H^{s-1/2}}.
\end{equation}
Since ${\mathpzc G}^\pm_\mu[0]$ and $\sqrt{\mu^\pm}\abs{D}$ are both of size
$O(\sqrt{\mu^\pm})$ (as first order operators) while the residual is of size $O(1)$, the singularity of the
symbolic approximation of the DN operator is of order $O(1/\sqrt{\mu^\pm})$. This is the reason why one cannot use standard symbolic analysis in a shallow water regime (though this is of course legitimate in other regimes, as in \cite{ShatahZeng2,MingZhang,ABZ}).\\
Our purpose here is to show that it is however possible to adapt the symbolic analysis approach to the present case. As said previously, it is necessary to take into account the information coming from the bottom, that is, to include the smooth ``tail'' of the symbol that is neglected in any expansion into homogeneous symbols. In the flat case, this ``tail'' corresponds to $\big(1-\tanh(\sqrt{\mu^\pm}\abs{D})\big)$. In order to describe the general case, it is necessary to define the usual principal symbol $g(X,\xi)$ and another symbol $t^\pm(X,\xi)$ associated to its tail; more precisely, we define
\begin{equation}\label{defSpm}
S^\pm(X,\xi)=\sqrt{\mu^\pm}g(X,\xi)\tanh(\sqrt{\mu^\pm}t^\pm(X,\xi)),
\end{equation}
with
\begin{eqnarray*}
g(x,\xi)&=&\sqrt{\abs{\xi}^2+\eps^2\mu(\abs{\nabla\zeta}^2\abs{\xi}^2-(\nabla\zeta\cdot\xi)^2)},\\
t^\pm(x,\xi)&=&(1\pm\eps^\pm\zeta)\int_{-1}^0\frac{\sqrt{\abs{\xi}^2+\eps^2\mu(z+1)^2(\abs{\nabla\zeta}^2\abs{\xi}^2-(\nabla\zeta\cdot\xi)^2)}}{1+\eps^2\mu(z+1)^2\abs{\nabla\zeta}^2}dz.
\end{eqnarray*}
We can then state the main result of this section.
\begin{theorem}\label{theotail}
Let $t_0>d/2$ and $\zeta\in H^{t_0+3}(\R^d)$ be such that (\ref{sectII1}) is satisfied.
Then for all $0\leq s\leq t_0$ and $\psi\in \dot{H}^{s+1/2}(\R^d)$, one has
$$
\babs{\Gpmb\psi\mp S^\pm(x,D)\psi}_{H^{s+1/2}}\leq \eps\mu^{3/4}M(t_0+3)\abs{\psi}_{\Hm^{s+1/2}}.
$$
The following inequality also holds
$$
\babs{\Gpmb\psi\mp S^\pm(x,D)\psi}_{H^s}\leq \eps\mu M(t_0+3) \abs{\psi}_{\Hm^{s+1/2}}.
$$
\end{theorem}
\begin{remark}
In the case $d=1$, one has $g(x,\xi)=\abs{\xi}$ and one can compute explicitly
$\dsp t^\pm(x,\xi)=(1\pm\eps^\pm\zeta)\frac{\arctan(\eps\sqrt{\mu}\partial_x\zeta)}{\eps\sqrt{\mu}\partial_x\zeta}\abs{\xi}$.
\end{remark}
\begin{remark}
One readily deduces from the theorem the following estimate for the standard symbolic approximation,
$$
\abs{\Gpmb\psi\mp\sqrt{\mu^\pm}g(x,D)\psi}_{H^{s+1/2}}\leq M(t_0+3)\abs{\nabla\psi}_{H^{s-1/2}},
$$
which is the nondimensionalized version of (\ref{formsymb}) generalizing (\ref{formsymb2}) to the nonflat case. The interest of the new symbolic expansion of Theorem \ref{theotail} is that it gives a control of order $O(\eps\sqrt{\mu})$
of $\abs{\nabla\psi}_{H^{s-1/2}}$ (deduced from the $O(\eps\mu^{3/4})$
control of $\abs{\psi}_{\Hm^{s+1/2}}$); it is therefore more precise in its dependence on $\mu$ \emph{and} $\eps$. Even outside the shallow water regime (ie if $\mu=O(1)$), the estimate of the theorem is more precise than the standard expansion if $\eps$ is small. In particular, the formula is exact in the flat case ($\eps=0$). Note also that the improvement factor, namely $\eps\sqrt{\mu}=\eps^\pm\sqrt{\mu^\pm}$, is the typical \emph{slope} of the wave, a quantity that plays an important role in many asymptotic expansions. In the second estimate of the theorem, we have a better control of order $O(\eps\mu)$ but in terms of regularity, it is half-a-derivative worse ($\abs{\psi}_{\Hm^{s+1/2}}$ instead of $\abs{\psi}_{\Hm^{s}}$).
\end{remark}
\subsection{Proof of Theorem \ref{theotail}}

Since the proof of the result for $\Gmb$ follows exactly the same lines as for $\Gpb$, we only do it in the latter case.\\
We also assume throughout this proof that $\Sigma(X,z)=(X,z+\sigma(X,z))$ is the trivial diffeomorphism given in Example \ref{exII1-1}; in particular, $1+\dz\sigma=1+\eps^+\zeta$ does not depend on $z$.

Let us denote by $\bP=\bP(X,z,\nabla,\dz)$ the elliptic operator
$$
\bP=\namp\cdot P(\Sigma)\namp.
$$
From the explicit expression of $P(\Sigma)$ given in (\ref{expP}), we can decompose
$\bP$ into
$$
\bP=\bP_I+\eps^+\mu^+\bP_{II},
$$
with
\begin{eqnarray*}
\bP_I&=&\frac{1+\mu^+\abs{\nabla\sigma}^2}{1+\dz\sigma}\dz^2-2\mu^+\nabla\sigma\cdot\nabla\dz
+\mu^+(1+\dz\sigma)\Delta,\\
\bP_{II}&=&\frac{1}{\eps^+}\big(\frac{1}{1+\dz\sigma}\dz(\abs{\nabla\sigma}^2)-\Delta\sigma\big)\dz.
\end{eqnarray*}
The strategy to prove Theorem \ref{theotail} is the following. Since by definition 
$$
\Gpb\psi=\dn\phi_{\vert_{z=0}},\quad \mbox{ with }\quad
\left\lbrace\begin{array}{l}
\bP \phi=0,\\
\phi_{\vert_{z=0}}=\psi,\qquad \dn\phi_{\vert_{z=-1}}=0
\end{array}\right.
$$
(recall that $\dn$ stands for the upward conormal derivative, see (\ref{conormal})),
we expect that at leading order, $\Gpb\psi\sim \dn\phi_{app}\,_{\vert_{z=0}}$ if $\phi_{app}$
solves the same boundary value problem as $\phi$ up to lower order terms. This is the approach used in \cite{LannesJAMS,AlazardMetivier} and which leads to the standard symbolic analysis of $\Gp\psi$. The difference here is that we want the approximation $\phi_{app}$ to be nonsingular with respect to $\mu^+$. This is this extra constraint which imposes the presence of the ``tail'' in the symbolic analysis of the DN operator.\\
Since $\bP=\bP_I$ up to first order terms that are also of size $O(\mu^+)$,  
the idea is to look for an explicit function $\phi_{app}$ satisfying $\bP_I\phi_{app}=0$
(up to lower order terms). In order to construct such a $\phi_{app}$, we first remark that 
if we freeze the coefficients of $\bP_I$
and take the Fourier transform with respect to the horizontal variables, 
$\bP_I$ becomes a second order differential equation with respect to $z$, with
characteristic polynomial
$$
\frac{1+\mu^+\abs{\nabla\sigma}^2}{1+\dz\sigma}\eta^2-2i\mu^+\nabla\sigma\cdot\xi \eta-\mu^+ (1+\dz\sigma)\abs{\xi}^2.
$$
The roots of this polynomial are given by
$$
\eta^\pm(X,z,\xi)=\pm\sqrt{\mu^+} a(X,z,\xi)+i\mu^+ b(X,z,\xi),
$$
with
\begin{eqnarray*}
a(X,z,\xi)&=&\frac{1+\dz\sigma}{1+\mu^+\abs{\nabla\sigma}^2}\sqrt{\abs{\xi}^2+\mu^+(\abs{\nabla\sigma}^2\abs{\xi}^2-(\nabla\sigma\cdot\xi)^2)},\\
b(X,z,\xi)&=&\frac{1+\dz\sigma}{1+\mu^+\abs{\nabla\sigma}^2}\nabla\sigma\cdot\xi.
\end{eqnarray*}
We then define $\phi_{app}$ as the inverse Fourier transform of the 
solution of this frozen coefficient differential operator satisfying the boundary conditions at $z=0$ and $z=-1$, namely, 
$$
\phi_{app}(X,z)=\Xi(X,z,D)\psi,
$$
where the symbol of the pseudodifferential operator $\Xi(X,z,D)$ is given by
$$
\Xi(X,z,\xi)=\frac{\cosh(\sqrt{\mu^+}\int_{-1}^z a(X,z',\xi)dz')}{\cosh(\sqrt{\mu^+}\int_{-1}^0 a(X,z',\xi)dz')}\exp(-i\mu^+\int_z^0 b(X,z',\xi)dz').
$$
The following lemma quantifies the accuracy of the approximation of $\phi$ by  $\phi_{app}$. 
\begin{lemma}\label{lemreste}
The approximation $\phi_{app}$ of $\phi$ solves
$$
\left\lbrace
\begin{array}{l}
\bP\phi_{app}=\eps^+\mu^+ h,\\
\phi_{app}\,_{\vert_{z=0}}=\psi,\qquad \dn \phi_{app}\,_{\vert_{z=-1}}=0,
\end{array}\right.
$$
where, for all $0\leq s\leq t_0$, $h$ satisfies
$$
\Abs{\Lambda^s h}_2\leq M(t_0+3)\abs{\psi}_{\Hm^{s+1/2}}.
$$
\end{lemma}
\begin{proof}
One can decompose $\bP\phi_{app}$ under the form
\begin{equation}\label{tailm}
\bP\phi_{app}=\bP_I\phi_{app}+\eps^+\mu^+\bP_{II}\phi_{app};
\end{equation}
we now analyze the two components of the right-hand-side separately.\\
- \emph{Analysis of $\bP_I\phi_{app}$}.
From the relations
\begin{eqnarray*}
\nabla\sigma\cdot \nabla \dz \phi_{app}&=&\nabla\sigma\cdot\nabla \Op(\dz \Xi)\psi\\
&=&\Op(\nabla\sigma\cdot \nabla\dz\Xi)\psi+\Op(i(\nabla\sigma\cdot\xi)\dz\Xi)\psi,\\
(1+\dz\sigma)\Delta \phi_{app}&=&-\Op\big((1+\dz\sigma)\Xi \abs{\xi}^2\big)\psi+(1+\dz\sigma)[\Delta,\Op(\Xi)]\psi,
\end{eqnarray*}
we deduce that
$$
{\bf P}_I\phi_{app}=\Op\big(\bP_I(X,z,i\xi,\dz)\Xi\big)\psi-2\mu\Op(\nabla\sigma\cdot\nabla\dz\Xi\big)\psi+\mu(1+\dz\sigma)[\Delta,\Op(\Xi)]\psi.
$$
From the definition of $\Xi$, it is also easy to check that
\begin{eqnarray*}
\bP_I(X,z,i\xi,\dz)\Xi&=&\eps^+\mu^+\frac{1+\mu^+\abs{\nabla\sigma}^2}{1+\dz\sigma}  \exp(-i\mu^+\int_z^0 b)\\
&\times&\Big(\frac{\dz a}{\eps^+\sqrt{\mu^+}} \frac{\sinh(\sqrt{\mu^+}\int_{-1}^z a )}{\cosh(\sqrt{\mu^+}\int_{-1}^0 a )}+i\frac{\dz b}{\eps^+} \frac{\cosh(\sqrt{\mu^+}\int_{-1}^z a )}{\cosh(\sqrt{\mu^+}\int_{-1}^0 a )}\Big);
\end{eqnarray*}
since moreover
$$
(1+\dz\sigma)[\Delta,\Op(\Xi)]=\Op((1+\dz\sigma)\Delta\Xi))+2\Op(i(1+\dz\sigma)\nabla\Xi\cdot\xi),
$$
one easily gets that
\begin{equation}\label{tail0}
\bP_I \phi_{app}=\eps^+\mu^+ \Op(R_1+R_2)\psi+\eps^+\mu^+ (1+\dz\sigma) \Op(\frac{\Delta\Xi}{\eps^+})\psi,
\end{equation}
with
\begin{eqnarray*}
\eps^+ R_1(X,z,\xi)&=& \frac{1+\mu^+\abs{\nabla\sigma}^2}{1+\dz\sigma}\exp(i\mu^+\int_z^0 b)\\
&\times& \Big(\frac{\dz a}{\sqrt{\mu^+}} \frac{\sinh(\sqrt{\mu^+}\int_{-1}^z a )}{\cosh(\sqrt{\mu^+}\int_{-1}^0 a )}+i\dz b \frac{\cosh(\sqrt{\mu^+}\int_{-1}^z a )}{\cosh(\sqrt{\mu^+}\int_{-1}^0 a )}\Big),\\
\eps^+ R_2(X,z,\xi)&=&2i(1+\dz\sigma)\nabla\Xi\cdot\xi-2\nabla\sigma\cdot\nabla\dz\Xi.
\end{eqnarray*}
We now turn to control all the components of (\ref{tail0}):\\
- \emph{Control of $\Op(R_j)\psi$, ($j=1,2$).} Let us remark that
$$
\Op(R_j)\psi=\Op(\widetilde{\bf R}_j)\cdot \nabla\psi^\dagger \qquad (j=1,2,3),
$$
with $\psi^\dagger=\exp(c_0\sqrt{\mu^+}z\abs{D})\psi$ and
$$
\widetilde{\bf R}_j(X,z,\xi)=-i\frac{R_j(X,z,\xi)}{\abs{\xi}}\frac{\xi}{\abs{\xi}}\exp(-c_0\sqrt{\mu^+}z\abs{\xi}),
$$
and where $c_0>0$ is such that $\frac{1+\dz\sigma}{1+\mu^+\abs{\nabla\sigma}^2}\geq 2 c_0$ for all $(X,z)\in\cS$ (and thus $a(X,z,\xi)\geq 2c_0\abs{\xi}$). It follows from the definition of $R_j$ and $c_0$ 
that $\widetilde{\bf R}_j$ is a pseudo-differential operator of order zero, whose coefficients depends on second and lower order derivatives of $\zeta$ (through $\sigma$). Using Theorem 1 of \cite{LannesJFA}, which gives precise estimates of the operator norm of pseudo-differential estimates, we then get that
$$
\forall 0\leq s\leq t_0,\quad \forall z\in [-1,0],\qquad \abs{\widetilde{\bf R}_j(X,z,D)f}_{H^s}\leq M\abs{f}_{H^s}
$$
(it is easy to check that the possibly singular terms $1/\eps^+$ in the definition of $R_j$ are compensated by the $\eps^+$ contained in $\sigma=\eps^+(z+1)\zeta$).
We thus deduce that
\begin{eqnarray}
\nonumber
\Abs{\Lambda^s\Op(R_1+R_2)\psi}_2&\leq& M \Abs{\Lambda^s\nabla\psi^\dagger}_{2},\\
\label{tail1}
&\leq& M \abs{\psi}_{\Hm^{s+1/2}},
\end{eqnarray}
the last inequality following from an easy computation (recall that $z\leq 0$), as in Proposition 2.2 of \cite{AL} for instance.\\
- \emph{One can prove along the same lines that} 
\begin{equation}\label{tail2}
\Abs{\Lambda^s\Op((1+\dz\sigma)\frac{\Delta\Xi}{\eps^+})\psi}_2\leq M(t_0+3)\abs{\psi}_{\Hm^{s+1/2}}
\end{equation}
(there is $M(t_0+3)$ rather than $M$ in the r.h.s. because $\Delta\Xi$ involves third order derivatives of $\zeta$).

\medbreak

From (\ref{tail0}), (\ref{tail1}) and (\ref{tail2}), we get that 
\begin{equation}\label{tail3}
\bP_I\phi_{app}=\eps^+ \mu^+ h_1
\quad\mbox{ with }\quad
\Abs{\Lambda^s h_1}_2\leq M(t_0+3)\abs{\psi}_{\Hm^{s+1/2}}.
\end{equation}
- \emph{Analysis of $\bP_{II}\phi_{app}$.} From the definition of $\bP_{II}$ and $\phi_{app}$, one gets easily
$$
\bP_{II}\phi_{app}=\Op(R_3)\psi,\quad \mbox{ with }\quad
R_3(X,z,\xi)=\frac{1}{\eps^+}\big( \frac{1}{1+\dz\sigma}\dz\abs{\nabla\sigma}^2-\Delta\sigma\big)\dz\Xi.
$$
Proceeding as in the previous point, we get, for all $0\leq s\leq t_0$,
\begin{equation}\label{tail4}
\Abs{\bP_{II}\phi_{app}}_2\leq M\abs{\psi}_{\Hm^{s+1/2}}.
\end{equation}
- \emph{End of the proof.} The result follows directly from (\ref{tailm}), (\ref{tail3}) and (\ref{tail4}), with $h=h_1+\bP_{II}\phi_{app}$ (it is straightforward to check that $\phi_{app}$ satisfies the boundary conditions).
\end{proof}

It follows from the lemma that
the difference $u=\phi-\phi_{app}$ solves
\begin{equation}\label{eqcontu}
\left\lbrace
\begin{array}{l}
\namp\cdot P(\Sigma^\pm)\namp u=-\eps^+\mu^+ h\\
u_\interff=0,\qquad \dn u_{\vert_{z=-1}}=0.
\end{array}\right.
\end{equation}
The following lemma allows some control on $u$.
\begin{lemma}\label{lemestell}
Let $u\in\dot{H}^1(\cS)$ and $\tilde h\in L^2(\cS)$ be such that
$$
\left\lbrace
\begin{array}{l}
\namp\cdot P(\Sigma^+)\namp u=\tilde h,\\
u_\interff=0,\qquad \dn u_{\vert_{z=- 1}}=0.
\end{array}\right.
$$
Then, for all $0\leq s\leq t_0$, 
$$
\Abs{ \Lambda^{s+1} \namp u}_{2}\leq \frac{1}{\sqrt{\mu}} M\Abs{\Lambda^{s} \tilde h}_2.
$$
\end{lemma}
\begin{proof}
 Let us first remark that $\tilde h=\namp\cdot {\bf g}$, with ${\bf g}=\big(\int_{-1}^z h\big){\bf e}_z$ and we thus deduce from Lemma \ref{lemkk} that
\begin{equation}\label{lemstep1}
\Abs{\Lambda^s \namp u}_2\leq M \Abs{\Lambda^s {\bf g}}_2\leq M \Abs{\Lambda^s \tilde h}_2.
\end{equation}
For $1\leq j\leq d$, one remarks further that $v=\partial_j u$ solves
$$
\left\lbrace
\begin{array}{l}
\namp\cdot P(\Sigma^+)\namp v=\namp\cdot {\bf g},\\
u_\interff=0,\qquad \dn u_{\vert_{z=- 1}}=-{\bf e}_z\cdot{\bf g}_{\vert_{z=-1}}.
\end{array}\right.
$$
with ${\bf g}=-\partial_j P(\Sigma^+)\namp u+\frac{1}{\sqrt{\mu^+}}\tilde h{\bf e}_j$. 
We thus deduce from (\ref{lemkk}) that
\begin{eqnarray}
\nonumber
\Abs{\Lambda^s\partial_j\namp u}_2&\leq& M \Abs{\Lambda^s {\bf g}}_2\\
\label{lemstep2}
&\leq& M(1+\frac{1}{\sqrt{\mu}})\Abs{\Lambda^s \tilde h},
\end{eqnarray}
where we used (\ref{lemstep1}) and (\ref{prodeststrip}) for the second 
inequality. The lemma is thus proved.  
\end{proof}
Applying Lemma \ref{lemestell} to (\ref{eqcontu}) with $\tilde h=-\eps^+\mu^+ h$, we get (replacing $\eps^+$ and $\mu^+$ by $\eps$ and $\mu$ since the ratios $H^+/H^-$ and $H^-/H^+$ are controlled by $M$),
\begin{eqnarray}
\nonumber
\Abs{\Lambda^{s+1}\namp u}_{2}&\leq& \eps\sqrt{\mu}M \Abs{\Lambda^s h}_2\\
\label{noiz}
&\leq& \eps\sqrt{\mu}M(t_0+3)\abs{\psi}_{\Hm^{s+1/2}},
\end{eqnarray}
where the last inequality follows from Lemma \ref{lemreste}.\\
In order to prove the theorem, let us now remark that 
$$
\Gpb\psi-\dn\phi_{app}\,_{\vert_{z=0}}=\dn u_{\vert_{z=0}}={\bf e}_z\cdot P(\Sigma^+)\namp u_{\vert_{z=0}}.
$$
By Green's formula, we get therefore, for all $\varphi\in L^2(\R^d)$,
\begin{eqnarray*}
\lefteqn{\int_{\R^d}\Lambda^{s+1/2}\big(\Gpb\psi-\dn\phi_{app}\,_{\vert_{z=0}}\big)\varphi}\\
&=&\int_{\cS} P(\Sigma^+)\namp u\cdot \namp \Lambda^{s+1/2}\varphi^\dagger
-\eps^+\mu^+\int_{\cS}h\Lambda^{s+1/2}\varphi^\dagger\\
&=& \int_\cS \Lambda^{s+1} P(\Sigma^+)\namp u\cdot \Lambda^{-1/2}\namp \varphi^\dagger
-\eps^+\mu^+\int_\cS\Lambda^s h \Lambda^{1/2}\varphi^\dagger,
\end{eqnarray*}
where $\varphi^\dagger$ is defined as
$$
\varphi^\dagger(\cdot,z)=\chi(\sqrt{\mu^+}z\abs{D})\varphi,
$$
($\chi$ being a smooth compactly supported function equal to $1$ in a neighborhood of the origin). Since $\Abs{\Lambda^{-1/2}\namp\varphi^\dagger}_2\lesssim (\mu^+)^{1/4}\abs{\varphi}_2$ and $\Abs{\Lambda^{1/2}\varphi^\dagger}_2\lesssim (\mu^+)^{-1/4}\abs{\varphi}_2$ (see Prop. 2.2 of \cite{AL}), we get therefore
\begin{eqnarray*}
\int_{\R^d}\Lambda^{s+1/2}\big(\Gpb\psi-\dn\phi_{app}\,_{\vert_{z=0}}\big)\varphi
&\leq& M\big(\mu^{1/4}\Abs{\Lambda^{s+1}\namp u}_2+\eps\mu^{3/4}\Abs{\Lambda^s h}_2\big)\abs{\varphi}_2\\
&\leq& \eps\mu^{3/4}M(t_0+3)\abs{\psi}_{\Hm^{s+1/2}}\abs{\varphi}_2,
\end{eqnarray*}
the last line being a consequence of (\ref{noiz}) and Lemma \ref{lemreste}.
Since moreover one computes easily that
\begin{eqnarray*}
\dn\phi_{app}\,_{\vert_{z=0}}&=&{\bf e}_z\cdot P(\Sigma^+)\namp\phi_{app}\,_{\vert_{z=0}}\\
&=&\sqrt{\mu^+}\sqrt{\abs{\xi}^2+\mu^+(\abs{\nabla\zeta}^2\abs{\xi}^2-(\nabla\zeta\cdot\xi)^2)}
\tanh(\int_{-1}^0a),
\end{eqnarray*}
the first estimate of the theorem follows from a standard duality argument.\\
For the second estimate, we proceed as above to get
$$
\int_{\R^d}\Lambda^s\big(\Gpb\psi-\dn\phi_{app}\,_{\vert_{z=0}}\big)\varphi=
 \int_\cS \Lambda^{s+1} P(\Sigma^+)\namp u\cdot \Lambda^{-1}\namp \varphi^\dagger
-\eps^+\mu^+\int_\cS \Lambda^sh \varphi^\dagger.
$$
Since $\Abs{\Lambda^{-1}\namp\varphi^\dagger}_2\lesssim \sqrt{\mu}\abs{\varphi}_2$ and $\Abs{\varphi^\dagger}_2\lesssim \abs{\varphi}_2$, we deduce as above by a duality argument that
\begin{eqnarray*}
\abs{\Gpb\psi-\dn\phi_{app}\,_{\vert_{z=0}}}_{H^s}&\lesssim&\sqrt{\mu}M\Abs{\Lambda^{s+1}\namp u}_2+\eps\mu\Abs{\Lambda^s h}_2\\
&\leq& \eps\mu M(t_0+3) \abs{\psi}_{\Hm^{s+1/2}},
\end{eqnarray*}
where we used Lemma \ref{lemestell} (with $\widetilde{h}=-\eps^+\mu^+ h$) and Lemma \ref{lemreste} to get
the second inequality. The end of the proof is then exactly as for the first estimate.
\subsection{Symbolic analysis of $(\Gmb)^{-1}\Gpb$}

Theorem \ref{theotail} shows that $\Op(S^\pm)$, with $S^\pm$ as given by (\ref{defSpm}),
 provides a good description of $\Gpmb$. It is thus natural
to expect that a good symbolic description of the operator $(\Gmb)^{-1}\Gpb$ (which exists by Proposition \ref{prop1}) is given by
$$
(\Gmb)^{-1}\Gpb\sim \Op(\frac{S^+}{S^-})=\Op\Big(-\frac{H^+}{H^-}\frac{\tanh(\sqrt{\mu^+}t^+)}{\tanh(\sqrt{\mu^-}t^-)}\Big),
$$ 
with $t^\pm(X,\xi)$ as in Theorem \ref{theotail}. The following corollary shows that this
is indeed the case. 
\begin{corollary}\label{coroinvertsymb}
Let $t_0>d/2$ and $\zeta\in H^{t_0+3}(\R^d)$ be such that (\ref{sectII1}) is satisfied.
Then for all $0\leq s\leq t_0$ and $\psi\in {H}^{s-1/2}(\R^d)$, one has, with $k=0,1$,
$$
\babs{\big[(\Gmb)^{-1}\Gpb-\Op(\frac{S^+}{S^-})\big]\psi}_{\Hm^{s+1/2}}\leq \eps\mu^{-k/4}M(t_0+3)\abs{\psi}_{H^{s-k/2}}.
$$
\end{corollary}
\begin{remark}
As for the symbolic analysis of $\Gpb$, it is crucial to take into account the nonhomogeneous tail of the symbol. A simple look at the flat case shows indeed that 
one cannot expect an estimate better than $O(1/\mu^{3/4})$ 
if the principal symbol (constant equal to $-H^+/H^-$ here) is used instead of $S¨+/S^-$. In the case $k=1$, e have the same $O(\eps\sqrt{\mu})$ improvement as in Theorem \ref{theotail}. Note also that the $\mu^{-1/4}$ in the r.h.s. of the estimate given in the corollary is not singular since it corresponds to the fact that the $\Hm^{s+1/2}$-norm (rather than the $H^{s+1/2}$-norm) is considered in the l.h.s. In the case $k=0$, there is a further gain of $\mu^{1/4}$ in the estimate, but a loss
of half-a-derivative, as in Theorem \ref{theotail}.
\end{remark}
\begin{remark}
Contrary to Proposition \ref{prop1} for instance, Corollary \ref{coroinvertsymb} requires that $\psi$ belongs to a standard Sobolev space rather than a homogeneous one.
\end{remark}
\begin{proof}
Denoting by $S^\pm(X,\xi)$ the  symbol of the DN operator $\Gpmb$ as given by
Theorem \ref{theotail}, we define
$
\widetilde{\psi}^-=\Op(\frac{S^+}{S^-})\psi.
$
We also denote by $\widetilde{\phi}^-$ the solution of (\ref{eqsectII1-3})$_-$ with Dirichlet condition $\widetilde{\psi}^-$ at $z=0$. Then, by definition of $\Gmb$, one has
$\dn \widetilde{\phi}^-_{\vert_{z=0}}=\Gmb\widetilde{\psi}^-$. If $\phi^-$ is as in (\ref{eqsectII1-8}), the quantity $u=\phi^--\widetilde{\phi}^-$ solves therefore
$$
\left\lbrace
\begin{array}{l}
\namm\cdot P(\Sigma^-)\namm u=0,\\
\dn u_{\vert_{z=0}}=\Gmb\psi^--\Gmb\widetilde{\psi}^-,\qquad
\dn u\vert_{z=-1}=0.
\end{array}\right.
$$
Proceeding as in the proof of Proposition \ref{prop1}, we get that
\begin{eqnarray*}
\int_{\cS^-}P(\Sigma^-)\namm\Lambda^s u\cdot \namm \Lambda^s u&=&-\int_{\R^d}\Lambda^s
\big(\Gpb\psi-\Gmb\widetilde{\psi}^-\big)\Lambda^s u_0\\
& &+\int_{\cS^-}[\Lambda^s,P(\Sigma^-)]\namm u\cdot \namm \Lambda^s u,
\end{eqnarray*}
where $u_0$ stands for $u_{\vert_{z=0}}$.
Let us now state the following lemma.
\begin{lemma}
One has, with $k=0,1$,
$$
\babs{(\Lambda^s \big(\Gpb\psi-\Gmb\widetilde{\psi}^-\big),\Lambda^s u_0\big)}
\leq \eps{\mu^{1-k/4}}M(t_0+3) \abs{\psi}_{H^{s-k/2}}\abs{u_0}_{\Hm^{s+1/2}}.
$$
\end{lemma}
\begin{proof}[Proof of the lemma]
We first consider the case $k=1$. Let us remark\footnote{We use the same notation $\Pp$ for the operator $\frac{\abs{D}}{(1+\sqrt{\mu}\abs{D})^{1/2}}$ and its symbol $\frac{\abs{\xi}}{(1+\sqrt{\mu}\abs{\xi})^{1/2}}$} that
$$
\big(\Lambda^s \Gpb\psi,\Lambda^s u_0\big)=\big(\Op\big(\frac{S^+}{\Pp^2}\big)^*\Lambda^s\psi,\Lambda^s\Pp^2 u_0\big)+R_1+R_2,
$$
with
\begin{eqnarray*}
R_1&=&\big([\Lambda^s,\Gpb ]\psi,\Lambda^s u_0\big)\\
R_2&=&\big(\Lambda^{s-1/2}\psi,\Lambda^{1/2}(\Gpb-\Op(S^+))\Lambda^s u_0\big).
\end{eqnarray*}
Remarking that $\mu^{1/4}\abs{\Pp \psi}_{H^{s-1}}\lesssim \abs{\psi}_{H^{s-1/2}}$, we can 
use (\ref{DNcommut}) and Theorem \ref{theotail} to get
$$
\abs{R_j}\leq \mu M \frac{\eps}{\mu^{1/4}}\abs{\psi}_{H^{s-1/2}}\abs{u_0}_{\Hm^{s+1/2}}
\qquad (j=1,2).
$$
Recalling that $\widetilde{\psi}^-=\Op(S^+/S^-)\psi$, we write similarly
$$
\big(\Lambda^s\Gmb\widetilde{\psi}^-,\Lambda^s u_0\big)=\big(\Op\big(\frac{S^-}{\Pp^2}\big)^*\Op(\frac{S^+}{S^-})\Lambda^s\psi,\Lambda^s\Pp^2 u_0\big)+R'_1+R'_2,
$$
with
\begin{eqnarray*}
R_1'&=&\big([\Lambda^s,\Gmb ]\widetilde{\psi}^-,\Lambda^s u_0\big)\\
R_2'&=&\big(\Lambda^{s-1/2}\widetilde{\psi}^-,\Lambda^{1/2}(\Gmb-\Op(S^-))\Lambda^s u_0\big),\\
R_3'&=&\big(\Lambda^{1/2}[\Lambda^s,\Op(\frac{S^+}{S^-})]\psi,\Lambda^{-1/2}\Op(\frac{S^-}{\Pp^2})\Lambda^s \Pp^2u_0\big).
\end{eqnarray*}
Remarking (this is a consequence of Theorem 1 and Corollary 30 of \cite{LannesJFA}) that $\Abs{\Op(S^+/S^-)}_{H^{s-1/2}\to H^{s-1/2}}\leq M$ ($0\leq s\leq t_0$), we get as above that $R_1'$ and $R'_2$ satisfy the same estimates as $R_1$ and $R_2$ respectively. 
Using Theorems 1 and 8 of \cite{LannesJFA}, we also get that 
\begin{eqnarray*}
\abs{[\Lambda^s,\Op(S^+/S^-)]\psi}_{H^{1/2}}&\leq& \eps M\abs{\psi}_{H^{s-1/2}}\\
\abs{\Op(S^-/\Pp^2)\Lambda^s \Pp^2u_0}_{H^{-1/2}}&\leq& \mu M \abs{\Pp^2u_0}_{H^{s-1/2}};
\end{eqnarray*}
 since $\mu\abs{\Pp^2 u_0}_{H^{s-1/2}}\lesssim \mu^{3/4}\abs{u_0}_{\Hm^{s+1/2}}$, this allows us to conclude that $R_3'$ satisfies the same estimate as $R_j$ and $R_j'$ ($j=1,2$).\\
We thus have
\begin{eqnarray}
\nonumber
\lefteqn{\big(\Lambda^s(\Gpb\psi-\Gmb\widetilde{\psi}^-),\Lambda^s u_0\big)=
R_1+R_2+R_1'+R_2'+R_3'}\\
\label{eheh}
&+&\big(\Lambda^{1/2}(\Op(\frac{S^+}{\Pp^2})^*-\Op(\frac{S^-}{\Pp^2})\Op(\frac{S^+}{S^-}))\Lambda^s\psi,\Lambda^{s-1/2} \Pp^2 u_0 \big)
\end{eqnarray}
We then deduce from the above estimates on the residual terms that
\begin{eqnarray*}
\lefteqn{\babs{\big(\Lambda^s(\Gpb\psi-\Gmb\widetilde{\psi}^-),\Lambda^s u_0\big)}
\leq \abs{\psi}_{H^{s-1/2}}\abs{\Pp u_0}_{H^s}}\\
&\times& \big(\eps{\mu}^{3/4} M +\mu^{-1/4}\Abs{\Op(\frac{S^+}{\Pp^2})^*-\Op(\frac{S^-}{\Pp^2})\Op(\frac{S^+}{S^-})}_{H^{-1/2}\to H^{1/2}}\big).
\end{eqnarray*}
The result follows therefore from the observation that
$$
\Abs{\Op(\frac{S^+}{\Pp^2})^*-\Op(\frac{S^-}{\Pp^2})\Op(\frac{S^+}{S^-})}_{H^{-1/2}\to H^{1/2}}
\leq \eps{\mu} M,
$$
which is a consequence of the composition estimate\footnote{Theorem 8 of \cite{LannesJFA} deals with commutator rather than composition estimates. However, the commutator estimates of Theorem 8 follows from a composition estimate, exactly in the same way as Theorem 7(iii) follows from Theorem 7(i).} of Theorem 8 of \cite{LannesJAMS}
and of the estimate\footnote{This estimate on the adjoint is not stated in \cite{LannesJAMS}. However it can classically be derived with the same techniques as the commutator estimates of Theorem 8 of that reference. See also Proposition 1.8 of \cite{Grenier} and Chapter 13, \S 9 of \cite{Taylor}.} $\Abs{\Op(\frac{S^+}{\Pp^2})^*-\Op(\frac{S^+}{\Pp^2})}_{H^{-1/2}\to H^{1/2}}\leq \eps{\mu} M$.\\
We now briefly indicate the modifications to be performed in the case $k=0$. 
For the estimate of $R_1$, we just replace
the control $\mu^{1/4}\abs{\Pp\psi}_{H^{s-1}}\lesssim \abs{\psi}_{H^{s-1/2}}$ by $\abs{\Pp\psi}_{H^{s-1}}\leq \abs{\psi}_{H^s}$.  For $R_2$, we rather write
$$
R_2=\big(\Lambda^s \psi,(\Gpb-\Op(S^+))\Lambda^s u_0\big),
$$
and use Cauchy-Schwarz inequality and the second estimate of Theorem \ref{theotail}. The same modifications must be done for $R_1'$ and $R_2'$ respectively. For $R_3'$, we write
$$
R_3'=\big(\Lambda[\Lambda^s,\Op(\frac{S^+}{S^-})]\psi,\Lambda^{-1}\Op(\frac{S^-}{\Pp^2})\Lambda^s \Pp^2u_0\big).
$$
and use the estimates
\begin{eqnarray*}
\abs{[\Lambda^s,\Op(S^+/S^-)]\psi}_{H^{1}}&\leq& \eps M\abs{\psi}_{H^{s}}\\
\abs{\Op(S^-/\Pp^2)\Lambda^s \Pp^2u_0}_{H^{-1}}&\leq& \mu M \abs{\Pp^2u_0}_{H^{s-1}}
\leq \mu M\abs{u_0}_{\Hm^{s+1/2}}.
\end{eqnarray*}
The control of the last term in (\ref{eheh}) is modified along the same lines as $R_3'$ and
the result follows.
\end{proof}

Thanks to the lemma, one can proceed as in the proof of Proposition \ref{prop1} to get
(in the case $k=1$, the modifications for the case $k=0$ are straightforward),
$$
\Abs{\Lambda^s \namm u}_2\leq \eps M(t_0+3)\big(\mu^{1/4}\abs{\psi}_{H^{s-1/2}}+\Abs{\Lambda^{s-1}\namm\phi^-}_2\big),
$$
and thus, after a continuous induction on $s$,
$$
\Abs{\Lambda^s \namm u}_2\leq \eps\mu^{1/4} M(t_0+3)\abs{\psi}_{H^{s-1/2}}.
$$
The result follows therefore from (\ref{DNstar}) and the fact that  $u_0=\big((\Gmb)^{-1}\Gpmb-\Op(S^+/S^-)\big)\psi$.
\end{proof}

\subsection{Symbolic analysis of $(\Gmb)^{-1}\Gb_\mu[\eps\zeta]$}

The following corollary shows that the asymptotic description of  $(\Gmb)^{-1}\Gb_\mu[\eps\zeta]$
that one naturally expects from the definition of $\Gb_\mu[\eps\zeta]$ in Corollary \ref{coro2} and 
Theorem \ref{theotail} is correct. In order to state the corollary, it is convenient
to introduce the zero-th order symbol
\begin{equation}\label{defsymbJ}
S_J(X,\xi)=\urp-\urm\frac{H^-}{H^+}\frac{S^+(X,\xi)}{S^-(X,\xi)}.
\end{equation}
\begin{corollary}\label{coro4}
Let $t_0>d/2$ and $\zeta\in H^{t_0+3}(\R^d)$ be such that (\ref{sectII1}) is satisfied.
Then for all $0\leq s\leq t_0$ and $\psi\in {H}^{s-1/2}(\R^d)$, one has, with $k=0,1$,
$$
\babs{\big[(\Gmb)^{-1}\Gb_\mu[\eps\zeta]-\frac{1}{\uH^+}\Op\big(\frac{S^+}{S^- S_J}\big)\big]\psi}_{\Hm^{s+1/2}}\leq {\eps}{{\mu}^{-k/4}}M(t_0+3)\abs{\psi}_{H^{s-k/2}}.
$$
\end{corollary}
\begin{proof}
Thanks to Remark \ref{remequiv}, we can write 
$$
(\Gmb)^{-1}\Gb_\mu[\eps\zeta]=\frac{1}{\uH^+}(\Gmb)^{-1}\Gpb J[\zeta]^{-1}.
$$
It is a direct consequence of Corollary \ref{coroinvertsymb} that $\Op(S_J)$ provides a 
good approximation of $J[\zeta]$. The following lemma shows a corresponding property for 
its inverse. Throughout this proof, we only consider the case $k=1$; the modifications that
must be performed in the case $k=0$ are absolutely similar to those made in the proof of Corollary \ref{coroinvertsymb}.
\begin{lemma}\label{leminvJ}
Under the assumptions of the corollary, one has, for all $0\leq s\leq t_0$ and $k=0,1$,
$$
\babs{\big(J[\zeta]^{-1}-\Op(\frac{1}{S_J}\big)\psi}_{\Hm^{s+1/2}}\leq {\eps}{{\mu}^{-k/4}}M(t_0+3)\abs{\psi}_{H^{s-k/2}}.
$$
\end{lemma}
\begin{proof}[Proof of the lemma]
From Lemma \ref{lemma1}, we know that
\begin{eqnarray*}
\lefteqn{\babs{(J[\zeta]^{-1}-\Op(\frac{1}{S_J})\psi}_{\Hm^{s+1/2}}\leq M
\babs{\big(1-J[\zeta]\Op(\frac{1}{S_J})\big)\psi}_{\Hm^{s+1/2}}}\\
&\leq& M\Big(\babs{\big(1-\Op(S_J)\Op(\frac{1}{S_J})\big)\psi}_{\Hm^{s+1/2}}
+\babs{\big(J[\zeta]-\Op(S_J))\Op(\frac{1}{S_J})\big)\psi}_{\Hm^{s+1/2}}\Big)\\
&\leq&M(t_0+3)\Big(\frac{1}{\mu^{1/4}}\babs{\big(1-\Op(S_J)\Op(\frac{1}{S_J})\big)\psi}_{H^{s+1/2}}
+\frac{\eps}{{\mu}^{1/4}}\babs{\Op(\frac{1}{S_J})\psi}_{H^{s-1/2}}\Big),
\end{eqnarray*}
where we used Corollary \ref{coroinvertsymb} to derive the last inequality. Since the pseudodifferential estimates of \cite{LannesJAMS} show that $\Abs{1-\Op(S_J)\Op(1/S_J)}_{H^{s-1/2}\to H^{s+1/2}}\leq \eps M$ and $\Abs{\Op(1/S_J)}_{H^{s-1/2}\to H^{s-1/2}}\leq M$, the result follows easily.
\end{proof}
To conclude the proof of the corollary, let us remark that
$\babs{\big[(\Gmb)^{-1}\Gb_\mu[\eps\zeta]-\frac{1}{\uH^+}\Op\big(\frac{S^+}{S^- S_J}\big)\big]\psi}_{\Hm^{s+1/2}}$ is bounded
from above by
\begin{eqnarray*}
\lefteqn{\frac{1}{\uH^+}\babs{\big[(\Gmb)^{-1}\Gpb\Op(\frac{1}{S_J})-\Op\big(\frac{S^+}{S^- S_J}\big)\big]\psi}_{\Hm^{s+1/2}}}\\
& &+\frac{1}{\uH^+}\babs{(\Gmb)^{-1}\Gpb\big(J[\zeta]^{-1}-\Op(\frac{1}{S_J})\big)}_{\Hm^{s+1/2}}\\
&\leq& \frac{1}{\uH^+}\babs{\big[(\Gmb)^{-1}\Gpb\Op(\frac{1}{S_J})-\Op\big(\frac{S^+}{S^- S_J}\big)\big]\psi}_{\Hm^{s+1/2}}+\frac{\eps}{{\mu}^{1/4}}M(t_0+3)\abs{\psi}_{H^{s-1/2}},
\end{eqnarray*} 
where the second inequality comes from Proposition \ref{prop1} and Lemma \ref{leminvJ}. It is then an easy consequence of Corollary \ref{coroinvertsymb} and of the pseudodifferential estimates used many times in this section that the first component in the r.h.s. of the last equality is also bounded from above by $\frac{\eps}{{\mu}^{1/4}}M(t_0+3)\abs{\psi}_{H^{s-1/2}}$, which concludes the proof.
\end{proof}

\subsection{Symbolic analysis of $(\urp\frac{1}{\uH^-}\Gmb-\urm\frac{1}{\uH^+}\Gpb)^{-1}\partial_j$}

For the sake of clarity, we write 
\begin{equation}\label{defGtilde}
\widetilde{\Gb}=-(\urp\frac{1}{\uH^-}\Gmb-\urm\frac{1}{\uH^+}\Gpb)
\quad\mbox{ and }\quad
\widetilde{S}=-(\urp\frac{1}{\uH^-}S^--\urm\frac{1}{\uH^+}S^+).
\end{equation}
The operator   $\Pp^2\circ\widetilde{\Gb}^{-1}\partial_j$ ($1\leq j\leq d$) is well defined thanks 
Remark \ref{remGder}. It will be useful 
to describe it using a pseudodifferential operator, as in the following
corollary.
\begin{corollary}\label{coro5}
Let $t_0>d/2$ and $\zeta\in H^{t_0+3}(\R^d)$ be such that (\ref{sectII1}) is satisfied.
Then for all $0\leq s\leq t_0$, $k=0,1$,  and $f\in {H}^{s+1/2}(\R^d)$, one has
\begin{eqnarray*}
\babs{\Pp^2\circ \widetilde{\Gb}^{-1}\partial_j f-\Op\big(\frac{\Pp^2}{\widetilde{S}}\big)\partial_j f}_{H^{s+k/2}}
\leq \frac{\eps}{{\mu}^{1+k/4}}M(t_0+3)\babs{(1+\sqrt{\mu}\abs{D})^{1/2}f}_{H^s}.
\end{eqnarray*}
\end{corollary}
\begin{proof}
We can write
$$
\Pp^2\circ \widetilde{\Gb}^{-1}\partial_j f 
=\Op\big(\frac{\Pp^2}{\widetilde{S}}\big)\partial_j f
+R_1+R_2
$$
where, using the notation $\widetilde{f}=\widetilde{\Gb}^{-1}\partial_jf$,
$$
R_1=\Big[1-\Op\big(\frac{\Pp^2}{\widetilde{S}}\big)\circ\Op\big(\frac{\widetilde{S}}{\Pp^2}\big)\Big]\Pp^2 \widetilde{f},\qquad
R_2=\Op\big(\frac{\Pp^2}{\widetilde{S}}\big)
\big[\Op(\widetilde{S})-\widetilde{\Gb}\big]
 \widetilde{f}.
$$
From the composition estimates of \cite{LannesJFA}, we get, with $k=0,1$,
$$
\abs{R_1}_{H^{s+k/2}}\leq \eps M (t_0+3) \abs{\Pp^2 \widetilde{f}}_{H^{s-1+k/2}}\leq \eps \mu^{-k/4}
M(t_0+3)\abs{\Pp \widetilde{f}}_{H^s},
$$
 the
last inequality following from the observation that $\abs{\Pp f}_{H^{1/2}}\lesssim \mu^{-1/4}\abs{f}_{H^1}$ and $\abs{\Pp f}_2\lesssim \abs{f}_{H^1}$. Since moreover Remark \ref{remGder}
implies that 
$$
\abs{\Pp \widetilde{f}}_{H^s}\leq \mu^{-1}M\abs{(1+\sqrt{\mu}\abs{D})^{1/2}f}_{H^s},
$$
 we deduce that $R_1$ satisfies the bound given in the statement of the corollary.\\
For $R_2$, we get
\begin{eqnarray*}
\abs{R_2}_{H^{s+k/2}}&\leq& \frac{1}{\mu}M(t_0+3)\abs{[\Op(\widetilde{S})-\widetilde{\Gb}]
 \widetilde{f}}_{H^{s+k/2}}\\
&\leq&\eps\mu^{-k/4}M(t_0+3)\abs{\Pp \widetilde{f}}_{H^s},
\end{eqnarray*}
the second inequality being a consequence of the first ($k=1$) and second ($k=0$) points
of Theorem \ref{theotail}. One then shows as for $R_1$ that $R_2$ satisfies the
desired estimate.
\end{proof}

\section{Quasilinearization of the equations}\label{sectquasi}

The goal of this section is to derive a quasilinear system from the two-fluid equations (\ref{eqI-16nd}). 
Throughout this section, 
the operators $\Gmu$ and ${\mathpzc G}_{\mu^\pm}[\eps^\pm\zeta,1]$ are denoted $\Gb$ and $\Gpmb$ respectively,
when no confusion is possible.\\
For all $j\in\N^*$, 
we also denote by $d^j\Gb({\bf h})\psi$ and $d^j\Gpmb({\bf h})\psi$ (${\bf h}=(h_1,\dots,h_j)$) the $j$-th order
derivative
of the mappings $\zeta\mapsto \Gmu\psi$ and  $\zeta\mapsto {\mathpzc G}^\pm_{\mu^\pm}[\eps^\pm\zeta,1]\psi$ respectively in the direction ${\bf h}$.

\subsection{Definition and basic properties of some new operators}\label{sectdefnew}

We define first a linear operator $\T$  of order one, 
\begin{equation}
\label{eqIII-1}
\T f= \nabla\cdot(f\uVp)
+\urm\uH^- \Gb\circ (\Gmb)^{-1}\big(\nabla\cdot(f\jump{\uVpm})\big),
\end{equation}
where $\uVpm$ and $\uwpm$ are defined using the mappings $\wpm$ and $\Upm$ of Corollary \ref{coro1} as
\begin{equation}\label{defVw}
\uVpm=\Upm\psi-\eps\uwpm\nabla\zeta,\qquad \uwpm=\wpm\psi.
\end{equation}
In the water waves case 
 ($\rho^-=0$), this operator simplify considerably into
$\T f=\nabla\cdot (f \uVp)$. 
The symbolic analysis of Section \ref{sectsymbolic} allows us to prove that, in
the general case, it behaves roughly the same way in the sense that 
it is a first order operator with skew symmetric principal symbol. Note that it is
necessary
to use here a symbolic analysis ``with tail'' as in the previous section in order to
have a non singular dependence on the parameters in the estimates of the proposition.
\begin{proposition}\label{propT}
Let $t_0>d/2$ and $U=(\zeta,\psi)$, with $\zeta\in H^{t_0+3}(\R^d)$ satisfying (\ref{sectII1}) 
and $\psi\in \dot{H}^{t_0+3}(\R^d)$. Then
\item[(1)]  For all  $0\leq s\leq t_0$ and  $f\in H^{s+1/2}(\R^d)$, one has
$$
\abs{\T f}_{H^{s-1/2}}\leq M \abs{f}_{H^{s+1/2}}\abs{\nabla\psi}_{H^{t_0+1/2}}.
$$
\item[(2)] For all $\mfa=1+{\mathfrak b}$, with ${\mathfrak b}\in H^{t_0+1}(\R^d)$, and 
all $f\in L^2(\R^d)$,
$$
\big(\mfa \T f,f\big)\leq  M(t_0+3)(1+\abs{{\mathfrak b}}_{H^{t_0+1}})\abs{\nabla\psi}_{H^{t_0+1}}\abs{f}_2^2.
$$
\item[(3)] Let $K$ be a $d\times d$ matrix with entries in $H^{t_0+1}(\R^d)$; then for all $f\in H^1(\R^d)$,
$$
\big(\nabla\cdot K\nabla \T f,f\big)\leq  M(t_0+3)\abs{K}_{H^{t_0+1}}\abs{\nabla\psi}_{H^{t_0+1}}\abs{f}_{H^1}^2.
$$
\item[(4)] For all $f\in \dot{H}^{1/2}(\R^d)$, $g\in H^{1/2}(\R^d)$, one has
$$
\big(\T^* f,g\big)\leq M \abs{\nabla\psi}_{H^{t_0+2}}\abs{\Pp f}_2\big(\abs{g}_2+\mu^{1/4}\abs{g}_{H^{1/2}}\big).
$$
\end{proposition}
\begin{proof}
\emph{In the proofs below, we reduce $\T$ to its second component in (\ref{eqIII-1}), since
the estimates are much easier to prove for the first component and can therefore be omitted}.\\
For the first estimate, we write
$
\Gb\circ(\Gmb)^{-1}\nabla\cdot(f\jump{\uVpm})=A+B
$
with
\begin{eqnarray*}
A&=&\Gb\circ(\Gmb)^{-1}\nabla\cdot\big(\frac{\mu^{1/4}\abs{D}^{1/2}}{1+\mu^{1/4}\abs{D}^{1/2}}(f\jump{\uVpm})\big)\\
B&=&\Gb\circ(\Gmb)^{-1}\nabla\cdot\big(\frac{1}{1+\mu^{1/4}\abs{D}^{1/2}}(f\jump{\uVpm})\big).
\end{eqnarray*}
Recalling that $\Gb=\frac{1}{\uH}^+\Gpb\circ J[\zeta]^{-1}$, we use (\ref{eqsectII1-4sam}), Lemma \ref{lemma1} and Remark \ref{remGder} to get $\abs{A}_{H^{s-1/2}}\leq M \abs{f\jump{\uVpm}}_{H^{s+1/2}}$. Proceeding in the same way, but using (\ref{eqsectII1-4sambis}) instead of
(\ref{eqsectII1-4sam}), one can chech that $B$ can be controled similarly. Since
moreover, with $0\leq s\leq t_0$, one has
 $\abs{f\jump{\uVpm}}_{H^{s+1/2}}\leq M \abs{f}_{H^{s+1/2}}\abs{\nabla\psi}_{H^{t_0+1/2}}$
(from the definition of $\uVpm$  and Corollary \ref{coro1}), the result follows.\\
For the second point of the proposition, we write
 \begin{eqnarray*}
\lefteqn{\big(\mfa \Gb\circ(\Gmb)^{-1}\nabla\cdot(f\jump{\uVpm}),f\big)=
\big(\nabla\cdot(f\jump{\uVpm}),\Op(\frac{S^+}{S^-S_J})(\mfa f)\big)}\\
& &+\big(f\jump{\uVpm},\nabla[(\Gmb)^{-1}\circ\Gb-\Op(\frac{S^+}{S^-S_J})](\mfa f)\big).
\end{eqnarray*}
Since the operator $f\mapsto \mfa \Op(\frac{S^+}{S^-S_J})^*\nabla\cdot(f\jump{\uVpm})$
is skew symmetric at leading order, the result for the first term of the r.h.s. 
follows from the pseudo-differential estimates of \cite{LannesJFA}. In order to control
the second term, we decompose it into $C_1+C_2$ with
\begin{eqnarray*}
C_1&=&\big(f\jump{\uVpm},\frac{\mu^{1/4}\abs{D}^{1/2}\nabla}{1+\mu^{1/4}\abs{D}^{1/2}}((\Gmb)^{-1}\circ\Gb-\Op(\frac{S^+}{S^-S_J}))(\mfa f)\big),\\
C_2&=&\big(f\jump{\uVpm},\frac{\nabla}{1+\mu^{1/4}\abs{D}^{1/2}}((\Gmb)^{-1}\circ\Gb-\Op(\frac{S^+}{S^-S_J}))(\mfa f)\big).
\end{eqnarray*}
From Cauchy-Schwarz inequality and Corollary \ref{coro4} (with $k=1$), we get $\abs{C_1}\leq M(t_0+3) \abs{f\jump{\uVpm}}_2\abs{\mfa f}_2$. 
For $C_2$, we proceed similarly but with
$k=0$ in Corollary \ref{coro4} to get that $C_2$ can be controled as $C_1$ and the
result follows.  \\
The proof of the third point being very similar to the second one, we omit it. For the last point, the term to control is
\begin{eqnarray*}
A&=&\big(\jump{\uVpm}\cdot\nabla(\Gmb)^{-1}\Gb f,g\big)\\
&=&-\frac{1}{\uH^+}\big(\Pp (\Gmb)^{-1}\Gpb J[\zeta]^{-1} f,(1+\sqrt{\mu}\abs{D})^{1/2}\frac{\nabla}{\abs{D}}(g\jump{\uVpm})\big).
\end{eqnarray*}
It is therefore a consequence of Cauchy-Schwarz inequality, Proposition \ref{prop1} and Lemma \ref{lemma1} that $\abs{A}\leq M \abs{\Pp f}_2\abs{(1+\sqrt{\mu}\abs{D})^{1/2}(g\abs{\uVpm})}_2$,
from which one deduces the result by a standard commutator estimate and Corollary \ref{coro1}.
\end{proof}
The linearization formula for $\partial^\alpha(\Gmu\psi)$ (see Proposition \ref{proplin} below)
is quite simple if given in terms of the ``good unknowns''  $\zeta_{(\alpha)}$ and $\psia$, where $\zetaa$ and $\psia$ are defined as
\begin{eqnarray}
\label{defw0}
\zetaa&=&\partial^\alpha\zeta,\\
  \label{defw}
	\psi_{(\alpha)}
        &=&\partial^\alpha\psi-\eps \uw \partial^\alpha\zeta,\quad\mbox{ with }\quad
        \uw=\urp\uwp-\urm\uwm,
\end{eqnarray}
or equivalently, if $\psi^\pm=\phi^\pm_{\vert_{z=0}}$ ($\phi^\pm$ solving (\ref{eqsectII2-1})), 
$$
	\psi_{(\alpha)}=\urp \psi^+_{(\alpha)}-\urm \psi^-_{(\alpha)},
\quad\mbox{ with }\quad
	\psi^\pm_{(\alpha)}=\partial^\alpha\psi^\pm-\eps \uwpm\partial^\alpha\zeta.
$$
The principal part of the linearization formula for $\partial^\alpha(\Gmu\psi)$  is given in terms of $\Gb\psia$ and $\T\partial^\alpha\zeta$, but to handle the
surface tension, it is also necessary to take into account the subprincipal part (with respect to $\psi$). In order to describe this subprincipal part, it is
convenient to introduce the following notation,
\begin{equation}\label{notacheck}
\zetaaa=(\zeta_{(\check{\alpha}^1)},\dots,\zeta_{(\check{\alpha}^{d+1})}),\qquad
\psiaa=(\psi_{(\check{\alpha}^1)},\dots,\psi_{(\check{\alpha}^{d+1})}),
\end{equation}
where $\check{\alpha}^j\in\N^{d+1}$ is such that $\check{\alpha}^j+\alpha_j{\bf e}_j=\alpha$.\\
We can now define the operator
$\Ga\psi_{\av{\check{\alpha}}}$ that arises in the description of the 
subprincipal part of  $\partial^\alpha(\Gmu\psi)$,
\begin{equation}\label{defGa}
\Ga \psiaa=\sum_{j=1}^{d+1}\alpha_jd{\mathpzc G}(\partial_j\zeta)
\psi_{(\check{\alpha}^j)}.
\end{equation}

Next, the linearization of the surface tension term shows that 
$$
\frac{1}{\eps\sqrt{\mu}}\partial^\alpha\mfk(\eps\sqrt{\mu}\zeta)=-\nabla\cdot \K[\eps\sqrt{\mu}\nabla\zeta]\nabla\partial^\alpha\zeta+\Kae\zetaaa+\dots
$$
where the dots stand for derivatives of $\zeta$ of order lower or equal to $\abs{\alpha}$ and $\zetaa$ is as in (\ref{notacheck}). The positive definite $d\times d$ matrix 
$\K[\nabla\zeta]$ is defined as
\begin{equation}\label{eqIII-4}
\K[\nabla\zeta]=\frac{(1+\abs{\nabla\zeta}^2)\Id-\nabla\zeta\otimes\nabla\zeta}{(1+\abs{\nabla\zeta}^2)^{3/2}},
\end{equation}
while the second order operator $\Ka$ is given for all $F=(f_1,\dots,f_{d+1})^T$
by
\begin{equation}\label{eqIII-4bis}
\Ka F=-\nabla\cdot \Big[\sum_{j=1}^{d+1}\big(d\K(\nabla\partial_j\zeta)\nabla f_j+d\K(\nabla f_j)\nabla\partial_j\zeta\big)\Big].
\end{equation}

Finally, we introduce here the first order operator $\E: \dot{H}^{1/2}(\R^d)^d\to H^{-1/2}(\R^d)^d$ that plays
an important role for Kelvin-Helmholtz instabilities,
\begin{equation}\label{eqIII-3}
\E=\nabla \circ(\urp\frac{1}{\uH^-}\Gmb-\urm\frac{1}{\uH^+}\Gpb)^{-1}\circ \nabla^T;
\end{equation}
the following proposition gathers some of its main properties.
\begin{proposition}\label{propE}
Let $t_0>d/2$  and $\zeta\in H^{t_0+2}(\R^d)$ be such that (\ref{sectII1}) is satisfied.
\item[(1)] There exists a constant $c\leq M$ such that for all $V\in \dot{H}^{1/2}(\R^d)$,
$$
0\leq\mu(\E V,V) \leq c\,\abs{(1+\sqrt{\mu}\abs{D})^{1/2}V}^2_2;
$$
we denote by ${\mfe}(\zeta)$ the smallest such constant.
\item[(2)] If moreover $\zeta$ is time dependent and satisfy (\ref{sectII1}) uniformly in time,
then, for all $F\in H^{1/2}(\R^d)^d$
$$
\babs{\big([\dt,\E]F,F\big)}\leq \frac{\eps}{\mu}M \abs{\dt \zeta}_\infty \babs{(1+\sqrt{\mu}\abs{D})^{1/2}F}^2_2.
$$
\end{proposition}
\begin{proof}
With $\widetilde{\Gb}$ as in (\ref{defGtilde}),
we can write
$
(\E V, V)
=\big(\widetilde{\Gb} w,w\big),
$
with $w=\widetilde{\Gb}^{-1}\nabla\cdot V$ and the positiveness of $\E$ follows therefore
from the positiveness of $\widetilde{\Gb}$.
Since moreover $\widetilde{\Gb}=-\frac{1}{\uH^-}\Gmb\circ J[\zeta]$, the upper bound
follows from Lemma \ref{lemma1} and Remark \ref{remGder}.\\
For the second point of the proof, we remark that $[\dt,\widetilde{\Gb}^{-1}]=-\widetilde{\Gb}^{-1}[\dt,\widetilde{\Gb}]\widetilde{\Gb}^{-1}$, so that
$$
\big([\dt,\E] F,F\big)=-\big([\dt,\widetilde{\Gb}]\widetilde{\Gb}^{-1}\nabla\cdot F, \widetilde{\Gb}^{-1}\nabla\cdot F\big).
$$
Using Proposition 3.6 of \cite{AL}, we deduce that
$$
\babs{\big([\dt,\E] F,F\big)}\leq \mu M \abs{\Pp \widetilde{\Gb}^{-1}\nabla\cdot F}_2,
$$
and the result follows from Remark \ref{remGder}.
\end{proof}

\subsection{Linearization formulas}\label{sectlin}

We establish here linearization formulas for the operator $\Gb$.
 These formula play an important role in the quasilinearization of the equations
presented in Section \ref{sectQL}. Proposition \ref{proplin} below gives a linearization formula at leading
order; in order to measure the size of the neglected terms, it is necessary to introduce at this point
 the energy ${\mathcal E}^N_\sigma(U)$ defined for all $N\in\N$, $\sigma\geq 0$, and $U=(\zeta,\psi)$ as
\begin{equation}\label{eqenergy}
{\mathcal E}^N_\sigma(U)=\abs{\nabla \psi}^2_{H^{t_0+2}}
+\sum_{\alpha\in \N^{d+1},\abs{\alpha}\leq N}\abs{\partial^\alpha\zeta}_{H^1_\sigma}^2+\abs{\psi_{(\alpha)}}_{\Hm^{1/2}}^2.
\end{equation}
It is convenient to introduce a specific notation for the terms involving $\zeta$ only in 
this energy, namely,
\begin{equation}\label{notaHn}
\abs{\zeta}_{H^{1,N}_\sigma}=\sum_{\alpha\in \N^{d+1},\abs{\alpha}\leq N}\abs{\partial^\alpha\zeta}_{H^1_\sigma}
\end{equation}
(the difference between $H^{1,N}_\sigma$ and $H^{1+N}_\sigma$  is that the former allows for time derivatives of $\zeta$). Another quantity that will be related to the energy through the stability
criterion derived in Section \ref{sectmain} is $\abs{\zeta}_{<N+1/2>}$ defined as
\begin{equation}\label{enzeta}
\enz=\sum_{\alpha\in\N^{d+1},\abs{\alpha}=N}\babs{\abs{D}^{1/2}\partial^\alpha\zeta}_2.
\end{equation}
For all $T>0$, we also denote by $E^N_{\sigma,T}$ the functional space
\begin{equation}\label{eqensp}
E^N_{\sigma,T}=\{U\in C([0,T];H^{N}\times\dot{H}^{t_0+3}(\R^d)),
\quad \sup_{0\leq t\leq T} \cE^N_\sigma(U(t))<\infty \}.
\end{equation}

We finally denote by $\mfm^N(U)$ any constant of the form
\begin{equation}\label{eqmfm}
\mfm^N(U)=C\big(M,\cE^N_\sigma(U),\frac{1}{\Bo}\big),
\end{equation}
with $M$ as in (\ref{eqM}).
We can now give an important linearization formula for $\Gb$ (note that for $d=1,2$,
the condition (\ref{condN}) is satisfied with $t_0=3/2$ and $N\geq 5$).
\begin{proposition}\label{proplin}
Let $T>0$, $t_0>d/2$ and $N\in\N$ be such that 
\begin{equation}\label{condN}
[(N+1)/2]\geq 1\vee t_0+1/2
\quad\mbox { and }\quad N\geq t_0+7/2.
\end{equation}
Let also 
$U=(\zeta,\psi)\in E^N_{\sigma,T}$ be such that (\ref{sectII1}) is uniformly satisfied on $[0,T]$.
Then for all $\alpha\in\N^{d+1}$ with $1\leq \abs{\alpha}\leq N$,
one has
\begin{eqnarray*}
  	\frac{1}{\mu}\partial^\alpha(\Gb\psi)=\frac{1}{\mu}\Gb\psia+ \eps R_\alpha,& &(\abs{\alpha}\leq N-1)\\
	\frac{1}{\mu}\partial^\alpha(\Gb\psi)=\frac{1}{\mu}\Gb\psia-\eps\T\partial^\alpha\zeta+
        \frac{1}{\mu}\Ga\psiaa+\eps R_\alpha,& &(\abs{\alpha}=N)
\end{eqnarray*}
where the linear operators $\T$ and $\Ga$ are defined in (\ref{eqIII-1}) and (\ref{defGa})
while $R_\alpha$ satisfies the estimate
$$
\forall 0\leq t\leq T,\qquad
\abs{R_\alpha(t)}_{H^1_\sigma} \leq \mfm^N(U(t)).
$$
\end{proposition}
\begin{proof} Let us denote by $\phi^\pm$ the solution of the transmission problem (\ref{eqsectII2-1})
provided by Proposition \ref{prop2}, and let us write $\psi^\pm={\phi^\pm}_\interff$ (or equivalently define $\psi^\pm$ as in (\ref{manif})). For the sake of notational simplicity, we also omit to write the dependence on the time variable when
this does not raise any confusion. Consistently with the notations of \S \ref{sectSD}, we also
denote by $d\Gb(h)\psi$ the derivative of the mapping $\zeta\mapsto{\mathpzc G}_\mu[\eps\zeta]\psi$ in the direction $h$.\\
We focus on the most difficult case, namely $\abs{\alpha}=N$ (when $\abs{\alpha}\leq N-1$, it is
a consequence of the first estimate of Proposition \ref{propT} that $\eps\T\partial^\alpha\zeta$ can be put in the residual term).
The proof relies heavily on the following lemma (note that the last estimate of the third
point is not used in the present proof  but will be necessary to
establish Proposition \ref{proptechnical}).
\begin{lemma}\label{lemmformdergen}
Let $t_0>d/2$ and $\zeta\in H^{t_0+2}(\R^d)$ be such that (\ref{sectII1}) is satisfied.
\item[(1)]For all $\psi\in \dot{H}^{1/2}(\R^d)$ and $h \in H^{t_0+2}(\R^d)$, one has
$$
d\Gb(h)\psi=-
\eps\Gb(h\jump{\urpm \uwpm})
-\eps\mu\T h.
$$
\item[(2)] One can replace $\Gpmb$ by $\Gb$ or $\wpm$ 
in (\ref{estder0})-(\ref{estder0bisk}).
\item[(3)] Let $r\geq 0$ and $\delta\in \N^{d+1}$ be such that $r+\abs{\delta}\leq N-1$, with $N$ as in the Proposition.
Then one has
$$
 \abs{\Pp \partial^\delta\psi}_{H^{r+1}_\sigma}
\leq \mfm^N(U)
\quad\mbox{ and }\quad
\abs{\Pp\partial^\delta\psi}_{H^{r+1}}\leq \mfm^N(U)\big(1+\eps\mu^{1/4}\enz\big).
$$
\end{lemma}

\begin{proof}[Proof of the lemma]
(1) Recalling that $\Gb=\frac{1}{\uH^+}\Gpb\circ J[\zeta]^{-1}$, where $J[\zeta]$ is as in Lemma \ref{lemma1}, one gets
$$
	d\Gb(h)=\frac{1}{\uH^+}d\Gpb(h)J[\zeta]^{-1}-\Gb\circ dJ[\zeta](h)\circ J[\zeta]^{-1}.
$$
Since moreover one can compute
$$
dJ[\zeta](h)=\urm\frac{H^-}{H^+}(\Gmb)^{-1}\circ d\Gmb(h)\circ(\Gmb)^{-1}\circ\Gpb-\urm \frac{H^-}{H^+}(\Gmb)^{-1}\circ d\Gpb(h),
$$
we get, observing that $\psi^+=J[\zeta]^{-1}\psi$ and $\psi^-=\frac{H^-}{H^+}(\Gmb)^{-1}\circ\Gpb\psi^+$, that
$$
d\Gb(h)\psi=\frac{1}{\uH^+}d\Gpb(h)\psi^+-\urm\Gb\circ(\Gmb)^{-1}\big(
d\Gmb(h)\psi^--\frac{H^-}{H^+}d\Gpb(h)\psi^+\big).
$$
We can thus use (\ref{formuleder}) to get
$$
d\Gb(h)\psi=-\frac{\eps}{\uH^+}(1+\urm\uH^-\Gb\circ(\Gmb)^{-1})\Gpb(h\uwp)
+\eps
\urm \Gb(h\uwm)
-\eps\mu\T h.
$$
The result is then a consequence of Remark \ref{remequiv}.\\
(2)
Proceeding as in the first point, it is possible to show that
$d^j\Gb(\bh)$ and $d{\mathpzc w}^\pm(\bh)$ in terms of  $\Gpb$, $J[\zeta]^{-1}$ and their shape derivatives,
so that one can deduce the result 
from (\ref{estder0})-(\ref{estder0bisk}), Remark \ref{remcons} and Lemma \ref{lemma1}.\\
(3)  Let us first prove the estimate on the $H^{r+1}_\sigma$-norm of $\Pp\partial^\delta\psi$.
Thanks to the assumptions made on $r$ and $\delta$, one has
$$
\abs{\Pp\partial^\delta\psi}_{H^{r+1}_\sigma}\leq \sum_{\beta\in\N^{d+1},\abs{\beta}\leq N-1}\abs{\Pp\partial^\beta\psi}_{H^1_\sigma}.
$$
Since moreover $\partial^\beta\psi=\psi_{(\beta)}+\eps \uw \partial^\beta\zeta$, we get
\begin{eqnarray}
\nonumber
\abs{\Pp\partial^\delta\psi}_{H^{r+1}_\sigma}&\leq& \sum_{\beta\in\N^{d+1},\abs{\beta}\leq N-1}\big(\abs{\Pp\psi_{(\beta)}}_{H^1_\sigma}+\eps\abs{\Pp (\uw \partial^\beta \zeta)}_{H^1_\sigma}\big)\\
\nonumber
&\leq & \max\{1,\frac{1}{\sqrt{\Bo}}\} \sum_{\alpha\in\N^{d+1},\abs{\alpha}\leq N}
\big(\abs{\Pp\psi_{(\alpha)}}_{2}\big)+\eps\abs{\uw}_{H^{t_0+2}}\abs{\zeta}_{H_\sigma^{1,N}}\big),
\end{eqnarray}
where we used the identity $\abs{\Pp f}_2\leq \abs{\nabla f}_2$ to derive the second inequality.
Since $\uw=\urp\wpl\psi-\urm\wm\psi$, the result follows easily from Corollary \ref{coro1}.\\
For the estimate
in $H^{r+1}$-norm, one derives along the same lines the identity
$$
\abs{\Pp\partial^\delta\psi}_{H^{r+1}}\leq \sum_{\alpha\in \N^{d+1},\abs{\alpha}=N}
\big(\abs{\Pp\psi_{(\alpha)}}_{2}+\eps\abs{\Pp (\uw \partial^\alpha \zeta)}_{2}\big).
$$
From the product estimate (that follows directly from Lemma 5.1 of \cite{Iguchi} for instance),
\begin{equation}\label{bofbof}
\abs{ fg}_{\Hm^{1/2}}\lesssim \abs{\nabla f}_{H^{t_0}}\abs{g}_2+ \abs{f}_\infty\abs{g}_{\Hm^{1/2}},
\end{equation}
and the observation that $\abs{\uw}_{\infty}\leq \sqrt{\mu}\mfm^N(U)$ (Corollary \ref{coro1}),
we get $\abs{\Pp (\uw \partial^\alpha \zeta)}_{2}\leq \mfm^N(U)(1+\eps\sqrt{\mu}\abs{\Pp\partial^\alpha\zeta}_2)$. The result then follows easily.
\end{proof}
Denoting by $\pb{\partial^\alpha,\Gb}\psi$ the Poisson bracket (with $\check{\alpha}^j$
as in (\ref{defGa}))
$$
\pb{\partial^\alpha,\Gb}\psi=\sum_{j=1}^{d+1} \alpha_j d \Gb(\partial_j\zeta)\partial^{\check{\alpha}^j}\psi,
$$
we can write
\begin{equation}\label{premest}
\partial^\alpha \Gb\psi=\Gb\partial^\alpha\psi+d\Gb(\partial^\alpha\zeta)\psi+
\Ga\psiaa+\eps\mu R_\alpha,
\end{equation}
with $\eps\mu R_\alpha=(\pb{\partial^\alpha,\Gb}\psi-\Ga\psiaa)+\eps\mu R'_\alpha$ and where $\eps\mu R'_\alpha$ is a sum of terms of the form 
\begin{equation}\label{Ajbi}
d^j\Gb(\partial^{\iota^1}\zeta,\dots,\partial^{\iota^j}\zeta)\partial^{\delta}\psi=:A_{j,\delta,{\iota}},
\end{equation}
where $j\in\N$, ${\iota}=(\iota^1,\dots,\iota^j)\in\N^{(d+1)j}$ and $\delta\in\N^{d+1}$ satisfy
\begin{equation}\label{condcoeff}
\sum_{i=1}^j\abs{\iota^i}+ \abs{\delta}=N,
\qquad 0\leq \abs{\delta}\leq N-2\quad\mbox{ and }\quad
\abs{\iota^i}<N.
\end{equation}
We first give some control on the $A_{j,\delta,\iota}$ (and thus on $R_\alpha'$). Let $l$ be such that $\abs{\iota^l}=\max_{1\leq m\leq j}\abs{\iota^m}$. Let us distinguish 
four cases:\\
\emph{(i) The case $\abs{\delta}=j=1$ and $\iota=\iota^1$ (and thus $\abs{\iota}=N-1$)}. One then
has $A_{j,\delta,\iota}=d\Gb(\partial^\iota\zeta)\partial^\delta\psi$, and we can therefore use the
first point of Lemma \ref{lemmformdergen} to get
$$
\frac{1}{\mu}A_{j,\delta,\iota}=-\eps \frac{1}{\sqrt{\mu}}\Gb(\partial^\iota \zeta\jump{\urpm\frac{1}{\sqrt{\mu}}(\cw^\pm[\zeta]\partial^\delta\psi)})-\eps{\mathpzc T}[\zeta,\partial^\delta\psi] \partial^\iota\zeta.
$$
We can thus deduce from Corollaries \ref{coro1} and\ref{coro2}, and Proposition \ref{propT}
 that
\begin{eqnarray*}
\frac{1}{\mu}\abs{A_{j,\delta,\iota}}_{H^1_\sigma}&\leq&
 \eps M(t_0+7/2)\abs{\zeta}_{H^{1,N}_\sigma}\abs{\Pp\partial^\delta\psi}_{H^{t_0+5/2}_\sigma};
\end{eqnarray*}
since $t_0+7/2\leq N$, these two estimates, together with the third point of Lemma
\ref{lemmformdergen}, yield that
\begin{equation}\label{estimtrain}
\abs{A_{j,\delta,\iota}}_{H^1_\sigma}\leq \eps \mu \mfm^N(U).
\end{equation}
\emph{(ii) The case $\abs{\delta}\leq N-1\vee t_0-3/2$ and $\abs{\iota^l}<N-1$}.
It follows  from (\ref{estder0bisk})\footnote{We still refer to (\ref{estder0bisk}) if $\Gpmb$ is replaced by $\Gb$,
thanks to  Lemma \ref{lemmformdergen}.ii; we do the same for all the identities involved in Lemma \ref{lemmformdergen}.ii}  that
$$
\frac{1}{\mu}\abs{A_{j,\delta,\iota}}_{H^1_\sigma}\leq \eps M \abs{\partial^{\iota^l}\zeta}_{H^{5/2}_\sigma}
\prod_{m\neq l}\abs{\partial^{\iota^m}\zeta}_{H^{1\vee t_0+5/2}_\sigma}\abs{\Pp\partial^\delta\psi}_{H^{1\vee t_0+3/2}_\sigma}.
$$
From the assumption that $\abs{\iota^l}<N-1$ and  the definition of $l$,
we also get that for all $m\neq l$, $\abs{\iota^m}\leq [(N-2)/2]$, and thus, from the assumption made on $N$, we have $\abs{\iota^m}+1\vee t_0+3/2\leq N$. Using the assumption on $\delta$, we deduce from the last point of Lemma \ref{lemmformdergen} that (\ref{estimtrain}) holds.\\
\emph{(iii) The case $\abs{\delta}=0$ and $\abs{\iota^l}=N-1$ (we take for simplicity $l=1$)}. Then one has $A_{j,\delta,\iota}=d^2\Gb(\partial^{\iota^1}\zeta,\partial^{\iota^2}\zeta)\psi$. Inequality (\ref{estder0bis}) provides a good control in terms of regularity but not with respect to its dependence on $\mu$. We can remark however,
after differentiating the formula stated in Lemma \ref{lemmformdergen}.i that
\begin{eqnarray*}
d^2\Gb(h_1,h_2)\psi&=&-\eps\sqrt{\mu} d\Gb(h_2)(h_1\jump{\urpm\frac{1}{\sqrt{\mu}}\uwpm})\\
& &-\eps\sqrt{\mu}\Gb(h_1\jump{\urpm\frac{1}{\sqrt{\mu}}d\cw^\pm(h_2)\psi})
-\eps\mu d {\mathpzc T}(h_2)h_1.
\end{eqnarray*}
With $h_j=\partial^{\iota^j}\zeta$,  one can use (\ref{estder0}) to control the first term of the r.h.s., Corollary \ref{coro2} and Lemma \ref{lemmformdergen}.ii for the second one, and
Lemma \ref{lemmformdergen}.ii for the last one.  Inequality (\ref{estimtrain}) is then satisfied.\\
\emph{(iv) The case  $\abs{\delta}\geq N-1\vee t_0-1/2$}. Using (\ref{estder0k}) and the second point of the Lemma, we now get
$$
\abs{A_{j,\delta,\iota}}_{H^1_\sigma}\leq M \eps^j\mu
\prod_{m=1}^j\abs{\partial^{\iota^m}\zeta}_{H^{1\vee t_0+2}_\sigma}\abs{\Pp\partial^\delta\psi}_{H^{2}_\sigma}.
$$
Since, by (\ref{condcoeff}), $\abs{\delta}+2\leq N$ and $\abs{\iota^m}\leq t_0+2$, one can conclude as in the previous case that (\ref{estimtrain}) holds for 
$N$ as in the statement of the proposition.

\medbreak 

We have therefore proved that the residual $R'_\alpha$ satisfies the estimate stated in the
proposition. Remark now 
that $\frac{1}{\mu}(\pb{\partial^\alpha,\Gb}\psi-\Ga\psiaa)$ is a sum of terms of the form (with $\uw$ as in (\ref{defw}))
$$
\frac{1}{\sqrt{\mu}}d\Gb(\partial^\iota\zeta)(\frac{1}{\sqrt{\mu}}\uw\partial^\delta\zeta)
\qquad
(\abs{\iota}=1,\quad \iota+\delta=\alpha)
$$
that are  bounded in $H^1_\sigma$-norm by $\eps\mfm^N(U)$ thanks to (\ref{estder0})  and Corollary \ref{coro1}.  We deduce
that $R_\alpha$ satisfies the estimate of the proposition.
The fact that (\ref{premest}) coincides with the identity given in the 
proposition is a consequence of the first point 
of Lemma \ref{lemmformdergen}. 
\end{proof}

We end this section with another linearization formula that describes the way
$1/\uH^\pm \Gpmb\psiapm$ departs from $\Gb\psia$ (when $\alpha=0$, both terms are equal).
This proposition will be used to express $\psiapm$ in terms of the ``good unknowns'' $\zeta$ and $\psia$ in the proof of Proposition \ref{propIII-1}. In this sense, it can be viewed as a generalization of the transmission problem solved in Proposition \ref{prop2}. The additional term with respect to the case $\alpha=0$ is responsible for the Kelvin-Helmholtz instabilities, as shown in the next section --- of course, in the water waves case ($\urm=0$), one has $\Gpb=\Gb$ and $\psiap=\psia$
and this destabilizing term disappears.
\begin{proposition}\label{proptechnical}
Under the same assumptions as in Proposition \ref{proplin}, one has
\begin{eqnarray*}
\frac{1}{\uH^\pm}\Gpmb\psi^\pm_{(\alpha)}=r_\alpha^\pm, \qquad (\abs{\alpha}\leq N-1),\\
\frac{1}{\uH^\pm}\Gpmb\psi^\pm_{(\alpha)}=\Gb\psi_{(\alpha)}-\underline{\rho}^\mp\eps{\mu}\uH^\mp\Gb\circ(\Gb^\mp)^{-1}\nabla\cdot(\partial^\alpha\zeta\jump{\uVpm})+ r_\alpha^\pm,\qquad (\abs{\alpha}= N)
\end{eqnarray*}
where $r_\alpha^\pm$ satisfies the estimate, with $\enz$ as in (\ref{enzeta}),
$$
\urpm\abs{(\Gpmb)^{-1}r^\pm_\alpha}_{\Hm^{3/2}} \leq \mfm^N(U)\big(1
+\eps\mu^{1/4} \urp\urm\enz\big).
$$
\end{proposition}
\begin{remark}\label{remdax}
A consequence of this proposition is that for all $\abs{\alpha}\leq N$,
$$
\urpm\abs{\Pp \psi^\pm_{(\alpha)}}_2\leq  \mfm^N(U)\big(1
+\eps\mu^{1/4} \urp\urm\enz\big).
$$
\end{remark}
\begin{proof}
 As for the proof of Proposition \ref{proplin}, we focus on that case $\abs{\alpha}=N$. Applying $\partial^\alpha$ to 
 the identity $\frac{1}{\uH^\pm}\Gpmb\psi^\pm=\Gb\psi$ and using 
(\ref{formuleder}) and Lemma \ref{lemmformdergen}.i
we
obtain
$$
\frac{1}{\uH^\pm}\Gpmb\psi^\pm_{(\alpha)}-\eps{\mu}\nabla\cdot(\partial^\alpha\zeta \uVpm)=\Gb\psi_{(\alpha)}-\eps{\mu}\T\partial^\alpha\zeta+ r_\alpha^\pm,
$$
where $r_\alpha^\pm$ is a sum of terms of the form, 
$$
B^\pm_{j,\delta,\iota}=d^j\Gb(\partial^{\iota^1}\zeta,\dots,\partial^{\iota^j}\zeta)\partial^\delta\psi
-\frac{1}{\uH^\pm}d^j\Gpmb(\partial^{\iota^1}\zeta,\dots,\partial^{\iota^j}\zeta)\partial^\delta\psi^\pm
$$
with $\sum_{i=1}^j\abs{\iota^i}+\abs{\delta}=N$
, $0\leq \abs{\delta}\leq N-1$ and $\abs{\iota^i}<N$.
The formula of the proposition is obtained after replacing $\T$ by its explicit 
expression (\ref{eqIII-1}).\\
We now must control $\urpm (\Gpmb)^{-1}B^\pm_{j,\delta,\iota}$ in $\Hm^{3/2}(\R^d)$. The $+$ case
is more difficult than the $-$ case because $\urm$ is possibly very small while
we always have $\urp\geq 1/2$. Therefore, we only show how to control $B^+_{j,\delta,\iota}$; 
moreover, only three endpoint cases deserve a proof, namely:\\
\emph{(i) The case $j=\abs{\iota^1}=1$ and $\abs{\delta}=N-1$}. One then has 
\begin{eqnarray*}
B^+_{j,\delta,\iota}&=&d\Gb(\partial^{\iota^1}\zeta)\partial^\delta\psi
-\frac{1}{\uH^\pm}d\Gpb(\partial^{\iota^1}\zeta)\partial^\delta\psi^+,\\
&=&\frac{1}{\uH^+}d\Gpb(\partial^{\iota^1}\zeta)(J^{-1}\partial^\delta\psi-\partial^\delta\psi^+)-\frac{1}{\uH^+}\Gpb J^{-1}dJ(\partial^{\iota^1}\zeta)J^{-1}\partial^\delta\psi,
\end{eqnarray*}
where $J=J[\zeta]$ and where we used the identity $\Gb=\frac{1}{\uH^+}\Gpb J[\zeta]^{-1}$. Since by (\ref{manif}), $\psi^+=J[\zeta]^{-1}\psi$, we thus get
\begin{eqnarray*}
B^+_{j,\delta,\iota}&=&-\frac{1}{\uH^+}d\Gpb(\partial^{\iota^1}\zeta)J^{-1}[J,\partial^\delta]J^{-1}\psi
-\frac{1}{\uH^+}\Gpb J^{-1}dJ(\partial^{\iota^1}\zeta)J^{-1}\partial^\delta\psi\\
&=&B_1+B_2.
\end{eqnarray*}
We now need the following lemma.
\begin{lemma}\label{lemIHP}
For all $\delta\in \N^{d+1}$, $\abs{\delta}\leq N-1$,  one has
$$
\abs{[J,\partial^\delta]u}_{\Hm^{3/2}}\leq \eps\urm\mfm^N(U)(1+\mu^{1/4}\abs{\zeta}_{< N+1/2>}).
$$
\end{lemma}
\begin{proof}[Proof of the lemma]
Recalling that $J[\zeta]=\urp-\urm\frac{H^-}{H^+}(\Gmb)^{-1}\Gpb$, and writing $v=(\Gmb)^{-1}\Gpb u$, we get $[\partial^\delta,J]u=\urm\frac{H^-}{H^+}[\partial^\delta,(\Gmb)^{-1}\Gpb]$ and
we observe moreover that 
$$
[\partial^\delta,(\Gmb)^{-1}\Gpb]=
 -(\Gmb)^{-1}[\partial^\delta,\Gmb]v+(\Gmb)^{-1}[\partial^\delta,\Gpb]u.
$$
Since both terms in the r.h.s. can be handled similarly, we just give the details for
the first one. It is a sum of terms of the form
$$
C_{j,\delta',\iota}=(\Gmb)^{-1}d^j\Gmb(\partial^{\iota^1}\zeta,\dots,\partial^{\iota^j}\zeta)\partial^{\delta'}v,
$$
with $\sum_{i=1}^j \abs{\iota^i}+\abs{\delta'}=\abs{\delta}$ and $\abs{\delta'}<\abs{\delta}$.
We can control all these terms thanks to (\ref{eqsectII1-6gen}), (\ref{eqsectII1-6genbis}) and Remark \ref{remcons}, except the most singular one (with respect to $\zeta$), which corresponds
to $j=1$, $\iota^1=\delta$ and $\delta'=0$. For this term, we write, using (\ref{formuleder}),
$$
C_{1,0,\delta}=-\eps \partial^\delta\zeta (\cw^-[\eps\zeta]v)-\eps^-\mu^-(\Gmb)^{-1}\nabla\cdot(\partial^\delta\zeta (\Um v)).
$$
Using the inequality $\abs{\partial^\delta\zeta}_{\Hm^{3/2}}\leq \mu^{-1/4}\abs{\zeta}_{< N+1/2>}$ to control the first term and Remark \ref{remGder} to control the second one, we get
\begin{eqnarray*}
\abs{C_{1,0,\delta}}_{\Hm^{3/2}}&\leq& \eps \mu^{-1/4}\abs{\zeta}_{< N+1/2>}
\abs{\cw^-[\eps\zeta]v}_{H^{3/2\vee t_0}}\\
& &+\eps (\abs{\partial^\delta\zeta}_{H^1}+\mu^{1/4}\abs{\zeta}_{< N+1/2>})\abs{\Um v}_{H^{3/2\vee t_0}}\\
&\leq& \eps\mfm^N(U)(1+\mu^{1/4}\abs{\zeta}_{< N+1/2>}),
\end{eqnarray*}
where the second inequality is a consequence of Corollary \ref{coro1}. The lemma then
follows easily.
\end{proof}
Using (\ref{remcons}) and the lemma, we get the desired control on $B_1$; for $B_2$ there is
no particular difficulty and we omit the proof.\\
\emph{(ii) The case $j=\abs{\delta}=1$ and $\abs{\iota^1}=N-1$}. One proceeds exactly as in 
the proof of Lemma \ref{lemIHP}.\\
\emph{(iii) The case $\abs{\delta}=0$, $j=2$ and $\abs{\iota^1}=N-1$}. One then has 
$$
B_{j,\delta,\iota}^+=d^2\Gb(\partial^{\iota^1}\zeta,\partial^{\iota^2}\zeta)\psi
-\frac{1}{\uH^\pm}d^2\Gpb(\partial^{\iota^1}\zeta,\partial^{\iota^2}\zeta)\psi^+,
$$
and the scenario is the same as in case (iii) in the proof of Proposition \ref{proplin}.
\end{proof}

\subsection{The quasilinear system}\label{sectQL}

We show in Proposition \ref{propIII-1} below that the two-fluid  equations (\ref{eqI-16nd})
can be ``quasilinearized''. In these quasilinearized equations, an operator plays a central 
role with respect to the formation of Rayleigh-Taylor/Kelvin-Helmholtz instabilities. We thus call
it the \emph{instability operator}.
\begin{definition}[Instability operator]\label{defiRT}
Let $t_0>d/2$, $T>0$ and $(\zeta,\psi)\in E^N_{\sigma,T}$ ($N\geq t_0+2$)  be such that (\ref{sectII1})
is uniformly satisfied on $[0,T]$.
The instability operator $\RT$ is defined as
$$
\RT f =\mfa f-\eps^2\mu\urp\urm\jump{\uVpm}\cdot\E(f\jump{V^\pm})-\frac{1}{\Bo}\nabla\cdot \K[\eps\sqrt{\mu}\nabla\zeta]\nabla f,
$$
with
$$
\mfa=1+\eps\jump{\urpm\big(\dt+\eps\uVpm\cdot\nabla)\uwpm}.
$$
\end{definition}
For notational convenience, we also introduce the matricial operators
\begin{equation}\label{defAB}
{\mathpzc A}[U]=\left(\begin{array}{cc}0 & -\frac{1}{\mu}\Gb\\ \RT & 0\end{array}\right),
\qquad
{\mathpzc B}[U]=\left(\begin{array}{cc}\eps \T & 0\\ 0 & -\eps \T^*\end{array}\right),
\end{equation}
and we also define ${\mathpzc C}_\alpha[U]$ as
\begin{equation}\label{defC}
{\mathpzc C}_\alpha[U]=\left(\begin{array}{cc}0 & -\frac{1}{\mu}\Ga\\ \frac{1}{\Bo}\Kae & 0\end{array}\right).
\end{equation}
We can now state the main result of this section.
\begin{proposition}\label{propIII-1}
Let $T>0$, $t_0>d/2$ and $N$ be as in Proposition \ref{proplin}. If
$U=(\zeta,\psi)\in E^N_{\sigma,T}$ satisfies (\ref{sectII1}) uniformly on $[0,T]$ 
and solves (\ref{eqI-16nd}) then for all $\alpha\in\N^{d+1}$ with $1\leq \abs{\alpha}\leq N$,
the couple
$U_{(\alpha)}=(\zetaa,\psia)$ solves 
\begin{eqnarray*}
\dt U_{(\alpha)}+{\mathpzc A}[U]U_{(\alpha)}=\eps (R_\alpha,S_\alpha)^T,& &(\abs{\alpha}<N),\\
\dt U_{(\alpha)}+{\mathpzc A}[U]U_{(\alpha)}+{\mathpzc B}[U]U_{(\alpha)}
+{\mathpzc C}_\alpha[U]\Uaa=\eps (R_\alpha,S_\alpha)^T,& &(\abs{\alpha}=N),
\end{eqnarray*}
where $\Uaa=(\zetaa,\psiaa)^T$ and
the residuals $R_\alpha$ and $S_\alpha$ satisfy the estimates
$$
\abs{R_\alpha}_{H^1_\sigma}+\abs{S_\alpha}_{\Hm^{1/2}}\lesssim  \mfm^N(U)\big(1+
\eps\mu^{1/4}\urp\urm \abs{\jump{\uVpm}}_\infty\enz\big),
$$
uniformly on $[0,T]$.
\end{proposition}
\begin{remark}
The ``quasilinear'' system we refer to is the system of evolution equations
on   $(U_{(\alpha)})_{0\leq \abs{\alpha}\leq N}$ (with $U_{(0)}=U$) formed by
 (\ref{eqI-16nd}) for $\alpha=0$ and the
equations of the proposition for $1\leq \abs{\alpha}\leq N$. 
\end{remark}
\begin{proof}
\emph{Throughout this proof, we  denote by $\phi^\pm$ the solution
of the transmission problem (\ref{eqsectII2-1})
provided by Proposition \ref{prop2}, and write $\psi^\pm={\phi^\pm}_\interff$ (or equivalently, we define $\psi^\pm$ by (\ref{manif}))}.

The first equation is obtained directly by applying $\partial^\alpha$ to the first equation of (\ref{eqI-16nd}) and using
Proposition \ref{proplin}. 

For the second equation, we focus on the most difficult case, namely, $\abs{\alpha}=N$. We decompose first $\alpha=\beta+\gamma\in\N^{d+1}$, 
with $\abs{\gamma}=1$. Inspired by \cite{Iguchi}, we apply first $\partial^\gamma$ to the second equation of
(\ref{eqI-16nd}) and use the definition (\ref{defVw}) of $\uVpm$ and $\uwpm$ to get
\begin{eqnarray}
	\nonumber
	\dt \partial^\gamma\psi&+&\partial^\gamma\zeta+\eps
	\jump{\urpm \uVpm\cdot(\nabla\partial^\gamma\psi^\pm-\eps \uwpm\nabla\partial^\gamma\zeta)}\\
	\label{eqIII-5}
	&-&\frac{\eps}{\mu}
	\jump{\urpm\uwpm\partial^\gamma(\Gb\psi)}
        =-\frac{1}{\Bo}\frac{1}{\eps\sqrt{\mu}}\partial^\gamma \mfk(\eps\sqrt{\mu}\zeta).
\end{eqnarray}
We are now going to apply $\partial^\beta$ to this equation. We will thus find many
terms of the form $\partial^\beta (fg)$; in order to check that most of them can be put in
the residual term (i.e. that\footnote{The notation $A\leq \mfm^N(U)$ is used as a shortcut for $A(t)\leq \mfm^N(U(t))$, for all $0\leq t\leq T$} $\abs{\partial^\beta (fg)}_{\Hm^{1/2}}\leq \eps \mfm^N(U)$), we rely heavily on the product estimate (\ref{bofbof}).
 For the sake
of clarity, we also introduce the notation
\begin{equation}\label{notaequiv}
a\sim b\iff 
\abs{a-b}_{\Hm^{1/2}}\leq \eps \mfm^N(U)\big(1 +\eps\mu^{1/4}\urp \urm \abs{\jump{\uVpm}}_\infty\enz\big).
\end{equation}
\begin{lemma}\label{lemdbeta}
Under the assumptions of the proposition, and with $\alpha=\beta+\gamma$ ($\abs{\gamma}=1$), the
following identities hold:
\begin{eqnarray*}
\eps \partial^\beta \big(\urpm \uVpm\cdot(\nabla\partial^\gamma\psi^\pm-\eps \uwpm\nabla\partial^\gamma\zeta)\big)&\sim&\eps\urpm\uVpm\cdot\nabla\psi^\pm_{(\alpha)}+\eps^2\urpm(\uVpm\cdot\nabla\uwpm)\partial^\alpha\zeta\\
& &-\eps^2\urpm\uVpm\cdot\{\partial^\beta,\uwpm\}\nabla\partial^\gamma\zeta,\\
\frac{\eps}{\mu}\partial^\beta
	\big(\urpm\uwpm\partial^\gamma(\Gb\psi)\big)&\sim&\frac{\eps}{\mu}\urpm\uwpm\partial^\alpha(\Gb\psi)+\frac{\eps}{\mu}\urpm\{\partial^\beta,\uwpm\}\partial^\gamma(\Gb\psi),
\end{eqnarray*}
where $\{\partial^\beta,\uwpm\}=\sum_{j=1}^{d+1}\beta_j\partial_j\uwpm\partial^{\check\beta^j}$, 
with $\check\beta^j+\beta_j{\bf e}_j=\beta$, is the usual Poisson bracket.
\end{lemma}
\begin{proof}
Let us address the first assertion of the lemma. Since $\eps \urpm\partial^\beta \big(\uVpm\cdot(\nabla\partial^\gamma\psi^\pm-\eps\uwpm\nabla\partial^\gamma\zeta)\big)$ is a sum of terms of the form $A_{\beta',\beta''}=\eps\urpm\partial^{\beta'}\uVpm\cdot \partial^{\beta''}(\nabla\partial^\gamma\psi^\pm-\eps\uwpm\nabla\partial^\gamma\zeta)$, with $\beta'+\beta''=\beta$, and remarking that
$$
A_{0,\beta}\sim\eps\urpm\uVpm\cdot\nabla\psi^\pm_{(\alpha)}+\eps^2\urpm(\uVpm\cdot\nabla\uwpm)\partial^\alpha\zeta-\eps^2\urpm\uVpm\cdot\{\partial^\beta,\uwpm\}\nabla\partial^\gamma\zeta,
$$
we are led to
prove that $A_{\beta',\beta''}\sim 0$ if $\abs{\beta''}<N-1$. The most difficult 
configuration
 corresponds to $\beta''=0$ or $\abs{\beta''}=N-2$, and we thus omit the proof for
the intermediate cases. Moreover, since the proof follows the same lines for 
$\beta''=0$ and $\abs{\beta''}=N-2$, we just give it in the latter case. By (\ref{bofbof}), we get
$$
\abs{A_{\beta',\beta''}}_{\Hm^{1/2}}\leq \eps\urpm\abs{\partial^{\beta'}\uVpm}_{H^{t_0+1}}
\big(\abs{\partial^{\beta''}g}_2+\abs{\partial^{\beta''}g}_{\Hm^{1/2}}\big),
$$
with $g=\nabla\partial^\gamma\psi^\pm-\eps\uwpm\nabla\partial^\gamma\zeta$. Since
$\abs{\beta'}=1$, we deduce easily from Corollary \ref{coro1} that $ \abs{\partial^{\beta'}\uVpm}_{H^{t_0+1}}\leq \mfm^N(U)$; moreover, $\partial^{\beta''}g$ can be put under the form
$$
\partial^{\beta''}g=\nabla\psi^\pm_{(\beta''+\gamma)}-\eps[\partial^{\beta''},\uwpm]\nabla\partial^\gamma\zeta+\eps(\partial^{\gamma+\beta''}\zeta) \nabla\uwpm.
$$
 Since the last two components of the r.h.s. are bounded from above in $L^2$ and
$\Hm^{1/2}$ by $ \mfm^N(U)$ (this is a consequence of (\ref{bofbof}), Corollary \ref{coro1} and the identity $\abs{f}_{\Hm^{1/2}}\leq \abs{f}_{\dot{H}^1}$), we have obtained that
\begin{eqnarray*}
\abs{A_{\beta',\beta''}}_{\Hm^{1/2}}&\leq& \eps\urpm\mfm^N(U)(1+\abs{\nabla\psi^\pm_{(\beta''+\gamma)}}_2+\abs{\nabla\psi^\pm_{(\beta''+\gamma)}}_{\Hm^{1/2}})\\
&\leq&\eps\mfm^N(U)(1+\sum_{\abs{\kappa}\leq N}\urpm\abs{\psi^\pm_{(\kappa)}}_{\Hm^{1/2}})\\
&\leq & \eps\mfm^N(U)\big(1+
\eps\mu^{1/4}\urp\urm \abs{\jump{\uVpm}}_\infty\enz\big),
\end{eqnarray*}
where Remark \ref{remdax} has been used to derive the last inequality.\\
In order to prove the second assertion of the lemma, we proceed as above to check that it 
is sufficient to prove that $\frac{\eps}{\mu}\urpm\partial^{\beta'}\uwpm\partial^{\beta''+\gamma}(\Gb\psi)\sim 0$ if $\beta'+\beta''=\beta$ and $\abs{\beta''}<N-2$. We give the proof of the most difficult case, corresponding here to $\beta'=\beta$, $\beta''=0$. Thanks to (\ref{bofbof}),
Corollaries \ref{coro1} and \ref{coro2} and Lemma \ref{lemmformdergen}, it is enough to prove that
\begin{equation}\label{onemore}
\frac{\urpm}{\sqrt{\mu}}\abs{\partial^\beta \uwpm}_{\Hm^{1/2}}\leq  \mfm^N(U)\big(1+\eps\mu^{1/4} \urp\urm \abs{\jump{\uVpm}}_\infty \enz\big).
\end{equation}
Recalling that $\uwpm=\cw^\pm[\zeta]\psi$ and differentiating this relation with respect 
to $\partial^\beta$ yields after multiplication by $\eps$ (in order to use the convenient notation $\sim$)
\begin{eqnarray*}
\lefteqn{(1+\eps^2\mu\abs{\nabla\zeta}^2)\frac{\eps}{\sqrt{\mu}}\partial^\beta\uwpm
+2\eps^2\sqrt{\mu}\nabla\partial^\beta\zeta\cdot(\eps\uwpm\nabla\zeta)}\\
&\sim&\frac{\eps}{\sqrt{\mu^+}}\partial^\beta \Gpmb\psi^\pm
+\eps^2\sqrt{\mu}\nabla\zeta\cdot\partial^\beta\nabla\psi^\pm+
\eps^2\sqrt{\mu}\nabla\partial^\beta\zeta\cdot\nabla\psi^\pm.
\end{eqnarray*}
Since moreover, one can deduce from (\ref{formuleder}), (\ref{estder0}) and (\ref{estder0bis}) that
$$
\frac{\eps}{\sqrt{\mu^+}}\partial^\beta \Gpmb\psi^\pm\sim\frac{\eps}{\sqrt{\mu^+}}\Gpmb\psi^\pm_{(\beta)}
-\eps^2\sqrt{\mu}\nabla\cdot(\partial^\beta \zeta \uVpm),
$$
we get the formula (after multiplication by $\urpm$)
$$
(1+\eps^2\mu\abs{\nabla\zeta}^2)\frac{\eps\urpm}{\sqrt{\mu}}\partial^\beta \uwpm\sim \frac{\eps}{\sqrt{\mu^+}}\Gpmb(\urpm\psi^\pm_{(\beta)})+\eps^2\sqrt{\mu}\nabla\zeta\cdot\nabla
(\urpm \psi^\pm_{(\beta)}).
$$
Dividing this formula by $\eps$, it is then easy to deduce 
from Remark \ref{remdax}  (and (\ref{eqsectII1-4sam}) to handle the
first term of the r.h.s.) that  (\ref{onemore}) holds.
\end{proof}

Lemma \ref{lemdbeta} allows one to put most of the terms of the second equation of
$\partial^\beta$(\ref{eqIII-5}) in the residual. However, this is not the case of 
the Poisson brackets that appear in both 
identities of Lemma \ref{lemdbeta} because they are not \emph{uniformly} controlled by the energy, since they require (when evaluated in $\Hm^{1/2}$-norm) a control of $(N+1/2)$ derivatives of $\zeta$. Using the control of the derivatives of order $(N+1)$ provided by $\abs{\zeta}_{H^{1,N}_\sigma}$ would induce a singularity when the surface tension goes to zero. Fortunately, this singularity disappears when one takes the difference of the expressions considered in Lemma \ref{lemdbeta}, which is the quantity that appears in $\partial^\beta$(\ref{eqIII-5}). Indeed, one has
\begin{eqnarray*}
\lefteqn{-\eps^2\urpm\uVpm\cdot\pb{\partial^\beta,\uwpm}\nabla\partial^\gamma\zeta-\frac{\eps}{\mu}\urpm\pb{\partial^\beta,\uwpm}\partial^\gamma(\Gb\psi)}\\
&\sim&-\eps\urpm \pb{\partial^\beta,\uwpm}\big(\eps\uVpm\cdot\nabla\partial^\gamma\zeta+\frac{1}{\mu}\partial^\gamma(\Gb\psi)\big)\\
&\sim&0,
\end{eqnarray*}
where (\ref{formuleder}) has been used to derive the second identity.\\
It follows from this analysis that $\partial^\beta$(\ref{eqIII-5}) can be written under the form
\begin{eqnarray*}
	\nonumber
	\dt \partial^\alpha\psi&+&\partial^\alpha\zeta+\eps
	\jump{\urpm \uVpm\cdot(\nabla\psiapm+\eps \partial^\alpha\zeta\nabla\uwpm) }\\
	\nonumber
	&-&\frac{\eps}{\mu}
	\jump{\urpm\uwpm\partial^\alpha(\Gb\psi)}
	\sim-\frac{1}{\Bo}\frac{1}{\eps\sqrt{\mu}}\partial^\alpha \mfk(\eps\sqrt{\mu}\zeta).
\end{eqnarray*}
Using the first equation, one can replace $\frac{1}{\mu}\Gb\psi$ by $\dt \zeta$ in this expression,
\begin{eqnarray*}
	\nonumber
	\dt \psia+\mfa\partial^\alpha\zeta+
	\eps\jump{\urpm \uVpm \cdot \nabla\psiapm}
	\sim-\frac{1}{\Bo}\frac{1}{\eps\sqrt{\mu}}\partial^\alpha \mfk(\eps\sqrt{\mu}\zeta),
\end{eqnarray*}
with $\mfa$ as in Definition \ref{defiRT}.
We can then use the identity
$$
	\jump{f^\pm g^\pm}=\av{f^\pm}\jump{g^\pm}+\jump{f^\pm}\av{g^\pm}
$$
to obtain
\begin{equation}
\dt \psi_{(\alpha)}+\mfa \partial^\alpha\zeta+\eps\av{\uVpm}\cdot\nabla\psia
\label{eqIII-6}
+\eps\jump{V^\pm}\cdot\nabla{\av{\urpm\psiapm}}
\sim-\frac{1}{\Bo}\frac{\partial^\alpha \mfk(\eps\sqrt{\mu}\zeta)}{\eps\sqrt{\mu}}.
\end{equation}
The quantity $\nabla \av{\urpm\psiapm}$ depends on $\psi_{(\alpha)}$ and $\partial^\alpha\zeta$ and on their
first order derivatives. This dependence is trivial in the water waves case since $\av{\urpm\psiapm}=\frac{1}{2}\psia$; when $\urm\neq 0$, it is responsible for the Kelvin-Helmholtz 
instabilities. The following lemma makes this dependence precise.
\begin{lemma}
One has, with $\E$ as defined in (\ref{eqIII-3}),
\begin{eqnarray*}
\eps\jump{\uVpm}\cdot\nabla\av{\urpm\psiapm}&\sim&
\frac{\eps}{2}\jump{\uVpm}\cdot\nabla\big[ (1+2\urm\uH^-(\Gmb)^{-1}\circ\Gb)\psia\big]\\
& &-\eps^2\mu\urp\urm \jump{\uVpm}\cdot\E(\jump{\uVpm}\partial^\alpha\zeta).
\end{eqnarray*}
\end{lemma}
\begin{proof}
From Proposition \ref{proptechnical}, we get
$$
\frac{1}{\uH^\pm}\psi^\pm_{(\alpha)}= (\Gpmb)^{-1}\Gb \psia-\underline{\rho}^\mp \eps\mu \uH^\mp  (\Gpmb)^{-1}\Gb(\Gb^\mp)^{-1}\nabla\cdot (\partial^\alpha\zeta \jump{\uVpm})\\
+(\Gpmb)^{-1}r^\pm_\alpha,
$$
We deduce from this expression that
\begin{eqnarray*}
\lefteqn{\eps\jump{\uVpm}\cdot\nabla\av{\urpm\psiapm}=
\frac{\eps}{2}\jump{\uVpm}\cdot\nabla\big[ (1+2\urm\uH^-(\Gmb)^{-1}\circ\Gb)\psia\big]}\\
& &-\eps^2\mu\urp\urm \jump{\uVpm}\cdot\E(\jump{\uVpm}\partial^\alpha\zeta)+\eps\jump{\uVpm}\cdot\nabla \av{\urpm\uH^\pm(\Gpmb)^{-1}r^\pm_\alpha }.
\end{eqnarray*}
The lemma follows from the fact that $\eps\jump{\uVpm}\cdot\nabla \av{\urpm\uH^\pm(\Gpmb)^{-1}r^\pm_\alpha }\sim 0$, which is itself a direct consequence of the bounds on 
$r_\alpha^\pm$ given in  Proposition \ref{proptechnical}.
\end{proof}
Thanks to the lemma, we can write (\ref{eqIII-6}) under the form
\begin{eqnarray*}
\dt \psia-\eps \T^*\psia&+&\mfa\zetaa-\eps^2\mu\urp\urm\jump{\uVpm}\cdot\E(\zetaa\jump{\uVpm})\\
&\sim&-\frac{1}{\Bo}\frac{1}{\eps\sqrt{\mu}}\partial^\alpha \mfk(\eps\sqrt{\mu}\zeta).
\end{eqnarray*}
The proposition is thus a direct consequence of the observation that
$$
\frac{1}{\Bo}\frac{1}{\eps\sqrt{\mu}}\partial^\alpha \mfk(\eps\sqrt{\mu}\zeta) \sim
-\frac{1}{\Bo}\nabla\cdot \K[\eps\sqrt{\mu}\nabla\zeta]\nabla\partial^\alpha\zeta+\frac{1}{\Bo}\Kae\zetaaa.
$$
\end{proof}

\section{Main results}\label{sectmain}

\subsection{The stability criterion}\label{sectSC}

\subsubsection{A first criterion}\label{sectSC1}

We state and comment in this section the stability criterion
that ensures the existence of a ``stable'' solution, that is, of a solution that exists
on a time scale consistent with physical observations. 

Before stating this criterion, let us recall that 
$\mfe(\zeta)$ is defined in Proposition \ref{propE}, as
$$
\mfe(\zeta)=\sup_{V\in H^{1/2}(\R^d)^d,V\neq 0}\mu\frac{\big(\widetilde{\Gb}^{-1}\nabla\cdot V,\nabla\cdot V\big)}{\abs{(1+\sqrt{\mu}\abs{D})^{1/2}V}_2^2},
$$
with $\widetilde{\Gb}=\urm\frac{1}{\uH^+}\Gpb-\urp\frac{1}{\uH^-}\Gmb$. We also define
${\mathfrak c}(\zeta)$ as
$$
{\mathfrak c}(\zeta)=\mfe(\zeta)^2(1+\eps^2\mu\abs{\nabla\zeta}_\infty^2)^{3/2}
$$
and recall that $\dsp \Upsilon=(\urp\urm)^2\frac{a^4}{H^2}\frac{(\rho^++\rho^-)g'}{4\sigma}$.

The (nondimensionalized) stability criterion can then be stated as
\begin{equation}\label{SC}
\Upsilon{\mathfrak c}(\zeta)
\max_{\abs{\alpha}\leq 1} \babs{\partial^\alpha\jump{\uVpm}}_{\infty}^4<\inf_{\R^d}\mfa,
\end{equation}
 with $\mfa=1+\eps\jump{\urpm\big(\dt+\eps\uVpm\cdot\nabla\big)\uwpm}$
and where we recall that $\partial^\alpha$ can be either a time or space derivative.\\
In order to measure by which extent this criterion is satisfied, we introduce the
function ${\mathfrak d}(U)$ defined as
\begin{equation}\label{defd}
{\mathfrak d}(U)=\inf_{\R^d}\mfa-\Upsilon{\mathfrak c}(\zeta)
\max_{\abs{\alpha}\leq 1} \babs{\partial^\alpha \jump{\uVpm}}_{\infty}^4;
\end{equation}
quite intuitively, we expect the solutions to be more ``stable'' for large values of 
${\mathfrak d}(U)$ than for smaller ones.

\begin{remark}\label{constexpl}
In the flat case $\zeta=0$, one can check that 
$$
{\mathfrak c}(0)=\sup_{x\geq 0}\frac{x}{(1+x)(\urm\tanh(\uH^+x)+\urp \tanh(\uH^- x))}.
$$
In the general case, it is possible to obtain an upper bound on ${\mathfrak c}(\zeta)$
by tracking the constants in Proposition \ref{propE}. The most important information,
however, is that ${\mathfrak c}(\zeta)\sim 1$ for all the possible
values of $\eps$ and $\mu$.
\end{remark}

\subsubsection{An alternative version}\label{sectSC2}

We show in Theorem \ref{theomain} below that (\ref{SC}) ensures the existence of a ``stable''
solution if the fluid depths do not vanish. If we assume moreover that the jump of the horizontal
velocity is not identically zero at $t=0$ (which is the general configuration), 
it is possible to give a more elegant
version of the stability criterion, namely, 
\begin{equation}\label{SCalt}
\tag{\ref{SC}'}
\Upsilon{\mathfrak c}(\zeta)\,
\babs{\jump{\uVpm}}_{\infty}^4<\inf_{\R^d}\mfa,
\end{equation}
Consequently, one replaces ${\mathfrak d}(U)$ in (\ref{defd}) by 
\begin{equation}\label{defdalt}
\tag{\ref{defd}'}
{\mathfrak d}'(U)=\inf_{\R^d}\mfa-\Upsilon{\mathfrak c}(\zeta)
\, \babs{\jump{\uVpm}}_{\infty}^4.
\end{equation}
\begin{remark}
It is of course possible to give a version with dimensions 
of the stability criterion (\ref{SCalt}). Remarking that 
$(\rho^++\rho^-)g'\mfa= \jump{\dz P^\pm\,_{\vert_{z=\zeta}}}$, where $P^\pm$
is the pressure in the Euler equations (\ref{eqI-2}), we obtain
$$
\jump{-\dz P^\pm\,_{\vert_{z=\zeta}}}> \frac{1}{4}\frac{(\rho^+\rho^-)^2}{\sigma(\rho^++\rho^-)^2}
{\mathfrak c}(\zeta)\,\abs{\omega}_\infty^4,
$$
where $\omega=\jump{V^\pm\,_{\vert_{z=\zeta}}}$.
This is the criterion stated in (\ref{SCintro}) in the introduction.
\end{remark}

\begin{remark}\label{remWW}
In the water waves case ($\urm=0$), (\ref{SC}) and (\ref{SCalt}) reduce to the condition
$\inf_{\R^d}\mfa >0$ (or to the Taylor condition $-\dz P_{\vert_{z=\zeta}}>0$ in physical variables). This condition is the same with or without surface tension\footnote{with the convention that $(\urm)^2/\sigma=0$ if $\urm=0$ and $\sigma=0$}. This condition is usually not imposed in the
literature when surface tension is present\footnote{but it is of course imposed in the references considering the zero surface tension limit, e.g.  \cite{AM2,AM3,MingZhang,Pusateri2}}; the reason why it appears here is that we are
only interested in ``stable'' solutions. In the water waves case, this amounts to discarding
configurations where the water is \emph{above} the air; though there exists a local solution
thanks to surface tension, it is obviously not ``stable'' 
in the intuitive meaning of this word.
\end{remark}

\subsubsection{A practical criterion}\label{sectSC3}

We end this section with the derivation of a ``practical criterion'' that can be used
in (shallow water) applications to have an \emph{a priori} idea of the stability
of an interfacial wave if we are given its physical characteristics 
(amplitude and wavelength).\\
The nondimensionalization performed in Appendix \ref{apnd} (and supported by the
experimental data of \cite{Grue} for instance) is such that $\uVpm$ is roughly
of size one (this remains true as long as the dynamics does not depart too much from the
linear analysis, which is the case for ``stable'' waves until their breaking point).
It follows that $\mfa$ is roughly of size one and that a practical criterion can
be stated as
\begin{equation}\label{practSC}
\Upsilon\ll 1: \mbox{ Stable configuration; }
\quad
\Upsilon\gg 1: \mbox{ Unstable configuration.}
\end{equation}
When $\Upsilon\sim 1$, a refined analysis is of course needed and the full criterion (\ref{SC}) should be used.
\begin{remark}
It seems that the wavelength does not play any role in (\ref{practSC}). Its contribution
is however hidden in the assumption that $\uVpm$ is of size one, which is only
true in shallow water, i.e. when $H^2/\lambda^2\ll 1$.
\end{remark}
\subsection{Initial conditions}\label{sectinit}

The main step in the proof of the existence of stable solutions is to prove that
the energy 
$\cE^N_\sigma(U)$ defined in (\ref{eqenergy}), namely,
$$
{\mathcal E}^N_\sigma(U)=\abs{\nabla \psi}^2_{H^{t_0+2}}
+\sum_{\alpha\in \N^{d+1},\abs{\alpha}\leq N}\abs{\zetaa}_{H^1_\sigma}^2+\abs{\psi_{(\alpha)}}_{\Hm^{1/2}}^2,
$$
is controlled for positive times by the energy at $t=0$. Since $\cE^N_\sigma(U)$ 
involves time
derivatives, we have to specify which sense we give to the initial energy $\cE^N_\sigma(U^0)$.
If we denote, for all $\alpha\in \N^{d+1}$, 
$\alpha=(\alpha_0,\alpha_1,\dots,\alpha_d)$ so that
$\partial_t^{\alpha_0}$ corresponds to the time derivatives of $\partial^\alpha$, the problem
is to choose initial values $\Ua^0$ for $(\Ua)_{\vert_{t=0}}$ (with $\Ua=(\zetaa,\psia)$) 
when $\alpha_0>0$, in terms of $U^0$ and its spacial derivatives.\\
This can be done by 
a finite induction. When $\alpha_0=0$ (no time derivative), we take of course
$$
\Ua^{0}=\big(\partial^\alpha\zeta^0,\partial^\alpha\psi^0-\eps\uw^0\partial^\alpha\zeta^0\big)^T
\quad \mbox{ with }\quad
\uw^0=\jump{\urpm \cw^\pm[\eps\zeta^0]\psi^0}.
$$
Now, let $1\leq n\leq N$ and assume that $U_{(\beta)}\,_{\vert_{t=0}}=U^{0}_{(\beta)}$ has been chosen for all $\beta\in \N^{d+1}$ with $\beta_0<n$. We remark that for all $\alpha$ with $\alpha_0=n$
we have
$$
U_{(\alpha)}\,_{\vert_{t=0}}=\big(\dt \zeta_{(\alpha')},\dt \psi_{(\alpha')}+\eps\dt\uw\partial^{\alpha'}\zeta\big)_{\vert_{t=0}},
$$
with $\alpha'=(\alpha_0-1,\alpha_1,\dots,\alpha_d)$, and we are therefore led to set initial
conditions for $\dt U_{(\alpha')}$, which is achieved by using Proposition \ref{propIII-1}.

The initial energy, which we denote slightly abusively by $\cE^N_\sigma(U^0)$ is therefore
defined as
\begin{equation}\label{initNRJ}
{\mathcal E}^N_\sigma(U^0)=\abs{\nabla \psi^0}^2_{H^{t_0+2}}
+\sum_{\alpha\in \N^{d+1},\abs{\alpha}\leq N}\abs{\zetaa^0}_{H^1_\sigma}^2
+\abs{\psi_{(\alpha)}^0}_{\Hm^{1/2}}^2,
\end{equation}
with $\Ua^0$ as constructed above.
\begin{remark}
One proceeds in the same way to give a formulation of the stability criterions (\ref{SC})
or (\ref{SCalt})
at time $t=0$ in terms of $U^0$ and its spacial derivatives.
\end{remark}

\subsection{Local existence of ``stable'' solutions}\label{sectlocal}

We state here a theorem ensuring that the interfacial waves equations (\ref{eqI-16nd})
admit ``stable'' solutions if the depth of both fluid layers does not vanish
for the initial condition,
and if the stability criterion (\ref{SC}) or (\ref{SCalt}) 
is satisfied. By ``stable'' solution, we mean two things. \\
Firstly, the existence time must not shrink to $0$ as $\sigma\to 0$. More precisely,
the existence time is not measured by the size of the surface tension coefficient
$\sigma$ but by the size of ${\mathfrak d}(U)$.\\
Secondly, the existence time
must not vanish for ``acceptable'' values of the physical parameters $\eps$ and $\mu$, namely $0\leq \mu\leq 1$ and $0\leq \eps\leq 1$.

Recalling that the quantity $\cE^N(U^0)$ has been defined in (\ref{initNRJ}) and 
$E^N_{\sigma,T}$ in (\ref{eqensp}),
we can now state the following theorem (whose proof is given in \S \ref{sectproofth}).
\begin{theorem}\label{theomain}
Let $t_0=3/2$, $N\geq 5$ and  $U^0=(\zeta^0,\psi^0)^T\in L^2(\R^d)\times \dot{H}^{1/2}(\R^d)$ be such that
$\cE^N_\sigma(U^0)<\infty$ and that satisfies the nonvanishing depth condition
$$
\exists h^\pm_{min}>0,\qquad \inf_{X\in\R^d}(1\pm\eps^\pm \zeta^{0}(X))\geq h^\pm_{min}.
$$
If moreover $U^0$ satisfies the stability criterion (\ref{SC})
 then there exists $T>0$ and a unique solution
$U\in E^N_{\sigma,T}$ to (\ref{eqI-16nd}) with initial condition $U^0$. Moreover,
$$
\forall t\in [0,T],\qquad \cE^N_\sigma(U(t))\leq C(T)
$$
and $T$ depends only on 
 $\mfm^N(U^0)$ and  $\dsp{\mathfrak d}(U^0)$ (as defined in (\ref{eqmfm}) and (\ref{defd})
respectively).
\end{theorem}
If the jump of the horizontal component of the initial velocity is nonzero, 
one can replace the stability criterion (\ref{SC}) by its alternative version (\ref{SCalt}),
and get the following theorem.
\begin{theoremp}
Let $t_0=3/2$, $N\geq 5$ and  $U^0=(\zeta^0,\psi^0)^T\in L^2(\R^d)\times \dot{H}^{1/2}(\R^d)$ be such that
$\cE^N_\sigma(U^0)<\infty$ and that satisfies the nonvanishing depth condition
$$
\exists h^\pm_{min}>0,\qquad \inf_{X\in\R^d}(1\pm\eps^\pm \zeta^{0}(X))\geq h^\pm_{min},
$$
together with the condition $\omega^0=\jump{\underline{V}^{0,\pm}}\neq 0$.
If moreover $U^0$ satisfies the stability criterion (\ref{SCalt}) 
then the conclusion of Theorem \ref{theomain} still holds, with $T$ depending on
 $\mfm^N(U^0)$, $\dsp{\mathfrak d}'(U^0)$ and $\dsp\frac{\abs{\dt\omega^0}_\infty+\abs{{\nabla\omega^0}}_\infty}{\abs{{\omega^0}}_\infty}$.
\end{theoremp}
\begin{remark}
\item[(1)] The existence time also depends implicitly $\frac{H^\pm}{H^\mp}$; since we address
here physical configurations where $H^+$ and $H^-$ are of the same order\footnote{As said in the introduction, this assumption is made for the sake of clarity and the method presented here could be adapted to other configurations}, 
this dependence
is harmless.
\item[(2)] There is also an implicit dependence of $T$ on $1/\Bo$. 
Since we are not interested in situations where capillary forces dominate the effect of
gravity (e.g. droplets), the Bond number is always larger than one 
and the dependence on $1/\Bo$ is harmless.
\item[(3)] It is straightforward to deduce from Theorem \ref{theomain} some convergence results
as $\urm$ and/or $\sigma$ go to zero. For instance, convergence to the water waves equations 
without surface tension is obtained when $\urm$ and $\sigma$ go to zero with 
$\Upsilon\to 0$ -- and in particular if $(\rho^-)^2/\sigma\to 0$, a result established 
in \cite{Pusateri2} for a liquid droplet with rotational effects under the slightly stronger
condition $(\rho^-)^2\leq \sigma^{7/3}$. Theorem \ref{theomain} shows moreover that it is possible to take the
shallow water limit at the same time, which justifies the use of the standard (one fluid)
shallow water
models for the air/water interface (see Section \ref{sectairwater} for more details).
\end{remark}

\subsection{Persistence of ``stable'' solutions over large  times}\label{sectpersist}

We have seen in Remark \ref{remWW}
that in the water waves case without surface tension, the stability
criterion (\ref{SC}) reduces to the Rayleigh-Taylor criterion $\inf_{\R^d}\mfa >0$.
If this condition is satisfied, we know from \cite{AL} that the solution $U$ exists
for a time scale $O(1/\eps)$. When $\eps\sim 1$, this results coincides with
Theorem \ref{theomain}, but when $\eps\ll 1$, Theorem \ref{theomain} is weaker
than \cite{AL} since it provides an existence time of order $O(1)$ only.

We propose in this section a stronger version of the stability criterion (\ref{SC})
(one can adapt (\ref{SCalt}) in the same way)
that allows us to show that the solution furnished by Theorem \ref{theomain}
persists over larger times when $\eps\ll1$. This stronger criterion can be stated as
\begin{equation}\label{SCstrong}
\eps^{-2\gamma}\Upsilon^2{\mathfrak c}(\zeta)
\max_{\abs{\alpha}\leq 1} \babs{\partial^\alpha\jump{\uVpm}}_{\infty}^4<\inf_{\R^d}\mfa,
\quad\mbox{ with }\quad 0\leq \gamma\leq 1
\end{equation}
(in the case $\gamma=0$, (\ref{SCstrong}) coincides of course with (\ref{SC})).
\begin{theorem}\label{theomain2}
Let $t_0=3/2$, $N\geq 5$ and  $U^0=(\zeta^0,\psi^0)^T\in L^2(\R^d)\times \dot{H}^{1/2}(\R^d)$ be such that
$\cE^N_\sigma(U^0)<\infty$ and that satisfies the nonvanishing depth condition
$$
\exists h^\pm_{min}>0,\qquad \inf_{X\in\R^d}(1\pm\eps^\pm \zeta^{0}(X))\geq h^\pm_{min}.
$$
If moreover $U^0$ satisfies the strong stability criterion (\ref{SCstrong}) for some $\gamma\in [0,1]$, then there exists $T>0$ and a unique solution
$U\in E^N_{\sigma,\eps^{-\gamma}T}$ to (\ref{eqI-16nd}) with initial condition $U^0$. Moreover,
$$
\forall t\in [0,\eps^{-\gamma}T],\qquad \cE^N_\sigma(U(t))\leq C(T)
$$
and $T$ depends only on $\mfm^N(U^0)$ and  $\dsp{\mathfrak d}(U^0)$. 
\end{theorem}
\begin{remark}\label{rem2pareil}
If (\ref{SCstrong}) is satisfied with $\gamma=1$, we recover the same time scale
as the one provided by \cite{AL} for water waves. We refer to Section \ref{sectappl}
for further discussion.
\end{remark}
\begin{remark}\label{rempractstrong}
The practical criterion corresponding to (\ref{SCstrong}) is obtained by
replacing $\Upsilon$ by $\eps^{-2\gamma}\Upsilon$ in (\ref{practSC}).
\end{remark}

\subsection{Proof of Theorems \ref{theomain}, \ref{theomain}' and \ref{theomain2}}\label{sectproofth}

We give here the proof of Theorems \ref{theomain} and \ref{theomain2}. The proof of
Theorem \ref{theomain}' requires only minor adaptations with respect to the proof
of Theorem \ref{theomain}; we briefly indicate them in footnotes.

\subsubsection{Preliminary results}

We construct below a symmetrizer for the quasilinear system (\ref{eqam}).
It is of crucial importance that the energy norm associated to this symmetrizer
(see (\ref{ensym}) below) be (uniformly) equivalent to the energy $\cE^N_\sigma(U)$
previously introduced in (\ref{eqenergy}). A first ingredient is the coercivity property
of $\frac{1}{\mu}\Gb$ established in Proposition \ref{propG}. The second ingredient is
the following lemma\footnote{\label{footalt}If one replaces (\ref{SC}) by (\ref{SCalt}) then the estimate
of the third point only holds for $\alpha=0$. In the case $\abs{\alpha}=1$, we must
therefore replace the estimate of the lemma by
$$
\eps^2\sqrt{\mu}\urp\urm\mfe(\zeta)\babs{\partial^\alpha\jump{\uVpm}}^2_\infty\abs{u}_{\langle1/2\rangle}^2\leq \big(\frac{\abs{\partial^\alpha\jump{\uVpm}}_\infty}{\abs{\jump{\uVpm}}_\infty}
\big)^2\mfm^1(U)\abs{u}_{H^1_\sigma}.
$$
 } that shows that, up to lower order terms that can be controlled and under the
stability condition (\ref{SC}),  
$(\RT u,u)$ is uniformly 
equivalent to $\abs{u}^2_{H^{1}_\sigma}$.
\begin{lemma}\label{lemmequive}
Let $T>0$, $t_0>d/2$ and
$U=(\zeta,\psi)\in E^N_{\sigma,T}$ be such that 
(\ref{sectII1}) and (\ref{SC}) are uniformly satisfied on 
$[0,T]$.
Then the following identities hold uniformly on $[0,T]$,
\item[(1)] For all $u\in H^1_\sigma(\R^d)$, and with $\mfm^1(U)$ as defined in (\ref{eqmfm}),
one has
$$
(\RT u,u)\leq \mfm^1(U) \abs{u}_{H^1_\sigma}^2.
$$
\item[(2)] There exists a numerical constant $c_2$ 
such that 
\begin{eqnarray*}
\frac{1}{2}{\mathfrak d}(U)\abs{u}_{H^1_\sigma}\leq (\RT u, u)+a(U)\abs{u}_2^2+b(U)\abs{u}_{H^{-1/2}}^2,
\end{eqnarray*}
with $a(U)=\eps^2\urp\urm\mfe(\zeta)\babs{\jump{\uVpm}}_\infty^2$ and 
$b(U)= \eps^2\sqrt{\mu}\urp\urm c_2\mfe(\zeta)\babs{\jump{\uVpm}}^2_{H^{t_0+1}}$,
and ${\mathfrak d}(U)$ as in (\ref{defd}).
\item[(3)] With $\abs{\cdot}_{\langle 1/2\rangle}$ as defined in (\ref{enzeta}), one has
$$
\eps^2\sqrt{\mu}\urp\urm \max_{\abs{\alpha}\leq 1}\babs{\partial^\alpha\jump{\uVpm}}^2_\infty\abs{u}_{\langle1/2\rangle}^2\leq
\mfm^1(U)\abs{u}_{H^1_\sigma}.
$$
\end{lemma}
\begin{proof}
 By definition of $\RT$ (see Definition \ref{defiRT}), one has
$$
(\RT u,u)=(\mfa u,u)-\eps^2\mu\urp\urm (\Eb (u\jump{V^\pm}),u\jump{V^\pm})
+\frac{1}{\Bo}(\K\nabla u,\nabla u).
$$
In order to give some control on the three components of this expression, we need the
following straightforward estimates,
$$
\begin{array}{ccccc}
\dsp (\inf_{\R^d}\mfa)\abs{u}_2^2 &\leq & \dsp(\mfa u,u)&\leq& \mfm^1(U) 
\abs{u}_2^2\vspace{1mm}\\
 \dsp \frac{1}{(1+\eps^2\mu\abs{\nabla\zeta}_\infty^2)^{3/2}}\abs{\nabla u}_2^2&\leq&
\dsp(\K[\eps\sqrt{\mu}\zeta] \nabla u,\nabla u)&\leq& \dsp \abs{\nabla u}^2_2.
\end{array}
$$
and, using Proposition \ref{propE}, 
\begin{eqnarray*}
\lefteqn{0\leq (\Eb (u\jump{\uVpm}),u\jump{\uVpm})\leq\frac{\mfe(\zeta)}{\mu}\babs{(1+\sqrt{\mu}\abs{D})^{1/2}(u\jump{\uVpm})}_2^2}\\ 
&\leq& \frac{\mfe(\zeta)}{\sqrt{\mu}}\abs{\jump{\uVpm}}^2_\infty\babs{\abs{D}^{1/2}u}_2^2
+\frac{\mfe(\zeta)}{\mu}\abs{\jump{\uVpm}}^2_\infty\abs{u}^2_2
+\frac{c_2\mfe(\zeta)}{\sqrt{\mu}}\abs{\jump{\uVpm}}_{H^{t_0+1}}^2\abs{u}_{H^{-1/2}}^2,
\end{eqnarray*}
where we used the estimate 
\begin{eqnarray}\nonumber
\abs{(1+\mu^{1/4}\abs{D}^{1/2})(u\jump{\uVpm})}_2&\leq& \sqrt{c_2}\mu^{1/4}\abs{\jump{\uVpm}}_{H^{t_0+1}}\abs{u}_{H^{-1/2}}+\abs{\jump{\uVpm}}_\infty\abs{u}_{2}\\
\label{cubaonemore}
& &+\mu^{1/4}\abs{\jump{\uVpm}}_\infty\babs{\abs{D}^{1/2} u}_2
\end{eqnarray}
 to derive the last inequality (here $c_2$ is a numerical constant).\\
The upper bound  on $(\RT u,u)$ given in the lemma is a direct consequence of these
inequalities.
For the lower bound, remark that they also  yield
$$
(\RT u,u)\geq A- a(U) \abs{u}_2^2
-b(U)\abs{u}_{H^{-1/2}}^2,
$$
with 
\begin{eqnarray*}
A&=&(\inf_{\R^d}\mfa)\abs{u}_2^2-\sqrt{\mu} a(U) \abs{u}_{\langle 1/2\rangle}^2+\frac{1}{\Bo}\frac{\abs{\nabla u}_2^2}{(1+\eps^2\mu\abs{\nabla\zeta}_\infty^2)^{3/2}}\\
&=&\int_{\R^d}\Big((\inf_{\R^d}\mfa)-\sqrt{\mu} a(U)\abs{\xi}   +\frac{1}{\Bo}\frac{\abs{\xi}^2}{(1+\eps^2\mu\abs{\nabla\zeta}_\infty^2)^{3/2}}   \Big)
\abs{\widehat{u}(\xi)}^2d\xi.
\end{eqnarray*}
Basic calculus shows that under the stability condition (\ref{SC}), one has $2 A\geq {\mathfrak d}(U) \abs{u}_{H^1_\sigma}^2$, and the result follows. Remarking that $A$ remains positive
if one replaces $\abs{\jump{\uVpm}}_\infty$ by 
$\max_{\abs{\alpha}\leq 1}\babs{\partial^\alpha\jump{\uVpm}}_\infty$ in $a(U)$, 
one gets the last point of lemma as the first one.
\end{proof}

\subsubsection{The mollified quasilinear system}

Let $\chi:\R\to \R$ be a smooth, compactly supported function equal to one in a neighborhood of the origin.
For all $\iota>0$, we denote by $J^\iota$ the mollifier $J^\iota=\chi(\iota \abs{D})$. Let us consider the following
mollified version of the two-fluid wave equations (\ref{eqI-16nd}),
\begin{equation}\label{eqmol}
\left\lbrace
\begin{array}{lcl}
\dsp \dt \zeta - \frac{1}{\mu}J^\iota\Gb\psi=0,\\
\dsp \dt \psi+J^\iota\zeta +
\dsp \eps\frac{1}{2}J^\iota\Big(\jump{\urpm\babs{\nabla\psi^\pm}^2}
- \frac{1}{\mu}(1+\eps^2\mu\abs{\nabla\zeta}^2)\jump{\urpm(\uwpm)^2}\Big)\\
\qquad\qquad\qquad\dsp = -\frac{1}{\Bo} \frac{1}{\eps\sqrt{\mu}}J^\iota\mfk(\eps\sqrt{\mu}\zeta).
\end{array}\right.
\end{equation}
From standard results on ODEs, (\ref{eqmol}) has a unique maximal solution $U^\iota=(\zeta^\iota,\psi^\iota)$  with initial
data $(\zeta^0,\psi^0)$ on a time interval $[0,T_{max}^\iota)$. Proceeding exactly as for the proof of Proposition \ref{propIII-1}, one can check that for all $\alpha\in\N^{d+1}$, $1\leq \abs{\alpha}\leq N$, the quantity
$U^\iota_{(\alpha)}=(\zetaa^\iota,\psia^\iota)$ (with $\zetaa^\iota=\partial^\alpha\zeta^\iota$ and $\psia^\iota=\partial^\alpha\psi^\iota-\eps\jump{\urpm\cw^\pm[\eps\zeta^\iota]\psi^\iota} \partial^\alpha\zeta^\iota$) solves\footnote{for the sake of clarity, we write $\zetaa$ and $\psia$ instead of $\zetaa^\iota$ and $\psia^\iota$ when no confusion is possible.}
\begin{equation}\label{eqam}
\dt U_{(\alpha)}+J^\iota{\mathpzc A}[U]U_{(\alpha)}+J^\iota{\mathpzc B}[U]U_{(\alpha)}
+J^\iota{\mathpzc C}_\alpha[U]\Uaa=\eps (J^\iota R_\alpha,J^\iota S_\alpha+S'_\alpha)^T
\end{equation}
(as in Proposition \ref{propIII-1}, the operators ${\mathpzc A}[U]$ and ${\mathpzc C}_\alpha[U]$ must be
removed if $\abs{\alpha}<N$), where $R_\alpha$ and $S_\alpha$ satisfy 
the same estimates as in Proposition \ref{propIII-1},
and where $S_\alpha'$ is given by
$$
S_\alpha'=-(1-J^\iota)\big[\jump{\urpm \dt\uwpm}\zetaa\big].
$$
Our strategy is to derive energy estimates for the system of evolution equations on $(U^\iota_{(\alpha)})_{0\leq\abs{\alpha}\leq N}$ formed by (\ref{eqmol}) and (\ref{eqam}). After showing that these
energy estimates do not depend on the mollifying parameter $\iota$, we will construct by
compactness a solution $(U_{(\alpha)})_{0\leq \abs{\alpha}\leq N}$ to the limit system corresponding
to $\iota=0$ (that is, $J^\iota=1$) on a time interval $[0,T_{\sigma}]$. 
The component corresponding to $\alpha=0$ of this
solution furnishes a solution to (\ref{eqI-16nd}). Without further analysis, one has
\emph{a priori} $T_\sigma\to 0$ as $\sigma\to 0$. An energy estimate on the exact equations -- as opposed to the mollified ones -- allows us to prove that the existence time can be taken as in the statement of Theorem \ref{theomain}. We then prove that the solution persists over
larger time if the stronger criterion (\ref{SCstrong}) is satisfied.

\subsubsection{Symmetrizer and energy}

\emph{Since $d=1,2$, we can take $t_0=3/2>d/2$ and any $N\geq 5$ satisfies the condition (\ref{condN})}.
A symmetrizer for the ``quasilinear'' system (\ref{eqam})$_\alpha$ is given by
$$
{\mathpzc S}[U]={\mathpzc S}^1[U]+{\mathpzc S}^2_\alpha[U]+{\mathpzc S}^3[U],
$$
where ${\mathpzc S}^j[U]$ ($1\leq j\leq 3$) are diagonal operators defined as
\begin{eqnarray*}
{\mathpzc S}^1[U]&=&\diag(\RT,\frac{1}{\mu}\Gb),\\
{\mathpzc S}^2_\alpha[U]&=&\diag(\frac{1}{\Bo} \Kae,\frac{1}{\mu}\Ga),\quad(\mbox{or }{\mathpzc S}^2_\alpha[U]=0 \quad
\mbox{ if } \quad \abs{\alpha}<N),\\
{\mathpzc S}^3[U]&=&\diag(a(U)+b(U)\Lambda^{-1},0),
\end{eqnarray*}
where the constants $a(U)$ and $b(U)$ are as in Lemma \ref{lemmequive}. \\
 Multiplying (\ref{eqam}) by ${\mathpzc S}^1[U]$ symmetrizes the highest order terms, but since ${\mathpzc S}^1[U]$ is a second order operator, the subprincipal terms need to be symmetrized also (otherwise the commutator term is not controlled by the
energy) and this explains the presence of ${\mathpzc S}^2[U]$. The ${\mathpzc S}^3$ component contains two terms of order $0$ and $-1$ respectively; its role is to ensure that $({\mathpzc S}[U]U,U)$ controls the energy 
${\mathcal E}^N_\sigma(U)$.

This discussion motivates the introduction of the energies $\cF_\sigma^l(U)$, for all $1\leq l\leq N$,
\begin{equation}\label{ensym}
\cF^l_\sigma(U)=\sum_{0\leq\abs{\alpha}\leq l}\cF^\alpha_\sigma(U),
\end{equation}
with
\begin{equation}
\begin{array}{lclcl}
\dsp\cF^\alpha_\sigma(U)&=&
\dsp\frac{1}{2}\big([{\mathpzc S}^1[U]+{\mathpzc S}^3[U]]U_{(\alpha)},U_{(\alpha)}\big)&
& \dsp(\alpha\neq 0),\\
\dsp \cF^0_\sigma(U)&=&\dsp \abs{\zeta}_{H^1_\sigma}^2+\frac{1}{\mu}(\psi,\Gb_\mu[0]\psi)
& &\dsp (\alpha=0)
\end{array}
\end{equation}
where $U_{(\alpha)}=(\zetaa,\psia)$, and
This definition raises at least two important remarks:
\begin{enumerate}
\item The component ${\mathpzc S}^2[U]$ does not appear in the energy; the reason
is that  $({\mathpzc S}^2[U]\dt U_{(\alpha)},U_{(\alpha)})$ can be controlled by the energy without using the equation to replace the time derivative by space derivatives. This is made possible by the fact that the energy controls
both space \emph{and} time derivatives. This method was introduced to handle surface tension in the water waves case in \cite{RoussetTzvetkov} (see also \cite{BeyerGunther})
and can be adapted here with small 
 modifications.
\item The presence of ${\mathpzc S}^3(U)$ induces nonsymmetric terms. Since these terms
are of lower order, they do not raise any difficulty as far as local existence is concerned. However, they may destroy the stability criterion since they may interfere with the
zero order component (namely, $\mfa$) of $\RT$ -- or at least restrict its range of validity 
(very small values of $\urm$ for instance). This explains why the zero order term of 
${\mathpzc S}^3(U)$ is treated with special care in the computations below.
\end{enumerate}

Let us finally remark that, owing to Proposition \ref{propG}, Lemma \ref{lemmformdergen} and Lemma \ref{lemmequive}, 
the energy  $\cF^N_\sigma(U)$ is equivalent 
to $\cE^N_\sigma(U)$, as defined in (\ref{eqenergy}),
\begin{equation}\label{ineqEE3}
\cE^N_\sigma(U)\leq (M+\frac{2}{{\mathfrak d}(U)})\cF^N_\sigma(U)
\quad\mbox{ and }\quad
\cF^N_\sigma(U)\leq \mfm^1(U)\cE^N_\sigma(U),
\end{equation}
with ${\mathfrak d}(U)$ as in (\ref{defd}).

\subsubsection{Energy estimates}\label{sectEE}
\emph{For notational simplicity, we write $\Eb=\E$, $\K=\K[\eps\sqrt{\mu}\zeta]$, $\Tb=\T$ etc.}
We first consider the case $\alpha\neq 0$.
Taking the $L^2$-scalar product of (\ref{eqam}) with $(\Sy^1+\Sy^3)\Ua+\Sy^2_\alpha\Uaa$ yields
\begin{eqnarray*}
\lefteqn{\big((\Sy^1+\Sy^3)\dt U_{(\alpha)},\Ua\big)
+\big(\dt \Ua,\Sy^2_\alpha\Uaa\big)}\\
&+&\big(\Sy^3J^\iota{\mathpzc A}U_{(\alpha)},\Ua\big)+
\big(\Sy^3J^\iota{\mathpzc C}_\alpha \Uaa,\Ua\big)
+\big(J^\iota {\mathpzc C}_\alpha\Uaa,\Sy^2_\alpha\Uaa\big)\\
&+&\big((\Sy^1+\Sy^3) J^\iota{\mathpzc B}U_{(\alpha)},\Ua\big)
+\big(J^\iota {\mathpzc B}\Ua,\Sy^2_\alpha\Uaa\big)\\
&=&\eps \big(J^\iota( R_\alpha, S_\alpha+S'_\alpha)^T,(\Sy^1+\Sy^3)\Ua+\Sy^2_\alpha\Uaa\big),
\end{eqnarray*}
where we used the fact that $\big(J^\iota {\mathpzc A}\Ua,\Sy^1\Ua\big)=0$
and $\big(J^\iota {\mathpzc A}\Ua,\Sy^2_\alpha \Uaa\big)+\big(J^\iota {\mathpzc C}_\alpha \Uaa,\Sy^1 \Ua\big)=0$. 
By definition of $\cF^\alpha_\sigma(U)$ we get therefore
\begin{equation}\label{diffineq}
\frac{d}{dt}\Big(\cF^\alpha_\sigma(U)+(\Ua,\Sy^2_\alpha\Uaa)\Big)=\sum_{j=1}^7 A_j +
B_1+B_2,
\end{equation}
with
$$
\begin{array}{lcl}
\dsp A_1=\frac{1}{2}\big((\dt\Sy^1+\dt\Sy^3)U_{(\alpha)},U_{(\alpha)}),& &
\dsp A_2=\big(\Ua,\dt(\Sy^2_\alpha\Uaa)\big),\\
\dsp A_3=-\big(\Sy^3J^\iota{\mathpzc A}U_{(\alpha)},\Ua\big),
& &A_4= -\big(\Sy^3J^\iota{\mathpzc C}_\alpha \Uaa,\Ua\big)\\
\dsp A_5=-\big(J^\iota {\mathpzc C}_\alpha\Uaa,\Sy^2_\alpha\Uaa\big), & &
\dsp A_6=-\big((\Sy^1+\Sy^3) J^\iota{\mathpzc B}U_{(\alpha)},\Ua\big),\\
\dsp A_7=-\big(J^\iota {\mathpzc B}\Ua,\Sy^2_\alpha\Uaa\big)
\end{array}
$$
and
\begin{eqnarray*}
B_1&=&\eps \big(J^\iota( R_\alpha, S_\alpha)^T,(\Sy^1+\Sy^3)\Ua+\Sy^2_\alpha\Uaa\big),\\
B_2&=&\eps \big( J^\iota(0,S'_\alpha)^T,\Sy^1\Ua+\Sy^2_\alpha\Uaa\big).
\end{eqnarray*}
 We now give some control on the different components of the 
r.h.s. of (\ref{diffineq}).\\
- \emph{Control of $A_1$.} We can decompose $A_1=A_{11}+A_{12}+A_{13}+A_{14}$ with 
$$
\begin{array}{lcl}
A_{11}=([\dt,\RT]\zetaa,\zetaa),& & A_{12}=(\frac{1}{\mu}[\dt,\Gb]\psia,\psia),\\
A_{13}=((\dt a)\zetaa,\zetaa), & &
A_{14}=((\dt b)\Lambda^{-1/2}\zetaa,\Lambda^{-1/2}\zetaa).
\end{array}
$$
One gets easily that $\abs{A_{1j}}\leq \eps\mfm^N(U)$ for $2\leq j\leq 4$ (this is straightforward for $A_{13}$ and $A_{14}$, and a consequence of Proposition 3.6 of \cite{AL} for $A_{12}$). 
We thus focus our attention on $A_{11}$. First remark that 
\begin{eqnarray*}
\lefteqn{A_{11}=((\dt\mfa)\zetaa,\zetaa)+\frac{1}{\Bo}\big([\dt,\K]\nabla\zetaa,\nabla\zetaa\big)}\\
&-&2\eps^2\mu\urp\urm \big(\Eb(\zetaa\jump{\uVpm}),\zetaa\dt \jump{\uVpm}\big)\\
&-&\eps^2\mu\urp\urm \big([\dt,\Eb]\zetaa\jump{\uVpm}),\zetaa\jump{\uVpm}\big).
\end{eqnarray*}
It is straightforward that the first two components are bounded from above by $\eps\mfm^N(U)$.
To control the other two components of this identity, we use Proposition \ref{propE} and
(\ref{cubaonemore})
to
obtain
\begin{eqnarray}
\nonumber
\abs{A_{11}}&\leq& \eps\mfm^N(U) \big(1+\eps^2\sqrt{\mu}\urp\urm\abs{\jump{\uVpm}}^2_\infty\babs{\abs{D}^{1/2}\zetaa}_2^2\big)\\
\label{eqA11}
& &+\eps^2\sqrt{\mu}\urp\urm\abs{\jump{\uVpm}}_\infty\abs{\dt \jump{\uVpm}}_\infty
\babs{\abs{D}^{1/2}\zetaa}_2^2.
\end{eqnarray}
Thanks to the third point of Lemma \ref{lemmequive}, we get therefore\footnote{if we assume
(\ref{SCalt}) rather than (\ref{SC}) then one has to adapt this estimate as indicated
in the footnote (\ref{footalt})
}
$\abs{A_{11}}\leq \mfm^N(U)$.\\
- \emph{Control of $A_2$}. From the definition of $\Sy^2_\alpha$ and the fact that time
derivatives are allowed in the energy (\ref{eqenergy}), it is easy to get that
$\abs{A_2}\leq \eps\mfm^N(U)$.\\
- \emph{Control of $A_3$}. From the definition of $\Sy^3$, it follows that
$$
A_3=a(U)\big(J^\iota \frac{1}{\mu}\Gb \psia,\zetaa\big)+
b(U)\big(\frac{1}{\mu}\Gb \psia,J^\iota\Lambda^{-1}\zetaa\big).
$$
Owing to (\ref{eqsectII1-6}), we get
$$
\abs{A_3}\leq \mu M \abs{\Pp \psia}_{2}\big(a(U)\abs{\Pp \zetaa}_2
+b(U)\abs{\Pp\zetaa}_{H^{-1}}\big)
$$
and therefore, from the definition of $a(U)$ and $b(U)$, 
$\abs{A_3}\leq \eps \mfm^N(U)$.\\
- \emph{Control of $A_4$.} Since one has
$$
A_4=a(U)\big(J^\iota \frac{1}{\mu}\Gb_{(\alpha)} \psiaa,\zetaa\big)+
b(U)\big(\frac{1}{\mu}\Gb_{(\alpha)} \psiaa,J^\iota\Lambda^{-1}\zetaa\big).
$$
and because $\abs{\Gb_{(\alpha)}\psiaa}_2\leq \eps \mfm^N(U)$ (see (\ref{estder0k})
and Lemma \ref{lemmformdergen}), we easily get that
$\abs{A_4}\leq \eps \mfm^N(U)$.\\
- \emph{Control of $A_5$.} As for $A_4$, one gets without any difficulty that
$\abs{A_5}\leq \eps \mfm^N(U)$.\\
- \emph{Control of $A_6$.} We decompose $A_6=A_{61}+A_{62}$ with
$$
A_{61}=-\big(\Sy^1 J^\iota {\mathpzc B}\Ua,\Ua\big)
\quad\mbox{ and }\quad
A_{62}=-\big(\Sy^3 J^\iota {\mathpzc B}\Ua,\Ua\big).
$$
  We can first remark that
$$
A_{61}=-\eps\big(\RT J^\iota\Tb\zetaa,\zetaa\big)
+\eps\big(\frac{1}{\mu}\Gb J^\iota\Tb^*\psia,\psia\big).
$$
We therefore deduce that $\abs{A_{61}}\leq \eps \mfm^N(U)$ from the third point of Lemma \ref{lemmequive} and the following identities,
\begin{eqnarray}
\label{cub1}
\babs{\big(\RT J^\iota{\mathpzc T}\zetaa,\zetaa\big)}\leq \mfm^N(U)\big(1+\eps^2\urp\urm\sqrt{\mu}\abs{\jump{\uVpm}}_{W^{1,\infty}}^2\abs{\zeta}_{\langle N+1/2\rangle}^2\big),\\
\label{cub2}
\babs{\big(\frac{1}{\mu}\Gb J^\iota \Tb^*\psia,\psia\big)}\leq \mfm^N(U),
\end{eqnarray}
whose proofs are given in Appendix \ref{app:cuba}. The component $A_{62}$ is much easier to
handle and one gets readily that $\abs{A_{62}}\leq \eps\mfm^N(U)$.\\
- \emph{Control of $A_7$}. One can write
$$
A_7=\eps\big(\Tb\zetaa,\frac{1}{\Bo}\Kaeb\zetaaa\big)-\eps\big(\Tb^*\psia,\frac{1}{\mu}\Ga\psiaa\big).
$$
Using the first point of Proposition \ref{propT} and the definition of $\Kaeb$ to control the
first term, and the last point of Proposition \ref{propT} and (\ref{estder0})-(\ref{estder0k}) to control the second one, we get that $\abs{A_7}\leq \eps\mfm^N(U)$.\\
- \emph{Control of $B_1$.} Thanks to the bounds established on $R_\alpha$ and $S_\alpha$
in Proposition \ref{propIII-1} and to the third point of Lemma \ref{lemmequive}, 
there is no pain to get $\abs{B_1}\leq \eps\mfm^N(U)$.\\
- \emph{Control of $B_2$} From the definition of $S'_\alpha$, (\ref{eqsectII1-6}) and (\ref{estder0k}), we get that 
$$
\abs{B_2}\leq \mfm^N(U)\big(1+\abs{\zetaa}_{H^1}\big)\\
\leq  (1+\sqrt{\Bo})\mfm^N(U).
$$

\medbreak

Gathering all the informations coming from the above estimates, we deduce from (\ref{diffineq})
that for all $1\leq \abs{\alpha}\leq N$, 
\begin{equation}\label{estalpha1N}
\frac{d}{dt}\Big(\cF^\alpha_\sigma(U^\iota)+(\Ua^\iota,\Sy^2_\alpha\Uaa^\iota)\Big)\leq (1+\sqrt{\Bo})\mfm^N(U^\iota).
\end{equation}

For $\alpha=0$, we first rewrite the equations (\ref{eqmol}) under the form
\begin{equation}\label{eqmol2}
\left\lbrace
\begin{array}{lcl}
\dsp \dt \zeta - \frac{1}{\mu}J^\iota\Gb_\mu[0]\psi+\eps J^\iota {\mathcal N}_1(U)=0,\\
\dsp \dt \psi+J^\iota\big(1-\frac{1}{\Bo}\nabla\cdot(\frac{1}{\sqrt{1+\eps^2\mu\abs{\nabla\zeta}^2}}\nabla)\bullet \big)\zeta + \eps J^\iota{\mathcal N}_2(U)=0,
\end{array}\right.
\end{equation}
with
\begin{eqnarray*}
{\mathcal N}_1(U)&=&-\frac{1}{\eps\mu}\big(\Gb_\mu[\eps\zeta]\psi-\Gb_\mu[0]\psi\big),\\
{\mathcal N}_2(U)&=&\frac{1}{2}\Big(\jump{\urpm\babs{\nabla\psi^\pm}^2}
- \frac{1}{\mu}(1+\eps^2\mu\abs{\nabla\zeta}^2)\jump{\urpm(\uwpm)^2}\Big)\\
& &+\frac{1}{\eps\Bo}\nabla\cdot(1-\frac{1}{\sqrt{1+\eps^2\mu\abs{\nabla\zeta}^2}})\nabla\zeta.
\end{eqnarray*}
We  then take the $L^2$ product of (\ref{eqmol2}) with 
$
\big(\big(1-\frac{1}{\Bo}\Delta \big)\zeta,\frac{1}{\mu}\Gb_\mu[0]\psi\big)
$
 to get 
\begin{eqnarray}
\nonumber
\frac{d}{dt}\cF^0_\sigma(U)&\leq&  \eps \abs{{\mathcal N}_1(U)}_{H^{1}_\sigma}\abs{\zeta}_{H^{1}_\sigma}+\eps\abs{\Pp{\mathcal N}_2(U)}_2\abs{\Pp \psi}_2\\
\label{estalpha0}
&\leq& \eps\mfm^N(U),
\end{eqnarray}
where the control of ${\mathcal N}_1(U)$ and ${\mathcal N}_2(U)$ does not
raise any particular difficulty.

Summing (\ref{estalpha1N}) and (\ref{estalpha0})
over all $\alpha\in \N^{d+1}$, first with $\abs{\alpha}\leq N-1$ and
then with $\abs{\alpha}\leq N$ (and recalling that 
$\Sy_\alpha^2=0$ if $\abs{\alpha}<N$), we deduce that
\begin{eqnarray*}
\frac{d}{dt}\Big(\cF^{N-1}_\sigma(U^\iota)\Big)&\leq& (1+\sqrt{\Bo})\mfm^N(U^\iota),\\
\frac{d}{dt}\Big(\cF^N_\sigma(U^\iota)+\sum_{\abs{\alpha}=N}(\Ua^\iota,\Sy^2_\alpha\Uaa^\iota)\Big)&\leq& (1+\sqrt{\Bo})\mfm^N(U^\iota),
\end{eqnarray*}
and therefore, for any constant ${\mathfrak C}$,
\begin{equation}\label{pregron}
\frac{d}{dt}\widetilde{\cF}^N_\sigma(U^\iota)\leq (1+\sqrt{\Bo})\mfm^N(U^\iota).
\end{equation}
with $\widetilde{\cF}^N_\sigma(U^\iota)=\cF^N_\sigma(U^\iota)+{\mathfrak C}\cF^{N-1}_\sigma(U^\iota)+\sum_{\abs{\alpha}=N}(\Ua^\iota,\Sy^2_\alpha\Uaa^\iota)$.\\
Now, it follows from the definition of $\Sy^2_\alpha$ that there exists a constant
$c_3(U^\iota)\leq M$ such that, 
$$
\sum_{\abs{\alpha}=N}\abs{(\Ua^\iota,\Sy^2_\alpha\Uaa^\iota)}\leq \frac{1}{2}\cF_\sigma^N(U^\iota)+c_3(U^\iota)\cF_\sigma^{N-1}(U^\iota).
$$
As long as ${\mathfrak C}\geq c_3(U^\iota)$ (and ${\mathfrak C}\leq M$), one has therefore,
$$
\frac{1}{2}\cF^N_\sigma(U^\iota)\leq \widetilde{\cF}^N_\sigma(U^\iota)\leq M\cF^N_\sigma(U^\iota).
$$
With the help of (\ref{ineqEE3}), this in turn implies that
\begin{equation}\label{NRJEF}
\cE^N_\sigma(U^\iota)\leq 2(M+\frac{2}{{\mathfrak d}(U^\iota)})\widetilde{\cF}^N_\sigma(U^\iota);
\end{equation}
it follows therefore from (\ref{pregron})
 that
\begin{equation}\label{EEnew}
\frac{d}{dt}\widetilde{\cF}^N_\sigma(U^\iota)\leq (1+\sqrt{\Bo})\widetilde{\mfm}^N(U^\iota),
\end{equation}
where $\dsp \widetilde{\mfm}^N(U^\iota)=C\big(M,\widetilde{\cF}^N_\sigma(U^\iota),\frac{1}{\Bo},\frac{1}{{\mathfrak d}(U^\iota)}\big)$. 

\subsubsection{Construction of a solution for times of order $O(1)$}\label{sectsol1}

We deduce classically from (\ref{EEnew}) that
there exists $T>0$ as in the statement of the theorem such that
$\widetilde{\cF}^N_\sigma(U^\iota(t))\leq C(T,\widetilde{\cF}^N_\sigma(U^0))$ for all
$t\in [0,T*(1+\sqrt{\Bo})^{-1}]$ (we proceed as in \S \ref{sectinit} to define the
initial energy $\widetilde{\cF}^N_\sigma(U^0)$).\\
 Since this time interval does not depend on the mollifying parameter $\iota$, we can use a standard compactness argument (see e.g. \cite{ABZ} for details in a related
context)
to construct an exact solution $U\in E^N_{\sigma,T}$ to
(\ref{eqI-16nd}) -- note that we use (\ref{NRJEF}) to deduce a bound on $\cE^N_\sigma(U)$ from 
the bound on $\widetilde{\cF}^N_\sigma(U)$.

As said in the introduction, the Bond number $\Bo$ is generally very large for the applications
and the existence time of the solution constructed in the previous solution is much shorter
than what can be physically observed. This is the reason why we paid a lot of attention to
obtain energy estimates independent of $\Bo$. In fact the presence of the prefactor $\dsp(1+\sqrt{\Bo})$ in (\ref{EEnew}) only comes from the control of $B_2$ in \S \ref{sectEE}. Since
$B_2$ disappears if the energy estimates of \S \ref{sectEE} are performed on the exact equations
(\ref{eqI-16nd}) rather than the mollified ones, one can replace (\ref{EEnew}) for the
solution $U$ constructed above by
\begin{equation}\label{EEnewbis}
\frac{d}{dt}\widetilde{\cF}^N_\sigma(U)\leq \widetilde{\mfm}^N(U),
\end{equation}
from which we deduce that it is possible to extend the time interval of existence to
 $[0,T]$. We omit the proof of the uniqueness, which is completely standard.

\subsubsection{Proof of Theorem \ref{theomain2}}

In \S \ref{sectEE}, all the $A_j$ ($1\leq j\leq 7$) and $B_k$ ($1\leq k\leq 2$) are bounded from
above by $\eps\mfm^N(U)$, except $A_1$ and $B_2$. Since $B_2$ is a mollifying error that
vanishes for the exact solution, we can focus our attention on $A_1$, and more specifically
on $A_{11}$ since $A_{1l}$ ($2\leq l\leq 4$) are bounded from above by $\eps\mfm^N(U)$.
For $A_{11}$, we recall that we only have $\abs{A_{11}}\leq \mfm^N(U)$. More precisely,
(\ref{eqA11}) yields
$$
\abs{A_{11}}\leq \eps\mfm^N(U) 
+\eps^2\sqrt{\mu}\urp\urm\mfe(\zeta)\abs{\jump{\uVpm}}_\infty\abs{\dt \jump{\uVpm}}_\infty
\babs{\abs{D}^{1/2}\zetaa}_2^2.
$$
Using the third point of Lemma \ref{lemmequive}, the second component of the r.h.s. can be 
bounded from above by $\mfm^N(U)$ but not by $\eps \mfm^N(U)$. The stronger stability
criterion (\ref{SCstrong}) allows us to improve this estimate, as shown in the lemma
below (whose proof is very similar to the proof of the third point of Lemma \ref{lemmequive}
and therefore omitted).
\begin{lemma}
Let $T>0$, $t_0>d/2$ and
$U=(\zeta,\psi)\in E^N_{\sigma,T}$ be such that 
(\ref{sectII1}) and (\ref{SCstrong}) are uniformly satisfied on 
$[0,T]$.
Then  one has
$$
\eps^2\sqrt{\mu}\urp\urm\mfe(\zeta)\max_{\abs{\alpha}\leq 1}\babs{\partial^\alpha\jump{\uVpm}}^2_\infty\abs{u}_{\langle1/2\rangle}^2\leq
\eps^\gamma\mfm^1(U)\abs{u}_{H^1_\sigma}.
$$
\end{lemma}
Owing to this lemma, we get $\abs{A_{11}}\leq \eps^\gamma \mfm^N(U)$ and (\ref{EEnewbis}) 
can therefore
be improved into
$$
\frac{d}{dt}\widetilde{\cF}^N_\sigma(U)\leq \eps^\gamma \widetilde{\mfm}^N(U),
$$
from which the theorem follows easily.

\section{Applications}\label{sectappl}

Physical implications of the stability criterion (\ref{SC}) will be investigated in
separate works, but in order to enforce its relevance, 
we describe briefly here some of them. We also show how Theorem \ref{theomain} can be
used to give a rigorous justification
to two-fluid asymptotic models.

\medbreak

There are mainly two different kinds of applications. The first one is provided by
air-water interfaces. In this case stability of interfacial waves is a consequence
of the smallness of $\urm$ (as in the zero density/zero surface tension limit considered in \cite{Pusateri2} for instance). The second application concerns internal waves for which both densities are comparable and stability is a consequence of the smallness of the shear velocity at the interface.

\subsection{Air-water interface in coastal oceanography}\label{sectairwater}

We are interested here in applying the above theoretical results to typical 
configurations in coastal oceanography (see $\sharp 1$ in the introduction). 
We want to investigate two things in particular:
\begin{enumerate}
\item Can we use Theorem \ref{theomain2} to confirm that the standard models
used in coastal oceanography (e.g. Boussinesq) do not need any correction due
to the presence of the air?
\item Can we still neglect the density of the air near the breaking point?
\end{enumerate}

Before addressing these questions, let us give some numerical values for
the physical parameters involved. The density of the air (at one atmosphere and around $20$ Celsius) 
is $\rho^-=1.2 \,\mbox{kg}\cdot\mbox{m}^{-3}$. For sea water, the density is $\rho^+=1025\,\mbox{kg}\cdot\mbox{m}^{-3}$. These values lead to $\urm=0.001$ and $\urp=0.999$. For the air/water interface, 
the surface tension coefficient is $\sigma=0.073\, \mbox{N}\cdot\mbox{m}^{-1}$, and we
take for the acceleration
of gravity $g=9.81 \mbox{m}\cdot\mbox{s}^{-2}$. 

We also assume here that the air layer has the same height $H^+=H^-=H$ as the
water layer (this is of course not realistic, but as said in the introduction, the techniques used here could certainly be generalized to, say, an air layer of infinite depth, and the qualitative behavior should remain the same\footnote{Since $\urm$ is very small, the influence of $H^-$ on the typical length $H$ is negligible. In our examples, only the third decimal of $H$ is changed if $H^-$ is taken equal to $+\infty$}). 
\subsubsection{Validity of the one-fluid long wave models}\label{sectvalidity}

We are interested in long waves models for which $\eps^+\sim\mu^+\ll 1$.
It is known \cite{Craig,AL} in the one-fluid case that these models
describe correctly the solutions of the water waves equations for times of 
order $O(1/\eps)$.

Let us consider first a typical ``long wave''  over a depth $H^+=H=5\,\mbox{m}$, of wavelength $\lambda=35\,\mbox{m}$, and amplitude $0.1\,\mbox{m}$. 
As said in Remark \ref{rem2pareil}, in order to show that taking into account the density of the air does not yield significant modifications to the
standard one-fluid theory, we have to check that the strong stability criterion
(\ref{SCstrong}) is satisfied\footnote{For these values, the Bond number is $\Bo\sim 1.7.10^{8}$. A local well-posedness result relying only on the
extra control furnished by surface tension, see \S \ref{sectsol1} would therefore
give an existence about $\sqrt{\Bo}\eps^{-1}\sim 6.10^{5}$ times smaller than Theorem \ref{theomain2}}, with $\gamma=1$. We rather check that its practical
version (see Remark \ref{rempractstrong}) is satisfied, which is the case since
one computes easily that $\eps^{-2}\Upsilon\sim 4.10^{-4}\ll 1$. The solutions of the
one-fluid and two-fluid models exist therefore on the same time scale.

We also know from \cite{AL} (one-fluid) and \cite{BLS} (two-fluid) that these solutions are 
 correctly described by standard long wave models (Boussinesq, KdV, etc.). The influence of
the air density and the surface tension on these models being negligible (for instance, surface tension induces a modification of the dispersive term of the KdV model of less than $10^{-4}\%$), this confirms the relevance of the one-fluid models (without surface tension) to describe such phenemonons. 
\subsubsection{Kelvin-Helmholtz instabilities for breaking waves and white caps}\label{sectKHI}

A commonly observed value for $\eps^+$ near the breaking point of the wave
is $\eps^+\sim 0.4$, corresponding to an amplitude $a=6\,\mbox{m}$ for a depth
$H^+=15\,\mbox{m}$. One computes $\Upsilon\sim 0.27$, which is not small 
enough to give any definitive conclusion on the presence or not
of Kelvin-Helmholtz instabilities. We therefore have to use
the exact criterion (\ref{SC}). This requires further experimental work, but it is likely that \emph{for some configurations, Kelvin-Helmholtz instabilities are at the origin of wave breaking}. This might be the case for instance for spilling breakers (as opposed to plunging breakers). It is indeed not understood which kind of singularity of the water waves equations creates such breakers; this might be because it is a \emph{two-fluid} singularity similar to the ``white caps'' observed in presence of wind. The occurrence of these white caps is commonly explained by the \emph{linear} Kelvin criterion stated in the introduction. It would be interesting to check
whether a nonlinear version of our \emph{nonlinear} stability criterion (\ref{SC}) with
nonzero background explains part of the discrepancies of this theory (see \cite{Chandrasekhar}).

\subsection{Internal waves} \label{sectinternal}

Waves at the interface of two immiscible fluids of comparable density are called internal waves. Waves at the interface of two layers of water of different density also carry the same name and are often described using the two-fluid formalism (see the recent review \cite{HelfrichMelville}). The stability criterion (\ref{SC}) is discussed here in these two configurations.

\subsubsection{The Koop-Butler experiment}\label{sectKP}

We consider here a classical experiment by Koop and Butler \cite{KP}. The upper fluid is deionized water ($\rho^-=998\,\mbox{kg}\cdot \mbox{m}^{-3}$) and the lower fluid is Freon TF ($\rho^+=1563\,\mbox{kg}\cdot\mbox{m}^{-3}$). The depth of both layers in this experiment are $H^+=1.366\, \mbox{cm}$ and $H^-=6.948\, \mbox{cm}$. One gets therefore $\urp=0.610$, $\urm=0.390$ and
$H=1.989\,\mbox{cm}$.\\
 The interfacial tension coefficient $\sigma$ is not provided in \cite{KP} but the chemistry literature 
suggests that $\sigma=0.005 \, \mbox{N}\cdot\mbox{m}^{-1}$ is a reasonable value.
Koop and Butler observed different interfacial waves of amplitude ranging
from $a=0.034\,\mbox{cm}$ to $a=0.68 \,\mbox{cm}$.
The dimensionless parameter $\Upsilon$ satisfies therefore
$5.39\times 10^{-7}\leq \Upsilon\leq 0.086$. According to the
practical criterion (\ref{practSC}) , this correspond to stable configurations.
Though the practical version of the  strong stability criterion (\ref{SCstrong}) 
(see Remark \ref{rempractstrong}) is satisfied, we do not discuss here large
time existence because the viscosity of Freon TF should then be taken into account.

\subsubsection{The Grue \emph{et al.} experiment and oceanic internal waves}\label{sectG}

The configuration studied in \cite{Grue} is slightly different from the experiment
mentioned in the previous section. Indeed, the authors do not consider two immiscible
fluids but brine $\rho^+=1022 \, \mbox{kg}\cdot\mbox{m}^{-3}$ and water
$\rho^-=999 \, \mbox{kg}\cdot\mbox{m}^{-3}$ ($\urp=0.506$, 
$\urm=0.494$). This configuration is thus closer
to what is observed for oceanic internal waves.\\
The difference with the case of two immiscible fluids is that there is a thin zone 
(pycnocline)
in which the density is allowed to vary continuously between $\rho^+$ and $\rho^-$.
This continuous stratification is known to have a stabilizing role, which is important
for small amplitude interfacial waves.\\
It seems therefore that the physical framework studied in the present article
differs from \cite{Grue}. However, water-brine interfaces have been described
with great success with models based the two-fluid equations (\ref{eqI-1})-(\ref{eqI-6})
(see \cite{HelfrichMelville} and references therein).
It follows from our analysis that surface tension must be included to these
equations if we want to use them to describe water-brine interfaces. This surface
tension is somehow artificial and a natural question is to know which value
it should be given.

We propose here to use our analysis and the experiments of \cite{Grue} to
sketch a method to find a realistic value for $\sigma$. The main idea is that
\emph{its value should be such that the Kelvin-Helmholtz instabilities
predicted by the theoretical
analysis should coincide with those observed in the experiments}.

In \cite{Grue}, the depth of the two layers are respectively $H^+=0.62\, \mbox{m}$ and
$H^-=0.15\, \mbox{m}$, so that $H=0.243\,\mbox{m}$.\\
The authors report a set of measurements (Fig. 7 in \cite{Grue}) for interfacial waves of amplitude ranging from $a=0.033\,\mbox{m}$ to $a=0.226\,\mbox{m}$.
Kelvin-Helmholtz instabilities are not observed up to $a=0.184 \,\mbox{m}$ 
but are present
for the experiment with largest amplitude $a=0.226 \,\mbox{m}$. It is therefore
reasonable to assume that the critical amplitude is $a=0.2 \,\mbox{m}$. According
to the practical criterion (\ref{practSC}), this critical value should
correspond to $\Upsilon \sim 1$. We can then use this relation to get
$
\sigma=0.095 \,\mbox{N}\cdot \mbox{m}^{-1};
$
this value is comparable to the air-water surface tension, but slightly higher,
which is not surprising since the stabilizing effects of the continuous stratification
are taken into account with such a choice. Further investigation is required to check whether 
this approach can prove useful in the study of oceanic internal waves.

\subsection{Full justification of two-fluid asymptotic models}

We can use Theorem \ref{theomain2} to provide 
a rigorous justification of asymptotic models for internal waves,  
on the relevant time scale (to our knowledge, the only rigorous result so far
is a justification of the Benjamin-Ono equations on the too-short time scale $O(\Bo^{-1/2})$ that
can be found in \cite{OhiIguchi}):\\
Many asymptotic regimes and models exist in the literature (\cite{CC2,CGK,ND} 
and references therein). Theorem \ref{theomain2} can be used to justify rigorously
most of them along the procedure used in \cite{AL} for water waves models.
\begin{itemize}
\item (Consistency) One proves that any family of smooth enough, uniformly bounded, 
solutions to (\ref{eqI-16nd})
existing on the relevant time scale
 solves the asymptotic model up to a small residual. This is done for
a wide class of regimes in \cite{BLS}.
\item (Convergence) One proves that such exact solutions to the two-fluid equations (\ref{eqI-16nd}) remain close to exact solutions of the asymptotic model. From the previous step, this
only requires energy estimates on the asymptotic model (generally much easier than for the full two-fluid equations).
\item (Existence) One proves that smooth, uniformly bounded family of
solutions to (\ref{eqI-16nd}) whose existence is assumed in the previous steps exist indeed.
As in the water waves case, this is the most difficult step. It is here ensured by Theorem \ref{theomain2} if the stability criterion (\ref{SCstrong}) is satisfied (which is the case
for the stable waves of  the above experimental observations). 
\end{itemize}

Let us for instance implement this procedure in the ``shallow-water/shallow-water'' regime
characterized by $\eps\sim 1$, $\mu\ll 1$. In this regime, the interface elevation $\zeta$ and $v=\nabla\psi$ are commonly described \cite{CC2,CGK} by the solution $(\zeta^a,v^a)$ of the following system (for notational simplicity, we give it for  $d=1$, but there is no additional difficulty to take into account the nonlocal generalization of this system to the case $d=2$ derived in \cite{BLS} and studied in \cite{GLS}),
\begin{equation}\label{systSWSW}
\left\lbrace
\begin{array}{l}
\dsp \dt \zeta^a +\partial_x \Big[\frac{h^-(\zeta^a) h^+(\zeta^a)}{\urp h^{-}(\zeta^a)+\urm h^+(\zeta^a)}v^a\Big]=0\vspace{1mm}\\
\dsp \dt v^a + \partial_x \zeta^a+\eps\frac{1}{2}\partial_x \Big[\frac{\urp h^-(\zeta^a)^2-\urm  h^+(\zeta^a)^2}{[\urp h^{-}(\zeta^a)+\urm  h^+(\zeta^a)]^2}(v^a)^2\Big]=0
\end{array}\right.
\end{equation}
with $h^-(\zeta^a)=\uH^-(1-\eps^-\zeta^a)$ and $h^+(\zeta^a)=\uH^+(1+\eps^+\zeta^a)$.
We justify the approximation furnished by the solutions of this system in the following sense.
\begin{theorem}
Let $U^0=(\zeta^0,\psi^0)^T$ be as in the assumptions of Theorem \ref{theomain}, and 
assume that (\ref{SC}) is satisfied uniformly with respect to $\mu\in(0,1)$. Let also $v^0=\partial_x\psi^0 $ and assume that
\begin{equation}\label{hypass}
\inf_{\R^d} \Big(1-\eps^2\urp\urm \frac{(\uH^++\uH^-)^2}{[\rho^+ h^{-}(\zeta^0)+\rho^-  h^+(\zeta^0)]^3}(v^0)^2\Big)>0
\end{equation}
Then there exists $T>0$ such that for all $\mu\in (0,1)$:\\
(i) there exists a unique solution
$(\zeta^a,v^a)\in C([0,T];H^{N-1/2}(\R)^2)$ to (\ref{systSWSW}) with 
 initial condition $(\zeta^0,v^0)$;\\
(ii) the solution $U$ provided by Theorem \ref{theomain} exists on $[0,T]$.\\
Moreover, with $v=\partial_x \psi$, one has
$$
\abs{(\zeta,v)-(\zeta^a,v^a)}_{L^\infty([0,T]\times \R^d)}\leq \mu C(U^0).
$$
\end{theorem}
\begin{remark}
The condition (\ref{hypass}) ensures the hyperbolicity of (\ref{systSWSW}); it is not present for the standard one-fluid shallow-water equations. An asymptotic expansion of the instability operator $\RT$ shows that the term responsible for
the Kelvin-Helmholtz instability (the one involving $\E$) contributes to (\ref{hypass}),
which was suspected but not rigorously established so far. The shallow-water/shallow-water model (\ref{systSWSW}) is however certainly too unstable and the analysis performed here on
the full Euler equations shows that the surface tension term should be kept to control high frequency instabilities.
\end{remark}
\begin{proof}
For the first point, we  remark that $\abs{\zeta^0}_{H^{N-1/2}}+\abs{v^0}_{H^{N-1/2}}\leq \mfm^N(U)$; the existence of a solution $(\zeta^a,v^a)\in C([0,T];H^{N-1/2}(\R)^2)$ to (\ref{systSWSW}) with 
 initial condition $(\zeta^0,v^0)$ (with $T=T(\mfm^N(U))$) is therefore straightforward because the assumption made in the statement of the theorem insures the hyperbolicity of (\ref{systSWSW}), as shown in Theorem 1 of \cite{GLS}.\\
Taking a smaller $T$ if necessary, it
is possible to assume that the solution of Theorem \ref{theomain} exists also on $[0,T]$;
this theorem ensures also that $\sup_{[0,T]}\cE^N_\sigma(U)$ is uniformly bounded with respect to $\mu \in (0,1)$. Thanks to the consistency result provided by Theorem 4 of \cite{BLS}, this
ensures that $(\zeta,v)$ solves (\ref{systSWSW}) up to a residual $\mu(r^\mu_1,r^\mu_2)$,
with $(r_1^\mu,r_2^\mu)$ uniformly bounded with respect to $\mu\in (0,1)$ in 
$C([0,T];H^1(\R)^2)$. A standard error estimate on the hyperbolic system (\ref{systSWSW})
yields therefore
$$
\abs{(\zeta,v)-(\zeta^a,v^a)}_{L^\infty([0,T];H^1(\R)^2}\leq \mu C(U^0),
$$
and the estimate given in the theorem follows from the continuous embedding $H^1(\R)\subset L^\infty(\R)$.
\end{proof}

\appendix

\section{Nondimensionalization of the equations}\label{apnd}

It is quite straightforward to nondimensionalize the space variable $X$ and the interface deformation $\zeta$ using the quantities $a$, $\lambda$, $H^\pm$ and $H$ introduced in \S \ref{sectND}. We thus
define
$$
\widetilde{X}=\frac{X}{\lambda},\qquad \widetilde{\zeta}=\frac{\zeta}{a},
$$
where the tildes are used to denote dimensionless variables or unknowns. The nondimensionalization of the
time variable $t$ and the velocity potential $\psi$ requires the knowledge of a reference velocity $c$. Such a value can be obtained by direct observation, but is of course preferable to have a formula for $c$ in terms of the above quantities. The best way to have access to such an information is through the \emph{linear}
analysis of the equations. Neglecting all the nonlinear terms in (\ref{eqI-16}), one easily gets the
system (see Remark \ref{remGNlin} for the expression of ${\mathpzc G}[0]$)
$$
\left\lbrace
\begin{array}{l}
  \partial_t \zeta-{\mathpzc G}[0]\psi=0,\\
\dt \psi+g'\zeta=0,
\end{array}\right.
\mbox{ with } 
{\mathpzc G}[0]=\abs{D}\frac{\tanh(H^+\abs{D})\tanh(H^-\abs{D})}{\urp\tanh(H^-\abs{D})+\urm\tanh(H^+\abs{D})}.
$$
In the \emph{shallow water limit}, that is, when $H^\pm$ is small compared to the typical wavelength $\lambda$, one has ${\mathpzc G}[0]\sim -\uH\Delta$ (it is possible, but not necessary at this point, to give a precise meaning to the symbol $\sim$). The above linearized equation then reduces to a wave equation of speed $c$,
with
$$
c^2=g'\frac{H^+H^-}{\urp H^-+\urm H^+}=g'H.
$$
We then use this velocity to nondimensionalize $t$ and $\psi$,
$$
\widetilde{t}=\frac{t}{\lambda/c},\qquad \widetilde{\psi}(\widetilde{X})=\frac{\psi(X)}{\psi_0} \quad\mbox{ with }\quad
\psi_0=\frac{g'a\lambda}{c}.
$$
(the nondimensionalization of $\psi$ is obtained via the equation $\dt\psi+g'\zeta=0$). 

We now remark that
$$
{\mathpzc G}^\pm[\zeta]\psi=\frac{\psi_0}{H}{\mathpzc G}^\pm_\mu[\eps \widetilde{\zeta},\uH^\pm]\widetilde{\psi}
\quad\mbox{ and thus }\quad
 {\mathpzc G}[\zeta]\psi=\frac{\psi_0}{H}{\mathpzc G}_\mu[\eps \widetilde{\zeta}]\widetilde{\psi}
$$
where ${\mathpzc G}^\pm_\mu$ and ${\mathpzc G}_\mu$ are as defined in (\ref{scalGN}) and (\ref{14nd}).
Since moreover
$$
\nabla\psi^\pm=\frac{\psi_0}{\lambda}\widetilde{\nabla}\widetilde{\psi}^\pm,
$$
 it is
straightforward to deduce (\ref{eqI-16nd}) from (\ref{eqI-16}).

\section{Proof of (\ref{DNstar}), (\ref{eqsectII1-6}), (\ref{estder0bis}),  (\ref{estder0bisk}),
(\ref{eqsectII1-6gen})
and  (\ref{eqsectII1-6genbis})} \label{sectDNapp}

Before proving these identities, we give some remarks of general interest:\\
- The following classical product estimate holds
\begin{equation}\label{prodest}
\forall s\geq 0\quad \mbox{ and }\quad t_0>d/2,\qquad \abs{fg}_{H^s}\lesssim \abs{f}_{H^{s\vee t_0}}\abs{g}_{H^s}. 
\end{equation}
- For functions defined on the strip $\cS^\pm$, the above formula can be used in the
 horizontal directions to obtain 
\begin{equation}\label{prodeststrip}
\forall s\geq 0\quad \mbox{ and }\quad t_0>d/2,\qquad \Abs{\Lambda^s(fg)}_{2}\lesssim \Abs{f}_{L^\infty_zH^{s\vee t_0}}\Abs{\Lambda^s g}_{2}. 
\end{equation}
- For all $k\in\N$, we denote by $H^{s,k}(\cS)$ the space
\begin{equation}\label{pinard}
H^{s,k}(\cS^\pm)=\{f\in L^2(\cS^\pm), \Abs{f}_{H^{s,k}}<\infty\},
\qquad \Abs{f}_{H^{s,k}}^2=\sum_{j=0}^k\Abs{\Lambda^{s-k}\partial_z^k u}_2^2.
\end{equation}
- The space $H^{s+1/2,1}(\cS^\pm)$  is continuously embedding in $L^\infty_z H^s(\cS^\pm)$ (this is a variant of the trace lemma),
\begin{equation}\label{contemb}
\forall s\geq 0, \qquad \Abs{f}_{L^\infty_z H^s}\lesssim \Abs{f}_{H^{s+1/2,1}}.
\end{equation}

Finally, let us state the following lemma, which will be used several times.
\begin{lemma}\label{lemkk}
Let $t_0>d/2$ and $u\in \dot{H}^1(\cS^\pm)$, ${\bf g}\in H^{s,1}(\cS^\pm)^{d+1}$ be such that
$$
\left\lbrace
\begin{array}{l}
\nampm\cdot P(\Sigma^\pm)\nampm u=\nampm\cdot{\bf g},\\
u_\interff=0,\qquad \partial_n u_{\vert_{z=\mp 1}}=-{\bf e}_z\cdot {\bf g}_{\vert_{z=\mp 1}}.
\end{array}\right.
$$
Then, for all $0\leq s\leq t_0+1$,
$$
\Abs{\Lambda^s\nampm u}_2\leq M\Abs{\Lambda^s {\bf g}}_2
\quad\mbox{ and }\quad
\Abs{\nampm u}_{H^{s,1}}\leq M\big(\Abs{ \Lambda^s{\bf g}}_{2}+\Abs{\Lambda^{s-1}\dz{\bf g}\cdot{\bf e}_z}_2\big).
$$
\end{lemma}
\begin{proof}
The first estimate is a consequence of Proposition 2.4 of \cite{AL}. For the second estimate, and since $\mu^\pm$ is assumed to be bounded, it is enough to prove that $\Abs{\Lambda^{s-1}\dz^2 u}_{2}\leq M \Abs{{\bf g}}_{H^{s,1}}$. This follows immediately from isolating the component $\dz^2 u$ in the equation
$\nampm\cdot P\nampm u=\nampm\cdot {\bf g}$ and using the first estimate to control the terms involving $u$.
\end{proof}

\subsection{Proof of (\ref{DNstar}), (\ref{eqsectII1-6}) and (\ref{DNcommut})}
\subsubsection{Proof of  (\ref{DNstar})}\label{appDNstar}

The case $s=0$ being proved in Prop. 3.4 and Eq. (3.3) of \cite{AL}, we focus on the case $s>0$. Remarking first
that
$\abs{\Pp \psi^\pm}_{H^s}=\abs{\Pp (\Lambda^s \psi^\pm)}_2$, we can deduce from (\ref{DNstar}) with $s=0$
that
\begin{eqnarray*}
\sqrt{\mu}\abs{\Pp \psi^\pm}_{H^s}&\leq& M \Abs{\nampm(\Lambda^s\psi^\pm)^\mfh}\\
&\leq& M \big( \Abs{\nampm ((\Lambda^s\psi^\pm)^\mfh-\Lambda^s\phi^\pm)}_2+\Abs{\Lambda^s\nampm\phi^\pm}_2,
\end{eqnarray*}
where $(\Lambda^s\psi^\pm)^\mfh$ stands for the solution to (\ref{eqsectII1-3}) with Dirichlet data $\Lambda^s\psi^\pm$. Remarking  that $u=(\Lambda^s\psi^\pm)^\mfh-\Lambda^s\phi^\pm$ is as in Lemma \ref{lemkk} with ${\bf g}=[\Lambda^s,P^\pm]\nampm\phi^\pm$, we have
\begin{eqnarray}
\nonumber
\Abs{\nampm ((\Lambda^s\psi^\pm)^\mfh-\Lambda^s\phi^\pm)}_2&\leq& \Abs{{\bf g}}_2\\
\label{pistoche}
&\leq& M\Abs{\Lambda^{s-1}\nampm \phi^\pm}_2,
\end{eqnarray}
where we used the commutator estimate (\ref{nouv}) to derive the second inequality. The end of the proof is then straightforward.

\subsubsection{Proof of (\ref{eqsectII1-6})}\label{proofII1-6}

Integrating by parts in (\ref{eqsectII1-3}) one gets easily
\begin{equation}\label{base}
	\int_{\R^d}\Lambda^s\Gpmb\psi_1\Lambda^s\psi_2=
	\int_{\cS^\pm} \Lambda^s(P(\Sigma^\pm)\nampm\phi^\pm_1)\cdot
	\Lambda^s\nampm\psi_2^\dagger,
\end{equation}
where $\phi_1^\pm$ denotes the solution to (\ref{eqsectII1-3}) with Dirichlet data $\psi_1$,
while
$$
\psi_2^\dagger(\cdot,z)=\chi(\sqrt{\mu^\pm}z\abs{D})\psi_2,
$$
with $\chi$ a smooth compactly supported function equal to $1$ in a neighborhood of the origin.
One then gets
$$
	\babs{(\Lambda^s\Gpmb\psi_1,\Lambda^s\psi_2)}\leq
	\Abs{\Lambda^s(P(\Sigma^\pm)\nampm\phi^\pm_1)}_2
	\Abs{\Lambda^s\nampm\psi^\dagger_2}_2.
$$
From (\ref{prodeststrip}) one gets, for all $0\leq s\leq t_0+1$,
\begin{eqnarray*}
	\Abs{\Lambda^s(P(\Sigma^\pm)\nampm\phi^\pm_1)}_2
	&\lesssim&
	(1+\Abs{P(\Sigma^\pm)-\Id}_{L^\infty_zH^{t_0+1}})\Abs{\Lambda^s\nampm\phi^\pm_1}_2\\
        &\leq& M \Abs{\Lambda^s\nampm\phi^\pm_1}_2,
\end{eqnarray*}
where the second identity follows from (\ref{Psob}); 
(\ref{eqsectII1-6}) follows therefore from (\ref{estreg}) and the observation that 
$\Abs{\Lambda^s\nampm\psi^\dagger_2}_2\lesssim \sqrt{\mu}\abs{\Pp \psi_2}_{H^s}$ (Proposition 2.2 of \cite{AL}).

\subsubsection{Proof of (\ref{DNcommut})}\label{proofcommut}

Let us first prove (\ref{DNcommut}) in the case $1\leq s\leq t_0+1$. With the same notations as in \S \ref{proofII1-6}, one gets
\begin{eqnarray*}
([\Lambda^s,\Gpmb]\psi_1,\Lambda^s \psi_2)&=&\int_{\cS^\pm}[\Lambda^s,P(\Sigma^\pm)]\nampm\psi_1^\mfh\cdot \Lambda^s\nampm\psi_2^\dagger\\
& &+\int_{\cS^\pm}P(\Sigma^\pm)\nampm[\Lambda^s\psi_1^\mfh-(\Lambda^s\psi_1)^\mfh]\cdot\Lambda^s\nampm\psi_2^\dagger.
\end{eqnarray*}
Using Cauchy-Schwarz inequality, (\ref{nouv}) (for the first component of the r.h.s.) and
(\ref{pistoche}) (for the second component) we get 
$$
\babs{([\Lambda^s,\Gpmb]\psi_1,\Lambda^s \psi_2)}\leq M\Abs{\Lambda^{s-1}\nampm\psi_1^\mfh}_2\Abs{\Lambda^s\nampm\psi^\dagger_2}_2.
$$
Since $s-1\geq 0$, (\ref{estreg}) and 
the observation that 
$\Abs{\Lambda^s\nampm\psi^\dagger_2}_2\lesssim \sqrt{\mu}\abs{\Pp \psi_2}_{H^s}$
yield the result.

Let us now prove (\ref{DNcommut}) for $0\leq s\leq 1$. Let us write first that
$$
([\Lambda^s,\Gpmb]\psi_1,\Lambda^s \psi_2)=-(\Lambda^s[\Lambda,\Gpmb]\Lambda^{-1}\psi_1,\Lambda^s \psi_2)+([\Lambda^{s+1},\Gpmb]\Lambda^{-1}\psi_1,\Lambda^s\psi_2).
$$
Since $s+1\geq 1 $, we can use the computations above to show that the second term of the r.h.s. satisfies the desired estimate. We thus focus on the first term. Proceeding as above, we write
\begin{eqnarray*}
(\Lambda^s[\Lambda,\Gpmb]\Lambda^{-1}\psi_1,\Lambda^s \psi_2)=\int_{\cS^\pm}\Lambda^s[\Lambda,P(\Sigma^\pm)]\nampm(\Lambda^{-1}\psi_1)^\mfh\cdot \Lambda^s\nampm\psi_2^\dagger\\
+\int_{\cS^\pm}\Lambda^s\big(P(\Sigma^\pm)\nampm[\Lambda(\Lambda^{-1}\psi_1)^\mfh-\psi_1^\mfh]\big)\cdot\Lambda^s\nampm\psi_2^\dagger.
\end{eqnarray*}
Since $\Abs{\Lambda^s[\Lambda,P(\Sigma^\pm)]}_{H^{s,0}\to L^2(\cS^\pm)}+\Abs{\Lambda^s\big(P(\Sigma^\pm)\cdot\big)}_{H^{s,0}\to L^2(\cS^\pm)}\leq M$ (recall that $0\leq s\leq 1$),
we deduce (proceeding as for (\ref{pistoche}) for the second component) that
$$
\babs{(\Lambda^s[\Lambda,\Gpmb]\Lambda^{-1}\psi_1,\Lambda^s \psi_2)}\leq
M \Abs{\Lambda^{s}\nampm(\Lambda^{-1}\psi_1)^\mfh}_2
\Abs{\Lambda^s\nampm\psi^\dagger_2}_2,
$$
and the result follows as for the case $s\geq 1$.
\subsection{Proof of (\ref{estder0bis}), (\ref{estder0bisk}),  (\ref{eqsectII1-6gen}) and
 (\ref{eqsectII1-6genbis})}\label{apestpouet}

Throughout this proof, we assume that $\Sigma^\pm$ is a regularizing diffeomorphism
of the form $\Sigma^\pm(X,z)=(X,z+\sigma^\pm(X,z))$ with $\sigma^\pm$
as in Example \ref{exII1-2}. We denote by $d^j P^\pm (\bh)$ ($j\in\N$)
the $j$-th derivative of the
mapping $\zeta\mapsto P(\Sigma^\pm )$ in the direction $\bh=(h_1,\dots,h_j)$,
and
by $d^j\phi (\bh)$ the $j$-th derivative in the direction $\bh$ of
$\zeta\mapsto \phi $, where $\phi $ solves (\ref{eqsectII1-3}).

\subsubsection{Proof of (\ref{estder0bis})}\label{pestak}
With the same notations as in \S \ref{proofII1-6}, one has
$$
(\Lambda^{s-1/2}\Gpmb \psi,\varphi)=
	\int_{\cS^+} \Lambda^{s}(P(\Sigma^\pm)\nampm\phi^\pm)\cdot
	\Lambda^{-1/2}\nampm\varphi^\dagger;
$$
since $\Abs{\Lambda^{-1/2}\nampm\varphi^\dagger}_2\lesssim {\mu^{1/4}} \abs{\varphi}_2$,
it follows after differentiating this identity with respect to $\zeta$ and by a duality
argument that 
for all $0\leq s\leq t_0+1$, one has
$$
\abs{d^j\Gpmb(\bh) \psi}_{H^{s-1/2}}\lesssim {\mu}^{1/4}\sum \Abs{\Lambda^{s}d^{j_1}P^\pm (\bh_I)\nampm d^{j_2}\phi (\bh_{II})}_2,
$$
where the summation is taken over all integers $j_1$ and $j_2$ such that $j_1+j_2=j$ and on
all the $j_1$ and $j_2$ uplets $\bh_I$ and $\bh_{II}$ whose coordinates form a permutation of the
coordinates of $\bh$. 
The result follows therefore from the estimate, for all $0\leq s\leq t_0+1/2$
(and using Notation \ref{notah}),
\begin{equation}
\nonumber
\Abs{\Lambda^s d^{j_1}P^\pm(\bh_I)\nampm d^{j_2}\phi (\bh_{II})}_2\leq M \eps^j\sqrt{\mu}
\abs{h_l}_{H^{s+1/2}}
\label{keypest}
\abs{\check{\bf h}_l}_{H^{s\vee t_0+3/2}}\abs{\Pp \psi}_{H^{s\vee t_0+1/2}};
\end{equation}
the rest of this section is thus devoted to the proof of (\ref{keypest}). We first state three
lemmas and use them to prove the result, and then we prove the lemmas.
\begin{lemma}\label{lemap1}
For all $0\leq s\leq t_0+1$, $k\in\N^*$,  $\bh=(h_1,\dots,h_k)\in H^{s\vee t_0+1}(\R^d)^k$
and $1\leq l\leq k$, one has
$$
\Abs{\Lambda^s d^k P^\pm (\bh)}_{2}\leq M \eps^k \abs{h_l}_{H^{s+1/2}}\abs{\check{\bf h}_l}_{H^{s\vee t_0+1}}.
$$
\end{lemma}
\begin{lemma}\label{lemap2_1}
For all  $0\leq s\leq t_0+1/2$, $k\in \N^*$, $\bh=(h_1,\dots,h_k)\in H^{t_0+2}(\R^d)^k$ and $\psi \in\dot{H}^{s\vee t_0+1}(\R^d)$. Then $d^k\phi (\bh)$ satisfies
$$
\Abs{\nampm d^k\phi (\bh)}_{H^{s\vee t_0+1/2,1}}\leq M \eps^k \sqrt{\mu}
\abs{\bh}_{H^{s\vee t_0+3/2}}\abs{\Pp\psi }_{H^{s\vee t_0+1/2}}.
$$
\end{lemma}
\begin{lemma}\label{lemap2}
Under the assumptions of Lemma \ref{lemap2_1}, one also has 
$$
\Abs{\Lambda^s\nampm d^k\phi (\bh)}_{2}\leq M \eps^k\sqrt{\mu}\abs{h_l}_{H^{s+1/2}}
\abs{\check{\bh}_l}_{H^{s\vee t_0+3/2}}\abs{\Pp\psi }_{H^{s\vee t_0+1/2}}.
$$
\end{lemma}

Let us assume first that $h_l\in\bh_I$. Then, from (\ref{prodest}) and (\ref{contemb}), one deduces 
$$
\Abs{\Lambda^s d^{j_1}P^\pm(\bh_I)\nampm d^{j_2}\phi (\bh_{II})}_{2}\lesssim \Abs{\Lambda^s d^{j_1}P^\pm (\bh_I)}_{2}\Abs{\nampm d^{j_2}\phi (\bh_{II})}_{H^{s\vee  t_0+1/2,1}},
$$
and (\ref{keypest}) easily follows from the first two lemmas.\\
Now, if $h_l\in\bh_{II}$ then we rather write
$$
\Abs{\Lambda^s d^{j_1}P^\pm(\bh_I)\nampm d^{j_2}\phi (\bh_{II})}_{2}\lesssim \Abs{d^{j_1}P^\pm (\bh)}_{L^\infty H^{s\vee t_0}}\Abs{\Lambda^s \nampm d^{j_2}\phi (\bh_{II})}_{2},
$$
and Lemma \ref{lemap2} gives (\ref{keypest})
(the bound on $\Abs{ d^k P^\pm (\bh)}_{L^\infty H^{s\vee t_0}}$
being a straightforward consequence of (\ref{expP}) below).
\begin{proof}[Proof of Lemma \ref{lemap1}] Computing explicitly the matrix  $P(\Sigma^\pm)$ given by (\ref{PS}), one obtains
\begin{equation}\label{expP}
P(\Sigma^\pm) =
\left(\begin{array}{cc}
    \dsp \mbox{Id}_{d\times d}+\dz\sigma^\pm  & \dsp -\sqrt{\mu}\nabla\sigma^\pm \\
    \dsp -\sqrt{\mu}(\nabla\sigma^\pm)^T &\dsp  \frac{1+\mu\abs{\nabla\sigma^\pm}^2}{1+\dz\sigma^\pm };
\end{array}\right).
\end{equation}
Since $\zeta\mapsto \sigma^\pm $ is linear, it is easy to deduce the result from 
(\ref{prodest}) and the regularizing properties of $\sigma^\pm$ as 
given in Example \ref{exII1-2}.
\end{proof}
\begin{proof}[Proof of Lemma \ref{lemap2_1}]
By differentiating  (\ref{eqsectII1-3}) $k$ times with respect to $\zeta$ in the 
direction $\bh$, one gets
\begin{equation}\label{eqkb}
\left\lbrace
\begin{array}{l}
\nampm\cdot P(\Sigma^\pm) \nampm d^k\phi^\pm (\bh)=\nampm\cdot{\bf g},\\
d^k\phi^\pm (\bh)_\interff=0,\qquad \partial_n d^k\phi^\pm (\bh)_{\vert_{z=\mp 1}}=-{\bf e}_z\cdot {\bf g}_{\vert_{z=\mp 1}},
\end{array}\right.
\end{equation}
where ${\bf g}$ is a sum of terms of the form
\begin{equation}\label{formsec}
d^{k_1}P^\pm (\bh_I)\nampm d^{k_2}\phi^\pm (\bh_{II})=:A(\bh_I,\bh_{II}),
\end{equation}
where $k_1,k_2\in\N$, $k_2<k$, $k_1+k_2=k$, and $\bh_I$ and $\bh_{II}$ are respectively
$k_1$ and $k_2$-uplets whose coordinates form a permutation of the coordinates of $\bh$.
Since by Lemma \ref{lemkk} we have
$$
\Abs{\nampm d^k\phi^\pm (\bh)}_{H^{s\vee t_0+1/2,1}}\leq M \Abs{ {\bf g}}_{H^{s\vee t_0+1/2,1}},
$$
we can deduce the result from the fact that
\begin{equation}\label{intermap}
\Abs{ A(\bh_I,\bh_{II})}_{H^{s\vee t_0+1/2,1}}
\leq M \eps^k\sqrt{\mu} \abs{\bh}_{H^{s\vee t_0+3/2}}
\abs{\Pp \psi }_{H^{t_0+1}},
\end{equation}
which we now turn to prove. From (\ref{formsec}) and (\ref{prodeststrip}), one gets
\begin{eqnarray*}
\Abs{ A(\bh_I,\bh_{II})}_{H^{s\vee t_0+1/2,1}}
&\lesssim&\Abs{\nampm d^{k_2}\phi (\bh_{II})}_{H^{s\vee t_0+1/2,1}}\\
&\times& (\Abs{d^{k_1}P^\pm (\bh_I)}_{L^\infty_z H^{s\vee t_0+1/2}}+\Abs{\dz d^{k_1}P^\pm (\bh_I)}_{L^\infty_z H^{t_0}})
\end{eqnarray*}
and therefore, using (\ref{expP}) and the fact that $\zeta\mapsto \sigma^\pm [\zeta]$
s linear,
$$
\Abs{A(\bh_I,\bh_{II})}_{H^{s\vee t_0+1/2,1}}
\leq M \eps^{k_1}\abs{\bh_I}_{H^{s\vee t_0+3/2}}\Abs{\nampm d^{k_2}\phi (\bh_{II})}_{H^{s\vee t_0+1/2,1}}.
$$
Since $k_2<k$, one can deduce (\ref{intermap}) by a simple induction on $k$
(the case $k=0$ being a consequence of (\ref{estreg})).
\end{proof}
\begin{proof}[Proof of Lemma \ref{lemap2}]
Proceeding as in the proof of Lemma \ref{lemap2_1}, we can check that the result follows
from 
\begin{equation}\label{intermap_1}
\Abs{\Lambda^s A(\bh_I,\bh_{II})}_{2}
\leq M \eps^k\sqrt{\mu}\abs{h_l}_{H^{s+1/2}}\abs{\check{\bh}_l}_{H^{s\vee t_0+3/2}}\abs{\Pp \psi}_{H^{s\vee t_0+1/2}},
\end{equation}
that we now turn to prove.
We have to
distinguish two cases:\\
- If $h_l$ belongs to $\bh_I$ we use (\ref{prodeststrip}) and (\ref{contemb}) to write
\begin{eqnarray*}
\Abs{\Lambda^s{A(\bh_I,\bh_{II})}}_{2}&\leq&
\Abs{\Lambda^s d^{k_1}P^\pm (\bh_I)}_{2}\Abs{\nampm d^{k_2}\phi (\bh_{II})}_{L^\infty_zH^{s\vee t_0}}\\
&\leq& \Abs{\Lambda^s d^{k_1}P^\pm (\bh_I)}_{2}\Abs{\nampm d^{k_2}\phi (\bh_{II})}_{H^{s\vee t_0+1/2,1}},
\end{eqnarray*}
and (\ref{intermap_1}) follows directly from Lemmas \ref{lemap1} and \ref{lemap2_1}.\\
- If $h_l$ belongs to $\bh_{II}$, we rather write
\begin{eqnarray*}
\Abs{\Lambda^s {A(\bh_I,\bh_{II})}}_{2}
&\lesssim& \Abs{d^{k_1}P^\pm (\bh_I)}_{L^\infty_z H^{s\vee t_0}}
\Abs{\Lambda^s \nampm d^{k_2}\phi (\bh_{II})}_{2},
\end{eqnarray*}
and (\ref{intermap_1}) follows easily by induction as in the proof of Lemma \ref{lemap2_1}.
\end{proof}

\subsubsection{Proof of (\ref{estder0bisk})}

With the same notations as in \S \ref{proofII1-6}, one has
$$
(\Lambda^{s-1/2}\Gpmb \psi,\varphi)=
	\int_{\cS^+} \Lambda^{s+1/2}(P(\Sigma^\pm)\nam\phi^\pm)\cdot
	\Lambda^{-1}\nam\varphi^\dagger;
$$
since $\Abs{\Lambda^{-1}\nam\varphi^\dagger}_2\lesssim \sqrt{\mu} \abs{\varphi}_2$,
it follows after differentiating this identity with respect to $\zeta$ and by a duality
argument that 
for all $0\leq s\leq t_0+1$, one has
$$
\abs{d^j\Gpmb(\bh) \psi}_{H^{s-1/2}}\lesssim \sqrt{\mu}\sum \Abs{\Lambda^{s+1/2}d^{j_1}P^\pm (\bh_I)\nampm d^{j_2}\phi (\bh_{II})}_2,
$$
where the summation is as in \S \ref{pestak}. The result is therefore a direct consequence of (\ref{keypest}) (with a shift of $1/2$ derivative).

\subsubsection{Proof of (\ref{eqsectII1-6gen})}

We proceed as in \S \ref{proofII1-6} after differentiating (\ref{base}) with respect to $\zeta$.
One thus gets
$$
\babs{(\Lambda^s d^j\Gpmb(\bh)\psi_1,\Lambda^s\psi_2)}\leq \sqrt{\mu} 
\sum \Abs{\Lambda^s d^{j_1}P^\pm(\bh_I)\nampm d^{j_2}\phi^\pm(\bh_{II})}_2
\abs{\Pp\psi_2}_{H^s},
$$
where the summation is the same as in the previous section. Elliptic estimates on
(\ref{eqkb}) and an induction on $j$ show that
$$
\Abs{\Lambda^s d^{j_1}P^\pm(\bh_I)\nampm d^{j_2}\phi^\pm(\bh_{II})}_2\leq M \eps^j\sqrt{\mu}
\abs{\bh}_{H^{s\vee t_0+1}} \,\abs{\Pp \psi_1}_{H^s},
$$
and the result follows.

\subsubsection{Proof of (\ref{eqsectII1-6genbis})}

We proceed as for (\ref{eqsectII1-6gen}) but use (\ref{keypest}) to
control $\Abs{\Lambda^s d^{j_1}P^\pm(\bh_I)\nampm d^{j_2}\phi^\pm(\bh_{II})}_2$.

\section{Proof of  (\ref{cub1}) and (\ref{cub2})}\label{app:cuba}

\subsection{Proof of (\ref{cub1})}

We prove (\ref{cub1}) in the case $\iota=0$ (that is, $J^\iota=1$); the generalization
to the case $\iota>0$ only requires simple technical modifications and is thus omitted.\\
By definition of $\RT$, one has
\begin{eqnarray*}
\lefteqn{\big(\RT \Tb\zetaa,\zetaa\big)=-\eps^2\mu\urp\urm\big(\jump{\uVpm}\cdot\Eb (\jump{\uVpm}\Tb \zetaa),\zetaa\big)}\\
& &+\big(\mfa  \Tb\zetaa,\zetaa\big)
-\frac{1}{\Bo}\big(\nabla\cdot \K\nabla\Tb \zetaa,\zetaa\big).
\end{eqnarray*}
The fact that the second and third term of this identity are controlled by $\mfm^N(U)$
is a direct consequence of the last two points of Proposition \ref{propT}. We are thus led to control the first term.
Let us remark first that, by definition of $\T$,
\begin{eqnarray*}
\lefteqn{\big(\jump{\uVpm}\cdot\Eb (\jump{\uVpm}\Tb \zetaa),\zetaa\big)
=\big(\jump{\uVpm}\cdot\Eb (\jump{\uVpm}\nabla\cdot(\zetaa\uVp)),\zetaa\big)}\\
&+&\urm\uH^-\big(\jump{\uVpm}\cdot\Eb 
(\jump{\uVpm}\Gb(\Gmb)^{-1}\nabla\cdot(\zetaa\jump{\uVpm})),\zetaa\big)\\
&=&A_1+A_2.
\end{eqnarray*}
Owing to (\ref{cubaonemore}), the inequality (\ref{cub1}) follows from the estimate
\begin{eqnarray}\nonumber
\lefteqn{\eps^2\mu\urp\urm\big(\abs{A_1}+\abs{A_2}\big)
\leq \mfm^N(U)}\\\label{cuba}
&\!\!\!\!\!\times&\!\!\!\!\!\!\Big[1+\eps^2\urp\urm\big(\babs{(1+\mu^{1/4}\abs{D}^{1/2})F}_2^2+\babs{(1+\mu^{1/4}\abs{D}^{1/2})G}_2^2\big)\Big],
\end{eqnarray}
with $F=\zetaa\jump{\uVpm}$ and $G=\zetaa\partial_j\nabla\jump{\uVpm}$, that we prove now.\\
- \emph{Control of  $A_1$}. We decompose further
$A_1$ into $A_1=A_{11}+A_{12}+A_{13}+A_{14}$ with (summing over the repeated index $j$,
denoting $\widetilde{F}=(1+\mu^{1/4}\abs{D}^{1/2})F$ and with $\widetilde{\Gb}$ as in (\ref{defGtilde})),
\begin{eqnarray*}
A_{11}&=&\big((1+\sqrt{\mu}\abs{D})^{1/2}(\zetaa\uVp_j\partial_j\jump{\uVpm}),\frac{\nabla}{\abs{D}}\Pp\widetilde{\Gb}^{-1}\nabla\cdot F\big),\\
A_{12}&=&\mu^{1/4}\big(\partial_j[\abs{D}^{1/2},\uVp_j]F, \Pp'\widetilde{\Gb}^{-1}\nabla\cdot F\big),\\
A_{13}&=&-\big(\uVp_j\widetilde{F},\frac{(1+\sqrt{\mu}\abs{D})}{(1+\mu^{1/4}\abs{D}^{1/2})}\frac{\nabla\partial_j}{\abs{D}^2}\Op(\tau)\nabla\cdot\widetilde{F}\big),\\
A_{14}&=&-\big(\uVp_j\widetilde{F},\frac{(1+\sqrt{\mu}\abs{D})}{(1+\mu^{1/4}\abs{D}^{1/2})}\frac{\nabla\partial_j}{\abs{D}^2}[\Pp^2\widetilde{\Gb}^{-1}\frac{1}{(1+\mu^{1/4}\abs{D}^{1/2})}-\Op(\tau)]\nabla\cdot\widetilde{F}\big),
\end{eqnarray*}
where $\Pp'=\frac{\abs{D}}{1+\mu^{1/4}\abs{D}}$ and the symbol $\tau$ is given by (with $S^\pm$ as in (\ref{defSpm}))
$$
\tau(X,\xi)=\frac{\abs{\xi}^2}{(1+\sqrt{\mu}\abs{\xi})}\frac{\av{\uH^\pm}}{(\urp\uH^+S^--\urm\uH^-S^+)}\frac{1}{(1+\mu^{1/4}\abs{\xi}^{1/2})}.
$$
Using Cauchy-Schwarz inequality, standard commutator estimates and Remark \ref{remGder}, 
we get that $A_{11}$ and $A_{12}$ are controlled as in (\ref{cuba}).\\
For $A_{13}$, we remark that ${\bf O}(U)=\uVp_j\frac{(1+\sqrt{\mu}\abs{D})}{(1+\abs{\mu}^{1/4}\abs{D}^{1/2})}\frac{\nabla\partial_j}{\abs{D}}\Op(\tau)\nabla^T$ is a first order operator with skew symmetric principal symbol and that, as a consequence of the pseudodifferential estimates of \cite{LannesJFA}, one has $\mu\Abs{{\bf O}(U)+{\bf O}(U)^*}_{L^2\to L^2}\leq \mfm^N(U)$. It follows that
$A_{13}$ is also bounded by the r.h.s. of (\ref{cuba}).\\
For $A_{14}$, one deduces from Cauchy-Schwarz inequality that
\begin{eqnarray*}
\abs{A_{14}}&\leq& \abs{\underline{V}^+_j\widetilde{F}}_2
\big(\babs{[\Pp^2\widetilde{\Gb}^{-1}\frac{1}{(1+\mu^{1/4}\abs{D}^{1/2})}-\Op(\tau)]\nabla\cdot\widetilde{F}}_2\\
&+&\mu^{1/4}\babs{[\Pp^2\widetilde{\Gb}^{-1}\frac{1}{(1+\mu^{1/4}\abs{D}^{1/2})}-\Op(\tau)]\nabla\cdot\widetilde{F}}_{H^{1/2}}\big)
\end{eqnarray*}
and Corollary \ref{coro5} (with $k=0$ for the first term and $k=1$ for the second one)
implies that
$
\abs{A_{14}}\leq \frac{1}{\mu}\mfm^N(U) \abs{\widetilde{F}}_2^2$, from which
(\ref{cuba}) follows for for $A_{14}$ and therefore for $A_1$.\\
- \emph{Control of $A_2$}. Let us first rewrite $A_2$ as (with $F=\zetaa\jump{\uVpm}$)
$$
A_2=\urm\frac{1}{{\uH^+}}\big(\Gpb\widetilde{\Gb}^{-1}\nabla\cdot F),\jump{\uVpm}\cdot\nabla \widetilde{\Gb}^{-1}\nabla\cdot F \big).
$$
Using (\ref{idcuba}) 
with $u=\widetilde{\Gb}^{-1}\nabla\cdot F$ and ${\bf v}=\jump{\uVpm}$,
and using Remark \ref{remGder} yields,
with $\widetilde{F}$ as above,
$$
\abs{A_2}\leq \frac{1}{\mu}\mfm^N(U)\abs{\widetilde{F}}_2^2,
$$
which implies (\ref{cuba}). 

\subsection{Proof of (\ref{cub2})}

As for (\ref{cub1}), we just consider the case $\iota=0$ (i.e. $J^\iota=\Id$).
We first write
\begin{eqnarray*}
\big(\frac{1}{\mu}\Gb\Tb^*\psia,\psia\big)&=&-
\big(\frac{1}{\mu}\uVp\cdot\nabla\psia,\Gb\psia\big)\\
& &-\frac{1}{\mu}
\big( \jump{\uVpm}\cdot \nabla (\Gmb)^{-1}\Gb\psia,\Gb\psia\big)\\
&=&B_1+B_2.
\end{eqnarray*}
For $B_1$, we use the fact that $\Gb=\frac{1}{\uH^+}\Gpb\circ J[\zeta]^{-1}$ (see Remark \ref{remequiv}) to get
$$
B_1=-\frac{1}{\mu}\big(\uVp\cdot \nabla J[\zeta] g,\Gpb g\big),
$$
with $g=J[\zeta]^{-1}\psia$. Since $J[\zeta]=\urp \Id-\urm\frac{H^-}{H^+}(\Gmb)^{-1}\Gpb$, we deduce that
$$
B_1=-\frac{\urp}{\mu}\big(\uVp\cdot\nabla g,\Gpb g\big)
-\frac{\urm}{\mu}\frac{H^-}{H^+}\big( \uVp\cdot\nabla \widetilde{g},\Gmb \widetilde{g}\big),
$$
with $\widetilde{g}=(\Gmb)^{-1}\Gpb g$. It follows therefore from (\ref{idcuba}), Lemma \ref{lemma1}
and Proposition \ref{prop1} that
$
\abs{B_1}\leq \mfm^N(U)$.
For $B_2$, we just write
$$
B_2=\frac{1}{\mu}
\big( \jump{\uVpm}\cdot \nabla g^\sharp,\Gmb g^\sharp\big)
$$
with $g^\sharp=(\Gmb)^{-1}\Gb\psia$. Using (\ref{idcuba}) again, we also get $\abs{B_2}\leq \mfm^N(U)$, and the proof of (\ref{cub2}) is complete.

\bigbreak

\noindent
{\bf Acknowledgments.} The author wants to thank C. Bardos, J. Rauch and J.-C. Saut for many discussions, T. Alazard and T. Nguyen for their helpful remarks and F. Dias and J. Grue for their help on the experimental part.
 Ce travail a b\'en\'efici\'e d'une aide de l'Agence Nationale de la Recherche portant la r\'ef\'erence ANR-08-BLAN-0301-01.


\begin{thebibliography}{8.}
\bibitem{ABZ} \textsc{T. Alazard, N. Burq, C. Zuily}, {\it On the Cauchy problem for the water waves with surface tension}, submitted.
\bibitem{AlazardMetivier} \textsc{T. Alazard, G. M\'etivier}, {\it Paralinearization of the Dirichlet to Neumann operator, and regularity of three-dimensional water waves}, Comm. Partial Differential Equations {\bf 34} (2009), 1632-1704.
\bibitem{AL} \textsc{B. Alvarez-Samaniego, D. Lannes}, {\it Large time existence for 3d water-waves and asymptotics}, Invent. math. {\bf 171} (2008) 485–541.
\bibitem{Am} \textsc{D.Ambrose}, {\it Well-posedness of vortex sheets with surface tension}, SIAM J. Math. Anal.{\bf 35} (2003), 211-244.
\bibitem{AM2} \textsc{D. M. Ambrose, N. Masmoudi}, {\it The zero surface tension limit of two-dimensional water waves}, Comm. Pure Appl. Math. {\bf 58} (2005) 1287-1315.
\bibitem{AM3} \textsc{D. M. Ambrose, N. Masmoudi},  {\it The zero surface tension limit of three-dimensional water waves} 
Indiana Univ. Math. J. {\bf 58} (2009), 479-521.
\bibitem{Ammas} \textsc{D. Ambrose, N. Masmoudi}, {\it Well-posedness of 3D vortex sheets with surface tension}, Commun. Math. Sci. {\bf 5} (2007), 391-430.
\bibitem{BardosLannes} \textsc{C. Bardos, D. Lannes}, {\it Mathematics for $2D$ interfaces}, to appear in Panaroma et Synth\`eses.
\bibitem{BB} \textsc{T.B. Benjamin, T.J. Bridges}, {\it Reappraisal of the Kelvin–Helmholtz problem. Part 1. Hamiltonian structure}, J. Fluid Mech. {\bf 333} (1997) 301–325. 
\bibitem{BeyerGunther} \textsc{K. Beyer, M. G\"unther}, {\it On the Cauchy Problem for a Capillary Drop. Part I: 
Irrotational Motion}, Math. Meth. Appl. Sci. {\bf 21} (1998),  1149-1183. 
\bibitem{BLS} \textsc{J. L. Bona, D. Lannes, J.-C. Saut}, {\it Asymptotic models for internal waves},  J. Math. Pures Appl. {\bf 89} (2008), 538-566.
\bibitem{Chandrasekhar} \textsc{S. Chandrasekhar}, {\it Hydrodynamic and hydromagnetic stability},
The International Series of Monographs on Physics Clarendon Press, Oxford 1961 xix+654 pp.
\bibitem{CCS1} \textsc{A. Cheng, D. Coutand, S. Shkoller}, {\it On the motion of Vortex Sheets with surface tension},  Comm. on Pure and Appl. Math. {\bf 61} (2008), 1715-1752. 
\bibitem{CCS2} \textsc{A. Cheng, D. Coutand, S. Shkoller},  {\it On the limit as the density ratio tends to zero for two perfect incompressible
3-D fluids separated by a surface of discontinuity}, submitted.
\bibitem{CC2} \textsc{W. Choi, R. Camassa}, {\it Fully nonlinear internal waves in a two-ﬂuid system}, J. Fluid Mech. {\bf 396} (1999) 1-36. 
\bibitem{CHS} \textsc{H. Christianson, V. M.  Hur, G. Staffilani}, {\it Strichartz estimates for the water-wave problem with surface tension}, submitted.
\bibitem{CCG1} \textsc{A. Cordoba, D. Cordoba, F. Gancedo}, {\it Interface evolution: water waves in $2-D$}, Advanaces in Math. {\bf 223} (2010), 120-173.
\bibitem{CoutandShkoller} \textsc{D. Coutand, S. Shkoller}, {\it Well-posedness of the free-surface incompressible Euler
equations with or without surface tension},  J. Am. Math. Soc. {\bf 20} (2007), 829-930.
\bibitem{Craig} \textsc{W. Craig}, {\it An existence theory for water waves and the Boussinesq and Korteweg-de Vries
scaling limits}, Comm. Partial Differential Equations {\bf 10} (1985), 787-1003.
\bibitem{CGK} \textsc{W. Craig, P. Guyenne, and H. Kalisch}, {\it Hamiltonian long-wave expansions for free surfaces and interfaces},
Comm.  Pure Appl. Math. {\bf 58} (2005), 1587-1641.
\bibitem{Craig-Sulem} \textsc{W. Craig, C. Sulem}, {\it Numerical simulation of gravity waves},  J. Comput. Phys {\bf 108} (1993), 73-83.
\bibitem{DenyLions} \textsc{J. Deny, J. L. Lions}, {\it Les espaces de Beppo Levi}, Ann. Inst. Fourier Grenoble {\bf 5} (1953-54), 304-370.
\bibitem{Ebin} \textsc{G. Ebin},  {\it Ill-posedness of the Raileigh-Taylor and Kelvin-Helmotz problems for incompressible fluids}, Comm. Partial Differential
Equations {\bf 13} (1988), 1265-1295.
\bibitem{EM} \textsc{L. Escauriaza, M. Mitrea}, {\it Transmission problems and spectral theory for singular integral operators 
on Lipschitz domains},  J. Funct. Anal. {\bf 216} (2004), 141-171.
\bibitem{GMS} \textsc{P. Germain, N. Masmoudi, J. Shatah}, {\it Global Solutions for the Gravity Water Waves Equation in Dimension 3}, submitted, 	arXiv:0906.5343v1.
\bibitem{Grenier} \textsc{E. Grenier}, {\it Pseudo-diﬀerential energy estimates of singular perturbations}, Comm. Pure Appl. Math. {\bf 50} (1997) 821–865.
\bibitem{Grue} \textsc{J. Grue, A. Jensen, P.-O. Rusas, J. K. Sveen}, {\it Properties of large-amplitude internal waves}, J. Fluid 
Mech. {\bf 380} (1999), 257-278.
\bibitem{GLS} \textsc{P. Guyenne, D. Lannes, J.-C. Saut}, {\it Well-posedness of the Cauchy problem for models of
large amplitude internal waves}, Nonlinearity {\bf 23} (2010), 237-275.
\bibitem{HelfrichMelville} \textsc{K. R. Helfrich, W. K. Melville}, {\it Long nonlinear internal waves}, In Annual review of ﬂuid mechanics {\bf 38} (2006), 395-425. 
\bibitem{Iguchi00} \textsc{T. Iguchi}, {\it A long wave approximation for capillary-gravity waves and an effect of the bottom},  
Comm. Partial Differential Equations {\bf 32} (2007), 37-85. 
\bibitem{Iguchi0} \textsc{T. Iguchi}, {\it A long wave approximation for capillary-gravity waves and the Kawahara equation},
Bull. Inst. Math. Acad. Sin. (N.S.) {\bf 2} (2007), 179-220. 
\bibitem{Iguchi} \textsc{T. Iguchi}, {\it A shallow water approximation for water waves}, J. Math. Kyoto Univ. {\bf 49} (2009), 13-55.
\bibitem{ITT} \textsc{T. Iguchi, N. Tanaka, A. Tani}, {\it On the two-phase free boundary problem for two-dimensional water 
waves}, Math. Ann. {\bf 309} (1997), 199-223.
\bibitem{KL} \textsc{V. Kamotski, G. Lebeau}, {\it On 2D Rayleigh-Taylor
  instabilities},  Asymptot. Anal. {\bf 42} (2005),  1-27.
\bibitem{KP} \textsc{C. G. Koop, G. Butler}, {\it An investigation of internal solitary waves in a two-ﬂuid system}, J. Fluid 
Mech. {\bf 112} (1981) 225-251.
\bibitem{LannesJAMS} \textsc{D. Lannes}, {\it Well-posedness of the water-waves equations}, J. Amer. Math. Soc. {\bf 18} (2005), 605-654.
\bibitem{LannesJFA} \textsc{D. Lannes}, {\it Sharp estimates for pseudo-diﬀerential operators with symbols of limited smoothness and 
commutators},  J. Funct. Anal. {\bf 232} (2006), 495-539.
\bibitem{LE} \textsc{G. Lebeau}, {\it R\'egularit\'e du probl\`eme de
Kelvin--Helmholtz pour l'\'equation d'Euler  $2d$},  ESAIM:
COCV {\bf 08} (2002), 801-825.
\bibitem{Lindblad} \textsc{H. Lindblad}, {\it Well-posedness for the motion of an incompressible liquid with free surface boundary}, Ann. 
of Math. (2) {\bf 162} (2005), 109-194.
\bibitem{MingZhang} \textsc{M. Ming, Z. Zhang}, {\it Well-posedness of the water-wave problemwith surface tension}, J. Math. Pures Appl. {\bf 92} (2009), 429-455.
\bibitem{Nalimov} \textsc{V. I. Nalimov}, {\it The Cauchy-Poisson problem} (Russian),
 Dinamika Splosn. Sredy Vyp. {\bf 18} 
Dinamika Zidkost. so Svobod. Granicami {\bf 254} (1974), 104-210.
\bibitem{ND} \textsc{H.Y. Nguyen, F. Dias}, {\it A Boussinesq system for two-way propagation of interfacial waves}, Physica D (2008), in press. 
\bibitem{OhiIguchi} \textsc{K. Ohi, T. Iguchi}, {\it A two-phase problem for capillary-gravity waves and the Benjamin-Ono equation}, Discrete 
Contin. Dyn. Syst. {\bf 23} (2009), 1205-1240.
\bibitem{Ovsjannikov} \textsc{L. V. Ovsjannikov}, {\it Cauchy problem in a scale of Banach spaces and its application to the
shallow water theory justification}, Applications of methods of functional analysis to problems
in mechanics (Joint Sympos., IUTAM/IMU, Marseille, 1975), pp. 426–437. Lecture Notes in
Math., 503. Springer, Berlin, 1976.
\bibitem{Pusateri1} \textsc{F. Pusateri}, {\it On the one fluid limit for vortex sheets}, submitted,  arXiv:0908.3353v1.
\bibitem{Pusateri2} \textsc{F. Pusateri}, {\it On the limit as the surface tension and density ratio tend to zero for the two–phase Euler equations}, submitted, arXiv:0912.3296v1.
\bibitem{ShatahZeng} \textsc{J. Shatah, C. Zeng}, {\it Geometry and a priori estimates for free boundary problems of the Euler equation},
Comm. Pure Appl. Math. {\bf 61} (2008), 698-744.
\bibitem{ShatahZeng2} \textsc{J. Shatah, C. Zeng}, {\it A priori estimates for fluid interface problems}, Comm. Pure Appl. Math. {\bf 61} (2008), 848-876.
\bibitem{SW} \textsc{G. Schneider, C. E. Wayne}, {\it The rigorous approximation of long-wavelength capillary-gravity waves},
Arch. Ration. Mech. Anal. {\bf 162} (2002), 247-285. 
\bibitem{Taylor2} \textsc{M. E. Taylor}, {\it Partial Diﬀerential Equations II: Qualitative Studies of Linear Equations}, volume 116 of 
Applied Mathematical Sciences. Springer.
\bibitem{Taylor} \textsc{M. E. Taylor}, {\it Partial diﬀerential equations. III}, volume 117 of Applied Mathematical Sciences. Springer-Verlag, New York, 1997. Nonlinear equations, Corrected reprint of the 1996 original.
\bibitem{RoussetTzvetkov} \textsc{F. Rousset, N. Tzvetkov}, {\it Transverse instability of the line solitary water-waves}, submitted.
\bibitem{RoussetTzvetkov2}  \textsc{F. Rousset, N. Tzvetkov}, {\it On the transverse instability of one dimensional capillary-gravity waves}, DCDS-B {\bf 13} (2010), 859-872.
\bibitem{SS}
  \textsc{C. Sulem,  P.-L. Sulem}, {\it Finite time analyticity
  for the two- and three-dimensional Rayleigh-Taylor
  instability},   Trans. Amer.
Math. Soc. {\bf 287} (1981), 127-160.
\bibitem{SSBF} \textsc{C. Sulem, P.-L. Sulem, C. Bardos, U. Frisch},
  {\it Finite time analyticity for the two- and three-dimensional
  Kelvin-Helmholtz instability},   Comm. Math. Phys.
  {\bf 80} (1981), 485-516.
 \bibitem{WU} \textsc{S. Wu}, {\it Well-Posedeness in Sobolev
 spaces of the full water wave
problem in $2-D$},  Invent. Math. {\bf 130} (1997), 439-72.
\bibitem{WUb} \textsc{S. Wu}, {\it Well-posedness in Sobolev spaces of the full water wave problem in 3-D},
J. Amer. Math. Soc. {\bf 12} (1999), 445-495. 
\bibitem{Wu2}  \textsc{S. Wu}, {\it Mathematical analysis of vortex
sheets},  Comm. Pure Appl. Math. {\bf 59}  (2006),
1065-1206.
\bibitem{Wu3} \textsc{S. Wu}, {\it Almost Global Wellposedness of the $2-D$ full Water Wave Problem}, Invent. Math.  {\bf 177}  (2009), 45-135. 
\bibitem{Wu4} \textsc{S. Wu}, {\it Global well-posedness of the 3-D full water wave problem}, submitted, arXiv:0910.2473v1.
\bibitem{Yosihara}  \textsc{H. Yosihara}, {\it Gravity waves on the free surface of an incompressible perfect fluid of finite
depth},  Publ. Res. Inst. Math. Sci. {\bf 18} (1982), 49-96.
\bibitem{Yosihara2} \textsc{H. Yosihara}, {\it Capillary-gravity waves for an incompressible ideal fluid}, J. Math. Kyoto Univ. {\bf 23} (1983), 649-694. 
\bibitem{Zakharov} \textsc{V. E. Zakharov}, {\it Stability of periodic waves of ﬁnite amplitude on the surface of a deep ﬂuid},  Journal of Applied Mechanics and Technical Physics {\bf 9} (1968) 190-194.
\end{thebibliography}
\end{document}